%% file: paper.tex
\documentclass[onecollarge]{svjour2}       

\smartqed  
\usepackage{graphicx}
\usepackage{color}
\usepackage{amsmath}
\usepackage{amssymb}
\usepackage{esint}
\usepackage{mathtools}
\usepackage{float}
\usepackage{amsmath,amsfonts,amssymb,latexsym,epsfig}
\usepackage[most]{tcolorbox}
\usepackage[colorinlistoftodos,prependcaption,textsize=tiny]{todonotes}
\usepackage{todonotes}
\usepackage{subcaption}
\usepackage{enumitem}
\usepackage{ntheorem}
\include{references}
\usepackage{blkarray}

\newcommand{\cL}{\mathcal L}

\DeclareMathOperator{\Tr}{Tr}
\DeclareMathOperator{\Hess}{Hess}
\DeclareMathOperator{\diag}{diag}
\DeclareMathOperator{\Span}{span}
\DeclareMathOperator{\im}{im}

\DeclareMathOperator{\R}{\mathbb{R}}

\DeclareMathOperator{\gen}{\mathcal{L}}

\newtheorem{assumption}{Assumption}

\newtheorem{algorithm}{Algorithm}

\newcommand{\inner}[2]{\left\langle #1, #2 \right\rangle}
\newcommand{\norm}[1]{\left|{#1}\right|} 
\newcommand{\Norm}[1]{\left\lVert{#1}\right\rVert} 

\newcommand{\notate}[1]{\textcolor{blue}{\textbf{[#1]}}}

\begin{document}

\title{Using Perturbed Underdamped Langevin Dynamics to Efficiently Sample
from Probability Distributions}

\titlerunning{Perturbed Underdamped Langevin Dynamics}        

\author{A. B. Duncan         \and
        N. N{\"u}sken \and  
        G. A. Pavliotis
}

\authorrunning{Duncan, N{\"u}sken, Pavliotis} 

\begingroup
\institute{A. B. Duncan \at
              School of Mathematical and Physical Sciences, University of Sussex, Falmer, Brighton, BN1 9RH United Kingdom
              \email{Andrew.Duncan@sussex.ac.uk}           
           \and
           N. N{\"u}sken \at
 Imperial College London, Department of Mathematics, South Kensington Campus, London SW7 2AZ,England \email{n.nusken14@imperial.ac.uk}
            \and
           G. A. Pavliotis \at
             Imperial College London, Department of Mathematics, South Kensington Campus, London SW7 2AZ,England \email{g.pavliotis@imperial.ac.uk}}

\date{Received: date / Accepted: date}

\maketitle
\begin{abstract}
In this paper we introduce and analyse Langevin samplers that consist of perturbations of the standard underdamped Langevin dynamics. The perturbed dynamics is such that its invariant measure is the same as that of the unperturbed dynamics. We show that appropriate choices of the perturbations can lead to samplers that have improved properties, at least in terms of reducing the asymptotic variance. We present a detailed analysis of the new Langevin sampler for Gaussian target distributions. Our theoretical results are supported by numerical experiments with non-Gaussian target measures.
\end{abstract}

\section{Introduction and Motivation}
\label{sec:introduction}
\input{introduction.tex}

\section{Construction of General Langevin Samplers}
\label{sec:background}
\input{background.tex}

\section{Perturbation of Underdamped Langevin Dynamics}
\label{sec:perturbed_langevin}
\input{perturbed_langevin.tex}

\section{Sampling from a Gaussian Distribution}
\label{sec:Gaussian}
\input{gaussian.tex}


\section{Numerical Experiments: Diffusion Bridge Sampling}
\label{sec:numerics}
\input{numerics.tex}

\section{Outlook and Future Work}
\label{sec:outlook}
\input{outlook.tex}

\section*{Acknowledgments}
 AD was supported by the EPSRC under grant No. EP/J009636/1. NN is supported by EPSRC through a Roth Departmental Scholarship. GP is partially supported by the EPSRC under grants No. EP/J009636/1, EP/L024926/1, EP/L020564/1 and EP/L025159/1. Part of the work reported in this paper was done while NN and GP were visiting the Institut Henri Poincar\'{e} during the Trimester Program "Stochastic Dynamics Out of Equilibrium". The hospitality of the Institute and of the organizers of the program is greatly acknowledged. 

\appendix
\section{Estimates for the Bias and Variance}
\label{app:proofs}
\input{appproofs.tex}

\section{Proofs of Section \ref{sec:perturbed_langevin}}
\label{app:hypocoercivity}
\input{apphypocoercivity.tex}

\section{Asymptotic Variance of Linear and Quadratic Observables in the Gaussian Case}
\label{app:Gaussian_proofs}
\input{app_Gaussian}

\section{Orthogonal Transformation of Tracefree Symmetric Matrices into a Matrix with Zeros on the Diagonal}
\label{tracefree}
\input{tracefree.tex}

\bibliographystyle{alpha}
\input{paper.bbl}

\end{document}

%% file: introduction.tex
Sampling from probability measures in high-dimensional spaces is a problem that appears frequently in applications, e.g. in computational statistical mechanics and in Bayesian statistics. In particular, we are faced with the problem of computing expectations with respect to a probability measure $\pi$ on $\mathbb{R}^{d}$, i.e. we wish to evaluate integrals of the form:
\begin{equation}
\label{eq:expectation}
\pi(f):=\int_{\mathbb{R}^{d}}f(x)\pi(\mathrm{d}x).
\end{equation}
As is typical in many applications, particularly in molecular dynamics
and Bayesian inference, the density (for convenience denoted by the same symbol $\pi$) is known only up to a normalization constant; furthermore, the dimension of the underlying space is quite often large enough to render deterministic quadrature schemes computationally
infeasible. 

A standard approach to approximating such integrals is
Markov Chain Monte Carlo (MCMC) techniques \cite{GelmanCaStDuVeRu2014,liu2008monte,robert2013monte}, where a Markov process $(X_{t})_{t\geq0}$ is constructed which is ergodic with respect to the probability measure $\pi$. Then, defining the long-time average
\begin{equation}
\label{eq:estimator}
\pi_{T}(f):=\frac{1}{T}\int_{0}^{T}f(X_{s})\mathrm{d}s
\end{equation}
for $f\in L^{1}(\pi)$, the ergodic theorem guarantees almost sure
convergence of the long-time average $\pi_{T}(f)$ to $\pi(f)$.\\

There are infinitely many Markov, and, for the purposes of this paper diffusion, processes that can be constructed in such a way that they are ergodic with respect to the target distribution. A natural question is then how to choose the ergodic diffusion process $(X_{t})_{t\geq0}$. Naturally the choice should be dictated by the requirement that the computational cost of (approximately) calculating~\eqref{eq:expectation} is minimized. A standard example is given by the \emph{overdamped Langevin dynamics} defined to be
the unique (strong) solution $(X_{t})_{t\ge0}$ of the following stochastic differential equation
(SDE):
\begin{equation}
\mathrm{d}X_{t}=-\nabla V(X_{t})\mathrm{d}t+\sqrt{2}\mathrm{d}W_{t},\label{eq:overdamped}
\end{equation}
where $V=-\log\pi$ is the potential associated with the smooth positive
density $\pi$. Under appropriate assumptions on $V$, i.e. on the measure $\pi(\mathrm{d}x)$, the process $(X_{t})_{t\ge0}$ is ergodic and in fact reversible with respect to the target distribution.
\\
Another well-known example is the \emph{underdamped Langevin
dynamics} given by $(X_t)_{t\ge0} = (q_t, p_t)_{t\ge0}$ defined on the extended space (phase space) $\mathbb{R}^{d}\times\mathbb{\mathbb{R}}^{d}$ by the following pair of coupled SDEs:
\begin{equation}
\begin{aligned}\mathrm{d}q_{t} & =M^{-1}p_{t}\mathrm{d}t,\\
\mathrm{d}p_{t} & =-\nabla V(q_{t})\mathrm{d}t-\Gamma M^{-1}p_{t}\mathrm{d}t+\sqrt{2\Gamma}\mathrm{d}W_{t},
\end{aligned}
\label{eq:langevin}
\end{equation}
where mass and friction tensors $M$ and $\Gamma$, respectively, are assumed to be symmetric positive definite matrices. It is well-known~\cite{pavliotis2014stochastic,LS2016} that $(q_{t},p_{t})_{t\ge0}$ is ergodic with respect to the measure $\widehat{\pi}:=\pi\otimes\mathcal{N}(0,M)$, having density with respect to the Lebesgue measure on $\mathbb{R}^{2d}$ given by
\begin{equation}
\label{eq:augmented target}
\widehat{\pi}(q,p)=\frac{1}{\widehat{Z}}\exp\left(-V(q)-\frac{1}{2}p\cdot M^{-1}p\right),
\end{equation}
where $\widehat{Z}$ is a normalization constant. Note that $\widehat{\pi}$ has marginal $\pi$ with respect to $p$ and thus for functions $f\in L^{1}(\pi)$, we have that $\frac{1}{t}\int_0^t f(q_{t})\,\mathrm{d}t\rightarrow\pi(f)$ almost surely. Notice also that the dynamics restricted to the $q$-variables is no longer Markovian. The $p$-variables can thus be interpreted as giving some instantaneous memory to the system, facilitating efficient exploration of the state space. Higher order Markovian models, based on a finite dimensional (Markovian) approximation of the generalized Langevin equation can also be used~\cite{ceriotti_al09}.
\\\\
As there is a lot of freedom in choosing the dynamics in~\eqref{eq:estimator}, see the discussion in Section~\ref{sec:background}, it is desirable to choose the diffusion process $(X_t)_{t\ge0}$ in such a way that $\pi_T(f)$ can provide a good estimation of $\pi(f)$. The performance of the estimator~\eqref{eq:estimator} can be quantified in various manners. The ultimate goal, of course, is to choose the dynamics as well as the numerical discretization in such a way that the computational cost of the longtime-average estimator is minimized, for a given tolerance. The minimization of the computational cost consists of three steps: bias correction, variance reduction and choice of an appropriate discretization scheme. For the latter step see Section~\ref{sec:numerics} and~\cite[Sec. 6]{duncan2016variance}. 

Under appropriate conditions on the potential $V$ it can be shown that both \eqref{eq:overdamped} and \eqref{eq:langevin} converge to equilibrium exponentially fast, e.g. in relative entropy.  One performance objective would then be to choose the process $(X_t)_{t\ge0}$ so that this rate of convergence is maximised.
Conditions on the potential $V$ which guarantee exponential convergence to equilibrium, both in $L^{2}(\pi)$ and in relative entropy can be found in \cite{markowich2000trend,bakry2013analysis}. A powerful technique for proving exponentially fast convergence to equilibrium that will be used in this paper is C. Villani's theory of hypocoercivity~\cite{villani2009hypocoercivity}. In the case when the target measure $\pi$ is Gaussian, both the overdamped ~\eqref{eq:overdamped} and the underdamped~\eqref{eq:langevin} dynamics become generalized Ornstein-Uhlenbeck processes. For such processes the entire spectrum of the generator -- or, equivalently, the Fokker-Planck operator -- can be computed analytically and, in particular, an explicit formula for the $L^2$-spectral gap can be obtained~\cite{Metafune_formula,OPP12,OPP2015}.  A detailed analysis of the convergence to equilibrium in relative entropy for stochastic differential equations with linear drift, i.e. generalized Ornstein-Uhlenbeck processes, has been carried out in \cite{Arnold2014}. 
\\\\
In addition to speeding up convergence to equilibrium, i.e. reducing the bias of the estimator~\eqref{eq:estimator}, one is naturally also interested in reducing the asymptotic variance. Under appropriate conditions on the target measure $\pi$ and the observable $f$, the estimator $\pi_T(f)$ satisfies a central limit theorem (CLT)~\cite{KomorowskiLandimOlla2012}, that is,
$$
	\frac{1}{\sqrt{T}}\left(\pi_T(f) - \pi(f)\right) \xrightarrow[T\rightarrow\infty]{d} \mathcal{N}(0, 2\,\sigma^2_f),
$$
where $\sigma^2_f < \infty$ is the \emph{asymptotic variance} of the estimator $\pi_T(f)$. The asymptotic variance characterises how quickly fluctuations of $\pi_T(f)$ around $\pi(f)$ contract to $0$. Consequently, another natural objective is to choose the process $(X_t)_{t\ge0}$ such that $\sigma^2_f$ is as small as possible. It is well known that the asymptotic variance can be expressed in terms of the solution to an appropriate Poisson equation for the generator of the dynamics~\cite{KomorowskiLandimOlla2012}
\begin{equation}\label{e:poisson}
- \cL \phi = f - \pi (f), \quad \sigma^2_f = \int_{\mathbb{R}^d} \phi (- \cL \phi) \, \pi(\mathrm{d}x).
\end{equation}
Techniques from the theory of partial differential equations can then be used in order to study the problem of minimizing the asymptotic variance. This is the approach that was taken in~\cite{duncan2016variance}, see also~\cite{asvar_Hwang}, and it will also be used in this paper.  

Other measures of performance have also been considered.  For example, in~\cite{LDgraphs,LargeDeviations}, performance of the estimator is quantified in terms of the rate functional of the ensemble measure $\frac{1}{t}\int_0^t \delta_{X(t)}(dx)$. See also~\cite{JoulinOllivier2010} for a study of the nonasymptotic behaviour of MCMC techniques, including the case of overdamped Langevin dynamics.  

Similar analyses have been carried out for various modifications of \eqref{eq:overdamped}. Of particular interest to us are the {\it Riemannian manifold MCMC}~\cite{GirolamiCalderhead2011} (see the discussion in Section~\ref{sec:background}) and the {\it nonreversible Langevin samplers}~\cite{Hwang1993,Hwang2005}. As a particular example of the general framework that was introduced in~\cite{GirolamiCalderhead2011}, we mention the preconditioned overdamped Langevin dynamics that was presented in~\cite{alrachid2016some} 
\begin{equation}\label{e:precond-mala}
	dX_t = -P \nabla V(X_t)\,dt + \sqrt{2P}\,dW_t.
\end{equation}
In this paper, the long-time behaviour of as well as the asymptotic variance of the corresponding estimator $\pi_T(f)$ are studied and applied to equilibrium sampling in molecular dynamics. A variant of the standard underdamped Langevin dynamics that can be thought of as a form of preconditioning and that has been used by practitioners is the {\it mass-tensor molecular dynamics}~\cite{Bennett1975267}.

The nonreversible overdamped Langevin dynamics 
\begin{equation}
\label{eq:nonreversible_overdamped}
	dX_t = -\left(\nabla V(X_t) - \gamma(X_t)\right)\,dt + \sqrt{2}\,dW_t,
\end{equation}
where the vector field $\gamma$ satisfies $\nabla\cdot(\pi \gamma) = 0$ is ergodic (but not reversible) with respect to the target measure $\pi$ for all choices of the divergence-free vector field $\gamma$.  The asymptotic behaviour of this process was considered for Gaussian diffusions in~\cite{Hwang1993}, where the rate of convergence of the covariance to equilibrium was quantified in terms of the choice of $\gamma$. This work was extended to the case of non-Gaussian target densities, and consequently for nonlinear SDEs of the form~\eqref{eq:nonreversible_overdamped} in~\cite{Hwang2005}. The problem of constructing the optimal nonreversible perturbation, in terms of the $L^2(\pi)$ spectral gap for Gaussian target densities was studied in~ \cite{LelievreNierPavliotis2013} see also~\cite{hwang2014}. Optimal nonreversible perturbations with respect to miniziming the asymptotic variance were studied in~\cite{duncan2016variance,asvar_Hwang}. In all these works it was shown that, in theory (i.e. without taking into account the computational cost of the discretization of the dynamics~\eqref{eq:nonreversible_overdamped}), the nonreversible Langevin sampler~\eqref{eq:nonreversible_overdamped} always outperforms the reversible one~\eqref{eq:overdamped}, both in terms of converging faster to the target distribution as well as in terms of having a lower asymptotic variance. It should be emphasized that the two optimality criteria, maximizing the spectral gap and minimizing the asymptotic variance, lead to different choices for the nonreversible drift $\gamma(x)$.


The goal of this paper is to extend the analysis presented in~\cite{duncan2016variance,LelievreNierPavliotis2013} by introducing the following modification of the standard underdamped Langevin dynamics:
\begin{equation}
\label{eq:perturbed_underdamped}
\begin{aligned}
\mathrm{d}q_{t} & =M^{-1}p_{t}\mathrm{d}t-\mu J_{1}\nabla V(q_{t})\mathrm{d}t, \\
\mathrm{d}p_{t} & =-\nabla V(q_{t})\mathrm{d}t-\nu J_{2}M^{-1}p_{t}\mathrm{d}t-\Gamma M^{-1}p_{t}\mathrm{d}t+\sqrt{2\Gamma}\mathrm{d}W_{t},
\end{aligned}
\end{equation}
where $M,\Gamma\in\mathbb{R}^{d\times d}$ are constant strictly positive definite matrices, $\mu$ and $\nu$ are scalar constants and $J_1, J_2 \in \mathbb{R}^{d\times d}$ are constant skew-symmetric matrices. 
As demonstrated in Section~\ref{sec:perturbed_langevin}, the process defined by \eqref{eq:perturbed_underdamped} will be ergodic with respect to the Gibbs measure $\widehat{\pi}$ defined in \eqref{eq:augmented target}.  
\\\\
Our objective is to investigate the use of these dynamics for computing ergodic averages of the form \eqref{eq:estimator}. To this end, we study the long time behaviour of \eqref{eq:perturbed_underdamped} and, using hypocoercivity techniques, prove that the process converges exponentially fast to equilibrium.  This perturbed underdamped Langevin process introduces a number of parameters in addition to the mass and friction tensors which must be tuned to ensure that the process is an efficient sampler.  For Gaussian target densities, we derive estimates for the spectral gap and the asymptotic variance, valid in certain parameter regimes.  Moreover, for certain classes of observables, we are able to identify the choices of parameters which lead to the optimal performance in terms of asymptotic variance. While these results are valid for Gaussian target densities, we advocate these particular parameter choices also for more complex target densities.  To demonstrate their efficacy, we perform a number of numerical experiments on more complex, multimodal distributions. In particular, we use the Langevin sampler~\eqref{eq:perturbed_underdamped} in order to study the problem of diffusion bridge sampling. 
\\\\
The rest of the paper is organized as follows.  In Section \ref{sec:background} we present some background material on Langevin dynamics, we construct general classes of Langevin samplers and we introduce criteria for assessing the performance of the samplers.  In Section \ref{sec:perturbed_langevin} we study qualitative properties of the perturbed underdamped Langevin dynamics \eqref{eq:perturbed_underdamped} including exponentially fast convergence to equilibrium and the overdamped limit. In Section~\ref{sec:Gaussian} we study in detail the performance of the Langevin sampler~\eqref{eq:perturbed_underdamped} for the case of Gaussian target distributions. In Section~\ref{sec:numerics} we introduce a numerical scheme for simulating the perturbed dynamics~\eqref{eq:perturbed_underdamped} and we present numerical experiments on the implementation of the proposed samplers for the problem of diffusion bridge sampling. Section~\ref{sec:outlook} is reserved for conclusions and suggestions for further work. Finally, the appendices contain the proofs of the main results presented in this paper and of several technical results.

%% file: background.tex
\subsection{Background and Preliminaries}

In this section we consider estimators of the form \eqref{eq:estimator} where $(X_t)_{t\ge0}$ is a diffusion process given by the solution of the following It\^{o} SDE:
\begin{equation}
\label{eq:sde_general}
	\mathrm{d}X_t = a(X_t)\,\mathrm{d}t + \sqrt{2}b(X_t)\,\mathrm{d}W_t,
\end{equation}
with drift coefficient $a: \R^d \rightarrow \R^d$ and  diffusion coefficient $b:\R^d \rightarrow \R^{d\times m}$ both having smooth components, and where $(W_t)_{t\ge0}$ is a standard $\R^m$--valued Brownian motion.  Associated with (\ref{eq:sde_general}) is the infinitesimal generator $\gen$ formally given by
\begin{equation}
\label{eq:generator_general}
\gen f = a\cdot\nabla f  + \Sigma : D^2 f, \quad f \in C^2_c(\R^d)
\end{equation}
where $\Sigma = bb^\top$, $D^2 f$ denotes the Hessian of the function $f$ and $\, : \,$ denotes the Frobenius inner product.  In general, $\Sigma$ is nonnegative definite, and could possibly be degenerate.   In particular, the infinitesimal generator \eqref{eq:generator_general} need not be uniformly elliptic.  To ensure that the corresponding semigroup exhibits sufficient smoothing behaviour,  we shall require that  the process \eqref{eq:sde_general} is hypoelliptic in the sense of H{\"o}rmander.   If this condition holds, then irreducibility of the process $(X_t)_{t\ge0}$ will be an immediate consequence of the existence of a strictly positive invariant distribution $\pi(x)\mathrm{d}x$, see \cite{kliemann1987recurrence}.
\\\\
Suppose that $(X_t)_{t\geq 0}$ is nonexplosive.  It follows from the hypoellipticity assumption that the process $(X_t)_{t\geq 0}$ possesses a smooth transition density $p(t,x, y)$ which is defined for all $t \geq 0$ and $x, y \in \R^d$, \cite[Theorem VII.5.6]{bass1998diffusions}.   The associated strongly continuous Markov semigroup $(P_t)_{t\geq 0}$ is defined by 
\begin{equation}
	P_t f(x) = \int_{\R^d} p(t, x, y)f(y)\,\mathrm{d}y,\quad t \geq 0.
\end{equation}  
Suppose that $(P_t)_{t\ge0}$ is invariant with respect to the target distribution $\pi(x)\,\mathrm{d}x$, i.e.
\begin{equation*}
	\int_{\mathbb{R}^d} P_t f(x)\pi(x)\,\mathrm{d}x = \int_{\mathbb{R}^d} f(x)\pi(x)\,\mathrm{d}x,\quad t \geq 0,
\end{equation*}
for all bounded continuous functions $f$. Then $(P_t)_{t\ge0}$ can be extended to a positivity preserving contraction semigroup on $L^2(\pi)$ which is strongly continuous.   Moreover, the infinitesimal generator corresponding to $(P_t)_{t\ge0}$ is given by an extension of $(\gen, C^{2}_c(\R^d))$, also denoted by $\gen$.
\\\\
Due to hypoellipticity, the probability measure $\pi$ on $\mathbb{R}^d$ has a smooth and positive density with respect to the Lebesgue measure, and (slightly abusing the notation) we will denote this density also by $\pi$. Let $L^2(\pi)$ be the Hilbert space of $\pi$-square integrable functions equipped with inner product $\inner{\cdot}{\cdot}_{L^2(\pi)}$ and norm $\Norm{\cdot}_{L^2(\pi)}$. We will also make use of the \emph{Sobolev space} 
\begin{equation}
H^1(\pi)=\{f \in L^2(\pi):\quad \Vert\nabla f\Vert^2_{L^2(\pi)}<\infty\}
\end{equation} of $L^2(\pi)$-functions with weak derivatives in $L^2(\pi)$, equipped with norm
\begin{equation*}
\Vert f \Vert ^2_{H^1(\pi)} = \Vert f \Vert ^2_{L^2(\pi)} + \Vert \nabla f \Vert ^2_{L^2(\pi).}
\end{equation*}

\subsection{A General Characterisation of Ergodic Diffusions}
\label{sec:characterisation}

A natural question is what conditions on the coefficients $a$ and $b$ of \eqref{eq:sde_general} are required to ensure that $(X_t)_{t\ge0}$ is invariant with respect to the distribution $\pi(x)\,\mathrm{d}x$.    The following result provides a necessary and sufficient condition for a diffusion process to be invariant with respect to a given target distribution.

\begin{theorem}
\label{theorem:invariance_theorem}
Consider a diffusion process $(X_t)_{t\ge0}$ on $\mathbb{R}^{d}$ defined by the unique, non-explosive solution to the It\^{o} SDE \eqref{eq:sde_general} with drift $a \in C^1(\mathbb{R}^{d}; \mathbb{R}^d)$ and  diffusion coefficient $b\in C^1(\mathbb{R}^d; \mathbb{R}^{d\times m})$.  Then $(X_t)_{t\ge0}$ is invariant with respect to ${\pi}$ if and only if  
\begin{equation}
\label{eq:invariant_drift}
a = \Sigma \nabla \log \pi + \nabla\cdot \Sigma + \gamma,
\end{equation} 
where $\Sigma = bb^\top$ and $\gamma: \mathbb{R}^D\rightarrow \mathbb{R}^D$  is a continuously differentiable vector field satisfying 
\begin{equation}
\label{eq:invariance_condition}
\nabla\cdot\left(\pi \gamma \right) = 0.
\end{equation}  
If additionally $\gamma \in L^1({\pi})$, then there exists a skew-symmetric matrix function $C:\mathbb{R}^d \rightarrow \mathbb{R}^{d\times d}$ such that
$$
  \gamma = \frac{1}{{\pi}} \nabla\cdot\left({\pi} C \right).
$$
In this case the infinitesimal generator can be written as an $L^2(\pi)$-extension of
$$
\mathcal{L}f = \frac{1}{{\pi}}\nabla\cdot\left((\Sigma + C){\pi}\nabla f\right),\quad f\in C^2_c(\mathbb{R}^d).
$$ 
\end{theorem}

The proof of this result can be found in~\cite[Ch. 4]{pavliotis2014stochastic}; similar versions of this characterisation can be found in~\cite{villani2009hypocoercivity} and \cite{Hwang2005}. See also~\cite{ma2015complete}.

\begin{remark}
If \eqref{eq:invariant_drift} holds and $\mathcal{L}$ is hypoelliptic it follows immediately that $(X_t)_{t\ge0}$ is ergodic with unique invariant distribution $\pi(x)\,\mathrm{d}x$.
\end{remark}

More generally, we can consider It\^{o} diffusions in an extended phase space:

\begin{equation}\label{e:SDE}
\mathrm{d} Z_{t} = b(Z_t) \, \mathrm{d}t + \sqrt{2}\sigma(Z_{t}) \, \mathrm{d}W_{t}, 
\end{equation}
where $(W_{t})_{t\ge0}$ is a standard Brownian motion in $\R^{N}$, $N \geq d$. This is a Markov process with generator

\begin{equation}\label{e:gen}
\mathcal{L} = b(z) \cdot \nabla_z +  \Sigma(z) : D^2_z , 
\end{equation}
where $\Sigma(z) = \big( \sigma \sigma^{T} \big)(z)$. We will consider dynamics $(Z_t)_{t\ge0}$ that is ergodic with respect to $\pi_z(z) \, \mathrm{d}z$ such that 
\begin{equation}\label{e:marginal}
\int_{\R^{m}} \pi_z (x, \, y) \, \mathrm{d}y = \pi(x).
\end{equation}
where $z = (x, \, y), \; x \in \R^d, \, y \in \R^m, \; d+m = N$.

There are various well-known choices of dynamics which are invariant (and indeed ergodic) with respect to the target distribution $\pi(x)\mathrm{d}x$.
\begin{enumerate}
  \item Choosing  $b = I$ and $\gamma = 0$ we immediately recover the overdamped Langevin dynamics (\ref{eq:overdamped}).
  \item Choosing $b = I$, and $\gamma \neq 0$ such that \eqref{eq:invariance_condition} holds gives rise to the nonreversible overdamped equation defined by \eqref{eq:nonreversible_overdamped}.  As it  satisfies the conditions of Theorem \ref{theorem:invariance_theorem}, it is ergodic with respect to $\pi$.  In particular choosing $\gamma(x) = J\nabla V(x)$ for a constant skew-symmetric matrix $J$ we obtain
  \begin{equation}
  \label{eq:nonrev_overdamped_J}
    \mathrm{d}X_t = -(I + J)\nabla V(X_t)\,\mathrm{d}t + \sqrt{2}\,\mathrm{d}W_t,
  \end{equation}
  which has been studied in previous works.  

  \item Given a target density $\pi > 0$ on $\R^d$, if we consider the augmented target density $\widehat{\pi}$ on $\R^{2d}$ given in \eqref{eq:augmented target},
  then choosing
  \begin{equation}
  \label{eq:underdamped_gamma}
    \gamma((q,p)) = \left(\begin{array}{c} M^{-1}p \\ -\nabla V(q)\end{array}\right)
  \end{equation}
  and 
  \begin{equation}
    \label{eq:underdamped_sigma}
    b = \left(\begin{array}{c}\boldsymbol{0} \\ \sqrt{\Gamma}\end{array}\right) \in \mathbb{R}^{2d \times d},
  \end{equation}
  where $M$ and $\Gamma$ are positive definite symmetric matrices, the conditions of Theorem \ref{theorem:invariance_theorem} are satisfied for the target density $\widehat{\pi}$.  The resulting dynamics $(q_t, p_t)_{t\ge0}$ is determined by the underdamped Langevin equation (\ref{eq:langevin}). It is straightforward to verify that the generator is hypoelliptic, \cite[Sec 2.2.3.1]{Free_energy_computations}, and thus $(q_t, p_t)_{t\ge0}$ is ergodic. 

  \item More generally, consider the augmented target density $\widehat{\pi}$ on $\mathbb{R}^{2d}$ as above, and choose 
  \begin{equation}
  \label{eq:underdamped_gamma_perturbed}
    \gamma((q,p)) = \left(\begin{array}{c} M^{-1}p - \mu J_1\nabla V(q) \\ -\nabla V(q) - \nu J_2 M^{-1}p\end{array}\right)
  \end{equation}
  and 
  \begin{equation}
    \label{eq:underdamped_sigma_pertured}
    b = \left(\begin{array}{c}\boldsymbol{0} \\ \sqrt{\Gamma}\end{array}\right) \in \mathbb{R}^{2d \times d},
  \end{equation}
  where $\mu$ and $\nu$ are scalar constants and $J_1, J_2 \in \mathbb{R}^{d\times d}$ are constant skew-symmetric matrices.  With this choice we recover the perturbed Langevin dynamics \eqref{eq:perturbed_underdamped}.  It is straightforward to check that \eqref{eq:underdamped_gamma_perturbed}  satisfies the invariance condition (\ref{eq:invariance_condition}), and thus Theorem \ref{theorem:invariance_theorem} guarantees that (\ref{eq:perturbed_underdamped}) is invariant with respect to $\widehat{\pi}$. 

  \item In a similar fashion, one can introduce an  augmented target density on $\mathbb{R}^{(m+2)d}$, with
  \begin{align*}
   \widehat{\widehat{\pi}}(q, p, u_1,\ldots, u_m) \propto e^{-\frac{|p|^2}{2} - \frac{u_1^2 + \ldots + u_m^2}{2}-V(q)},
  \end{align*}
  where $p, q, u_i \in \mathbb{R}^d$, for $i=1,\ldots, m$.  Clearly $\int_{\mathbb{R}^{d}\times \mathbb{R}^{md}} \widehat{\widehat{\pi}}(q, p, u_1,\ldots,u_m)\,\mathrm{d}p\,\mathrm{d}u_1\,\ldots \mathrm{d}u_m = \pi(q)$. We now define $\gamma:\mathbb{R}^{(m+2)d}\rightarrow \mathbb{R}^{(m+2)d}$ by
  $$
  \gamma(q,p, u_1,\ldots,u_m) = \left(\begin{array}{c}p \\ -\nabla_q V(q) + \sum_{j=1}^{m} \lambda_j u_j \\ -\lambda_1 p \\ \vdots \\ -\lambda_m p \end{array}\right) 
  $$
  and $b: \mathbb{R}^{(m+2)d}\rightarrow \mathbb{R}^{(m+2)d\times (m+2)d}$ by 
  $$
  b(q,p,u_1,\ldots, u_m) = \left(\begin{array}{cccccc}\boldsymbol{0} & \boldsymbol{0} & \boldsymbol{0} & \boldsymbol{0} & \ldots & \boldsymbol{0}\\ \boldsymbol{0} & \boldsymbol{0} & \boldsymbol{0} & \boldsymbol{0}& \ldots & \boldsymbol{0} \\ \boldsymbol{0} & \boldsymbol{0} & \sqrt{\alpha_1}I_{d\times d} & \boldsymbol{0} & \ldots & \boldsymbol{0} \\ \boldsymbol{0} & \boldsymbol{0} & \boldsymbol{0} & \sqrt{\alpha_2}I_{d\times d} & \ldots & \boldsymbol{0} \\ 
   \vdots & \vdots & \vdots & \vdots & \ddots & \vdots \\ \boldsymbol{0} & \boldsymbol{0} & \boldsymbol{0} &\boldsymbol{0} & \ldots & \sqrt{\alpha_m}I_{d\times d}\end{array}\right),
  $$
  where $\lambda_i \in \mathbb{R}$ and $\alpha_i > 0$, for $i=1,\ldots, m$.  The resulting process \eqref{eq:sde_general} is given by 
  \begin{equation}
  \label{eq:gle_markov}
  \begin{aligned}
    \mathrm{d}q_t &= p_t \,\mathrm{d}t \\ 
    \mathrm{d}p_t &= -\nabla_q V(q_t)\,\mathrm{d}t + \sum_{j=1}^{d}\lambda_j u^{j}(t)\,\mathrm{d}t \\
    \mathrm{d}u^{1}_t &= -\lambda_1 p_t\,\mathrm{d}t -\alpha_1 u^{1}_t\,\mathrm{d}t + \sqrt{2\alpha_1 }\,\mathrm{d}W^{1}_t\\
    \vdots &  \\
    \mathrm{d}u^{m}_t &= -\lambda_m p_t\,\mathrm{d}t -\alpha_m u^{m}_t\,\mathrm{d}t + \sqrt{2\alpha_m }\,\mathrm{d}W^{m}_t,
  \end{aligned}
  \end{equation}
  where $(W^1_t)_{t\ge0}, \ldots (W^m_t)_{t\ge0}$ are independent $\mathbb{R}^d$--valued Brownian motions.   This process is ergodic with unique invariant distribution $\widehat{\widehat{\pi}}$, and under appropriate conditions on $V$, converges exponentially fast to equilibrium in relative entropy \cite{ottobre2011asymptotic}.  Equation \eqref{eq:gle_markov} is a Markovian representation of a generalised Langevin equation of the form
  \begin{align*}
  \mathrm{d}q_t &= p_t \,\mathrm{d}t \\
  \mathrm{d}p_t &= -\nabla_{q}V(q_t) \,\mathrm{d}t - \int_0^t F(t-s)p_s\,\mathrm{d}s + N(t),
  \end{align*}
  where $N(t)$ is a mean-zero stationary Gaussian process with autocorrelation function $F(t)$, i.e.
  $$
    \mathbb{E}\left[ N(t) \otimes N(s) \right] = F(t-s)I_{d\times d},
  $$
  and 
  $$
    F(t) = \sum_{i=1}^{m} \lambda_i^2 e^{-\alpha_i|t|}.
  $$

  \item Let $\widetilde{\pi}(z) \propto \exp(-\Phi(z))$ be a positive density on $\mathbb{R}^N$ where $N > d$ such that 
  $$
    \pi(x) = \int_{\mathbb{R}^{N-d}}\widetilde{\pi}(x,z)\,\mathrm{d}z,
  $$
  where $(x, y)\in \mathbb{R}^d\times \mathbb{R}^{N-d}$. Then choosing $b = I_{D\times D}$ and $\gamma = 0$ we obtain the dynamics
  \begin{align*}
      \mathrm{d}X_t &= -\nabla_x \Phi(X_t, Y_t)\,\mathrm{d}t + \sqrt{2}\,\mathrm{d}W^{1}_t \\ 
      \mathrm{d}Y_t &= -\nabla_y \Phi(X_t, Y_t)\,\mathrm{d}t + \sqrt{2}\,\mathrm{d}W^{2}_t,
  \end{align*}
  then  $(X_t, Y_t)_{t\ge0}$ is immediately ergodic with respect to $\widetilde{\pi}$.
\end{enumerate}

\subsection{Comparison Criteria}
\label{sec:comparison}

For a fixed observable $f$, a natural measure of accuracy of the estimator $\pi_T(f) = t^{-1}\int_0^{t}f(X_s)\,\mathrm{d}s$ is the \emph{mean square error} (MSE) defined by
\begin{equation}
\label{eq:mse}
MSE(f, T) := \mathbb{E}_{x}\norm{\pi_T(f) - \pi(f)}^2,
\end{equation}
where $\mathbb{E}_{x}$ denotes the expectation conditioned on the process $(X_t)_{t\ge0}$ starting at $x$.  
It is instructive to introduce the decomposition $MSE(f, T) = \mu^2(f, T) + \sigma^2(f, T)$, where
\begin{equation}
\label{eq:bias_variance_decomposition}
  \mu(f, T) = \norm{\mathbb{E}_{x}[\pi_T(f)] - \pi(f)}\quad\mbox{ and }\quad \sigma^2(f, T) = \mathbb{E}_{x}\norm{\pi_T(f) - \pi(f)}^2 = \mbox{Var}[\pi_T(f)].
\end{equation}
Here $\mu(f, T)$ measures the bias of the estimator $\pi_T(f)$ and $\sigma^2(f, T)$ measures the variance of fluctuations of $\pi_T(f)$ around the mean.   
\\\\
The speed of convergence to equilibrium of the process $(X_t)_{t\ge0}$ will control both the bias term $\mu(f, T)$ and the variance $\sigma^2(f, T)$.  To make this claim more precise, suppose that the semigroup $(P_t)_{t\ge0}$ associated with $(X_t)_{t\ge0}$ decays exponentially fast in $L^2(\pi)$, i.e. there exist constants  $\lambda > 0$ and $C\ge1$ such that
\begin{equation}
\label{eq:hypocoercive estimate}
 \left\lVert P_t g - \pi(g) \right\rVert_{L^2(\pi)} \leq C e^{- \lambda t} \left\lVert g-\pi(g) \right\rVert_{L^2(\pi)},\quad g\in L^2(\pi).
\end{equation}
\begin{remark}
	If \eqref{eq:hypocoercive estimate} holds with $C=1$, this estimate is equivalent to $-\gen$ having a spectral gap in $L^2(\pi)$. Allowing for a constant $C>1$ is essential for our purposes though in order to treat nonreversible and degenerate diffusion processes by the theory of \emph{hypocoercivity} as outlined in \cite{villani2009hypocoercivity}.  
\end{remark}
The following lemma characterises the decay of the bias $\mu(f,T)$ as $T\rightarrow \infty$ in terms of $\lambda$ and $C$.  The proof can be found in Appendix \ref{app:proofs}.
\begin{lemma}
\label{lemma:bias}
Let $(X_t)_{t\geq 0}$ be the unique, non-explosive solution of \eqref{eq:sde_general}, such that $X_0 \sim \pi_0 \ll \pi$ and $\frac{d\pi_0}{d\pi} \in L^2(\pi)$, where $\frac{d\pi_0}{d\pi}$ denotes the Radon-Nikodym derivative of $\pi_0$ with respect to $\pi$.  Suppose that the process is ergodic with respect to $\pi$ such that the Markov semigroup $(P_t)_{t\geq 0}$ satisfies (\ref{eq:hypocoercive estimate}).  Then for $f \in L^\infty(\pi)$,
\begin{equation*}
\mu(f, T)  \leq \frac{C}{\lambda T}\left({1 - e^{-\lambda T}}\right)\lVert f \rVert_{L^\infty}\mbox{Var}_{\pi}\left[\frac{d\pi_0}{d\pi}\right]^{\frac{1}{2}}.
\end{equation*}
\end{lemma}

The study of the behaviour of the variance $\sigma^2(f, T)$ involves deriving a central limit theorem for the additive functional $\int_0^t f(X_t)-\pi(f)\,\mathrm{d}t$.  As discussed in \cite{cattiaux2012central}, we reduce this problem to proving well-posedness of the Poisson equation 
\begin{equation}
\label{eq:poisson_general}
-\gen \chi = f - \pi(f),\quad \pi(\chi) = 0.
\end{equation}
The only complications in this approach arise from the fact that the generator $\gen$ need not be symmetric in $L^2(\pi)$ nor uniformly elliptic. The following result summarises conditions for the well-posedness of the Poisson equation and it also provides with  us with a formula for the asymptotic variance.  The proof can be found in Appendix \ref{app:proofs}.

\begin{lemma}
\label{lemma:variance}
Let $(X_t)_{t\geq 0}$ be the unique, non-explosive solution of \eqref{eq:sde_general}  with smooth drift and diffusion coefficients, such that the corresponding infinitesimal generator is hypoelliptic.  Syppose that $(X_t)_{t\ge0}$ is ergodic with respect to $\pi$ and moreover, $(P_t)_{t\ge0}$ decays exponentially fast in $L^2(\pi)$ as in \eqref{eq:hypocoercive estimate}.  Then for all $f\in L^2(\pi)$, there exists a unique mean zero solution $\chi$ to the Poisson equation \eqref{eq:poisson_general}.  If $X_0 \sim \pi$, then for all $f \in C^\infty(\mathbb{R}^d) \cap L^2(\pi)$
\begin{equation}
\label{eq:CLT}
  \sqrt{T}\left(\pi_T(f) - \pi(f)\right) \xrightarrow[T\rightarrow\infty]{d} \mathcal{N}(0, 2\sigma^2_f),
\end{equation}
where $\sigma^2_f$ is the asymptotic variance defined by
\begin{equation}
\label{eq:asymptoticvariance}
\sigma^2_{f} = \inner{\chi}{(-\gen)\chi}_{L^2(\pi)} = \inner{\nabla \chi}{\Sigma\nabla\chi}_{L^2(\pi)}.
\end{equation}
Moreover, if $X_0 \sim \pi_0$ where $\pi_0 \ll \pi$ and $\frac{d\pi_0}{d\pi}\in L^2(\pi)$ then \eqref{eq:CLT} holds for all $f \in C^\infty(\mathbb{R}^d) \cap L^\infty(\pi)$.
\end{lemma}
Clearly, observables that only differ by a constant have the same asymptotic variance. 
In the	sequel, we will hence restrict our attention to observables $f\in L^{2}(\pi)$	satisfying $\pi(f)=0$, simplifying expressions~\eqref{eq:poisson_general}
	and~\eqref{eq:CLT}. The corresponding subspace of $L^2(\pi)$ will be denoted by 
\begin{equation}
L_{0}^2(\pi):=\{f \in L^2(\pi): \pi(f)=0\}.
\end{equation}
If the exponential decay estimate \eqref{eq:hypocoercive estimate} is satisfied, then Lemma \ref{lemma:variance} shows that $-\gen$ is invertible on $L^2_{0}(\pi)$, so we can express the asymptoptic variance as 
\begin{equation}
\label{eq:asym variance inverse}
\sigma_{f}^2=\langle f, (-\gen)^{-1} f \rangle_{L^2(\pi)}, \quad f \in L^2_{0}(\pi).
\end{equation}
Let us also remark that from the proof of Lemma \ref{lemma:variance} it follows that the inverse of $\mathcal{L}$ is given by
\begin{equation}
\mathcal{L}^{-1}=\int_0^{\infty}P_t \,\mathrm{d}t.
\end{equation}
We note that the constants $C$ and $\lambda$ appearing in the exponential decay estimate \eqref{eq:hypocoercive estimate} also control the speed of convergence of $\sigma^2(f, T)$ to zero.  Indeed, it is straightforward to show that if \eqref{eq:hypocoercive estimate} is satisfied, then the solution $\chi$ of \eqref{eq:poisson_general} satisfies
\begin{equation}\label{e:estim-sigma}
  \sigma^2_{f} = \inner{\chi}{f-\pi(f)}_{L^2(\pi)} \leq \frac{C}{\lambda}\Norm{f}^2_{L^2(\pi)}.
\end{equation}

Lemmas \ref{lemma:bias} and \ref{lemma:variance} would suggest that choosing the coefficients $\Sigma$ and $\gamma$ to optimize the constants $C$ and $\lambda$ in~\eqref{e:estim-sigma} would be an effective means of improving the performance of the estimator $\pi_T(f)$, especially since the improvement in performance would be uniform over an entire class of observables.  When this is possible, this is indeed the case. However, as has been observed in \cite{LelievreNierPavliotis2013,Hwang1993,Hwang2005}, maximising the speed of convergence to equilibrium is a delicate task.  As the leading order term in $MSE(f, T)$, it is typically sufficient to focus specifically on the asymptotic variance $\sigma^2_{f}$ and study how the parameters of the SDE \eqref{eq:sde_general} can be chosen to minimise $\sigma^2_{f}$. This study was undertaken in \cite{duncan2016variance} for processes of the form \eqref{eq:nonreversible_overdamped}.

%% file: perturbed_langevin.tex
The primary objective of this work is to compare the performances of the perturbed underdamped Langevin dynamics (\ref{eq:perturbed_underdamped}) and the unperturbed dynamics (\ref{eq:langevin}) according to the criteria outlined in Section \ref{sec:comparison} and to find suitable choices for the matrices $J_{1}$, $J_{2}$, $M$ and $\Gamma$ that improve the performance of the sampler.  We begin our investigations of (\ref{eq:perturbed_underdamped}) by establishing ergodicity and exponentially fast return to equilibrium, and by studying the overdamped limit of~\eqref{eq:perturbed_underdamped}. As the latter turns out to be nonreversible and therefore in principle superior to the usual overdamped limit~\eqref{eq:overdamped},e.g.~\cite{Hwang2005}, this calculation provides us with further motivation to study the proposed dynamics.
\\\\
For the bulk of this work, we focus on the particular case when the target measure is Gaussian, i.e. when the potential is given by $V(q)=\frac{1}{2}q^{T}Sq$
with a symmetric and positive definite precision matrix $S$ (i.e. the covariance matrix is given by $S^{-1}$). In this
case, we advocate the following conditions for the choice of parameters:\begin{subequations}
	\label{eq:optimal parameters}
	\begin{align}
	M & =S,\label{eq:M=00003DS}\\
	\Gamma & =\gamma S,\\
	SJ_{1}S & =J_{2},\label{eq: perturbation condition}\\
	\mu & =\nu.
	\end{align}
\end{subequations}
Under the above choices \eqref{eq:optimal parameters}, we show that the large perturbation limit $\lim_{\mu\rightarrow\infty} \sigma_f^2$ exists and is finite and we provide an explicit expression for it (see Theorem \ref{cor:limit_asym_var}). From this expression, we derive an algorithm for finding optimal choices for $J_1$ in the case of quadratic observables (see Algorithm \ref{alg:optimal J general}).
\\\\
If the friction coefficient is not too small ($\gamma > \sqrt {2}$), and under certain mild nondegeneracy conditions, we prove that adding a small perturbation will always decrease the asymptotic variance for observables of the form $f(q)=q\cdot Kq+l\cdot q+C$:
\[
\left. \frac{\mathrm{d}}{\mathrm{d}\mu}\sigma_{f}^{2}\right\rvert_{\mu=0}=0\quad\text{and }\quad \left. \frac{\mathrm{d}^{2}}{\mathrm{d}\mu^{2}}\sigma_{f}^{2}\right\rvert_{\mu=0}<0,
\]
see Theorem \ref{cor:small pert unit var}. 
In fact, we conjecture that this statement is true for arbitrary observables
$f\in L^{2}(\pi)$, but we have not been able to prove this. The dynamics (\ref{eq:perturbed_underdamped})
(used in conjunction with the conditions (\ref{eq:M=00003DS})-(\ref{eq: perturbation condition}))
proves to be especially effective when the observable is antisymmetric
(i.e. when it is invariant under the substitution $q\mapsto-q$) or when it
has a significant antisymmetric part. In particular, in Proposition~\ref{prop:antisymmetric observables} we show that under certain conditions on the spectrum of $J_1$, for any antisymmetric observable $f\in L^{2}(\pi)$ it holds that  $\lim_{\mu\rightarrow\infty}\sigma_{f}^{2}=0$.
\\\\
Numerical experiments and analysis show that departing significantly
from~\ref{eq: perturbation condition} in fact possibly decreases
the performance of the sampler. This is in stark contrast to~\eqref{eq:nonreversible_overdamped}, where it is not possible to increase the asymptotic variance by \emph{any} perturbation.  For that reason, until now it seems practical to use (\ref{eq:perturbed_underdamped})  as a sampler only when a reasonable estimate of the global covariance of the target distribution is available. In the case of Bayesian inverse problems and diffusion bridge sampling, the target measure $\pi$ is given with respect to a Gaussian prior. We demonstrate the effectiveness of our approach in these applications, taking the prior Gaussian covariance as $S$ in (\ref{eq:M=00003DS})-(\ref{eq: perturbation condition}).
\begin{remark}
	In \cite[Rem. 3]{LelievreNierPavliotis2013} another modification of (\ref{eq:langevin})
	was suggested (albeit with the simplifications $\Gamma=\gamma\cdot I$
	and $M=I$):
\end{remark}
\begin{align}
\mathrm{d}q_{t} & =(1-J)M^{-1}p_{t}\mathrm{d}t ,\nonumber \\
\mathrm{d}p_{t} & =-(1+J)\nabla V(q_{t})\mathrm{d}t-\Gamma M^{-1}p_{t}\mathrm{d}t+\sqrt{2\Gamma}\mathrm{d}W_{t},\label{eq: JJ perturbation}
\end{align}
$J$ again denoting an antisymmetric matrix. However, under the change
of variables $p\mapsto(1+J)\tilde{p}$ the above equations transform
into 
\begin{align*}
\mathrm{d}q_{t} & =\tilde{M}^{-1}p_{t}\mathrm{d}t,\\
\mathrm{d}\tilde{p_{t}} & =-\nabla V(q_{t})\mathrm{d}t-\tilde{\Gamma}\tilde{M}^{-1}\tilde{p}_{t}\mathrm{d}t+\sqrt{2\tilde{\Gamma}}\mathrm{d}\tilde{W}_{t},
\end{align*}
where $\tilde{M}=(1+J)^{-1}M(1-J)^{-1}$ and $\tilde{\Gamma}=(1+J)^{-1}\Gamma(1-J)^{-1}$.
Since any observable $f$ depends only on $q$ (the $p$-variables
are merely auxiliary), the estimator $\pi_T(f)$ as well as its associated convergence characteristics (i.e. asymptotic
variance and speed of convergence to equilibrium) are invariant under this transformation.
Therefore, (\ref{eq: JJ perturbation}) reduces to the underdamped
Langevin dynamics (\ref{eq:langevin}) and does not represent an independent approach to sampling. Suitable choices
of $M$ and $\Gamma$ will be discussed in Section \ref{sec:arbitrary covariance}.

\subsection{Properties of Perturbed Underdamped Langevin Dynamics}
\label{sec:hypocoercivity}

In this section we study some of the properties of the perturbed underdamped dynamics (\ref{eq:perturbed_underdamped}). First, note that its generator is given by
\begin{equation}
\label{eq:generator}
\mathcal{L}=\underbrace{\underbrace{M^{-1}p\cdot\nabla_{q}-\nabla_{q}V\cdot\nabla_{p}}_{\mathcal{L}_{ham}}\underbrace{-\Gamma M^{-1}p\cdot\nabla_{p}+\Gamma : D^2_{p}}_{\mathcal{L}_{therm}}}_{\mathcal{L}_0} \underbrace{-\mu J_{1}\nabla V \cdot \nabla_{q} - \nu J M^{-1} p \cdot \nabla_{p}}_{\mathcal{L}_{pert}},
\end{equation}	
decomposed into the perturbation $\mathcal{L}_{pert}$ and the unperturbed operator $\mathcal{L}_0$, which can be further split into the Hamiltonian part $\mathcal{L}_{ham}$ and the thermostat (Ornstein-Uhlenbeck) part $\mathcal{L}_{therm}$, see \cite{pavliotis2014stochastic,Free_energy_computations,LS2016}.

\begin{lemma}
\label{lem:hypoellipticity}
	The infinitesimal generator $\gen$~\eqref{eq:generator} is hypoelliptic.
\end{lemma}
\begin{proof}
	See Appendix \ref{app:hypocoercivity}.\qed
\end{proof}

An immediate corollary of this result and of Theorem \ref{theorem:invariance_theorem} is that the perturbed underdamped Langevin process \eqref{eq:perturbed_underdamped} is ergodic with unique invariant distribution $\widehat{\pi}$ given by \eqref{eq:augmented target}.
\\\\
As explained in Section \ref{sec:comparison}, the exponential decay estimate \eqref{eq:hypocoercive estimate} is crucial for our approach, as in particular it guarantees the well-posedness of the Poisson equation \eqref{eq:poisson_general}. 
From now on, we will therefore make the following assumption on the potential $V,$ required to prove exponential decay in $L^2(\pi)$:

\begin{assumption}
	\label{ass:bounded+Poincare}
	Assume that the Hessian of $V$ is \emph{bounded} and that the target measure $\pi(\mathrm{d}q) = \frac{1}{Z}e^{-V}\mathrm{d}q$ satisfies a \emph{Poincare inequality}, i.e. there exists a constant $\rho>0$ such that 
	\begin{equation}
	\int_{\mathbb{R}^d}\phi^2\mathrm{d}\pi \le \rho \int_{\mathbb{R}^d} \vert \nabla \phi \vert ^2 \mathrm{d}\pi, 
	\end{equation}
	holds for all $\phi \in L_{0}^2(\pi)\cap H^1(\pi)$.
\end{assumption}
Sufficient conditions on the potential so that Poincar\'{e}'s inequality holds, e.g. the Bakry-Emery criterion, are presented in~\cite{bakry2013analysis}.
\begin{theorem}
	\label{theorem:Hypocoercivity}Under Assumption \ref{ass:bounded+Poincare} there exist constants $C\ge 1$ and $\lambda>0$ such that the semigroup $(P_t)_{t\ge0}$ generated by $\gen$ satisfies exponential decay in $L^2(\pi)$ as in \eqref{eq:hypocoercive estimate}.
\end{theorem}
\begin{proof}
	See Appendix \ref{app:hypocoercivity}.
\end{proof}
\begin{remark}
	The proof uses the machinery of hypocoercivity developed in \cite{villani2009hypocoercivity}.
	However, it seems likely that using the framework of \cite{DolbeaultMouhotSchmeiser2015},
	the assumption on the boundedness of the Hessian of $V$ can be substantially
	weakened.
\end{remark}

\subsection{The Overdamped Limit}
\label{sec:overdamped}

In this section we develop a connection between the perturbed underdamped
Langevin dynamics (\ref{eq:perturbed_underdamped}) and
the nonreversible overdamped Langevin dynamics (\ref{eq:nonreversible_overdamped}). The analysis is very similar to the one presented in \cite[Section 2.2.2]{Free_energy_computations} and we will be brief. For convenience in this section we will perform the analysis on the $d$-dimensional torus $\mathbb{T}^d \cong (\mathbb{R} / \mathbb{Z})^d$, i.e. we will assume $q \in \mathbb{T}^d$.
Consider the following scaling of (\ref{eq:perturbed_underdamped}):
\begin{subequations}
\begin{eqnarray}
\mathrm{d}q_{t}^{\epsilon} & = &  \frac{1}{\epsilon}M^{-1}p_{t}^{\epsilon},\mathrm{d}t-\mu J_{1}\nabla_{q}V(q_{t})\mathrm{d}t, \\
\mathrm{d}p_{t}^{\epsilon} & = & -\frac{1}{\epsilon}\nabla_{q}V(q_{t}^{\epsilon})\mathrm{d}t-\frac{1}{\epsilon^{2}}\nu J_{2}M^{-1}p_{t}^{\epsilon}\mathrm{d}t-\frac{1}{\epsilon^{2}}\Gamma M^{-1}p_{t}^{\epsilon}\mathrm{d}t+\frac{1}{\epsilon}\sqrt{2\Gamma}\mathrm{d}W_{t},
\end{eqnarray}
\label{eq:rescaling}
\end{subequations}
valid for the small mass/small momentum regime 
\begin{equation*}
M  \rightarrow\epsilon^{2}M, \quad   p_{t}  \rightarrow\epsilon p_{t}.
\end{equation*}
Equivalently, those modifications can be obtained from subsituting
$\Gamma\rightarrow\epsilon^{-1}\Gamma$ and $t\mapsto\epsilon^{-1}t$,
and so in the limit as $\epsilon\rightarrow0$ the dynamics (\ref{eq:rescaling})
describes the limit of large friction with rescaled time. It turns
out that as $\epsilon\rightarrow0$, the dynamics (\ref{eq:rescaling})
converges to the limiting SDE 
\begin{equation}
\mathrm{d}q_{t}=-(\nu J_{2}+\Gamma)^{-1}\nabla_{q}V(q_{t})\mathrm{d}t-\mu J_{1}\nabla_{q}V(q_{t})\mathrm{d}t+(\nu J_{2}+\Gamma)^{-1}\sqrt{2\Gamma}\mathrm{d}W_{t}.\label{eq:overdamped limit}
\end{equation}
The following proposition makes this statement precise.
\begin{proposition}
	\label{prop: overdamped limit}Denote by $(q_{t}^{\epsilon},p_{t}^{\epsilon})$
	the solution to (\ref{eq:rescaling}) with (deterministic) initial
	conditions $(q_{0}^{\epsilon},p_{0}^{\epsilon})=(q_{init},p_{init})$
	and by $q_{t}^{0}$ the solution to (\ref{eq:overdamped limit}) with
	initial condition $q_{0}^{0}=q_{init}.$ For any $T>0$, $(q_{t}^{\epsilon})_{0\le t\le T}$
	converges to $(q_{t}^{0})_{0\le t\le T}$ in $L^{2}(\Omega,C([0,T]),\mathbb{T}^{d})$
	as $\epsilon\rightarrow0$, i.e. 
	\[
	\lim_{\epsilon\rightarrow0}\mathbb{E}\big(\sup_{0\le t\le T}\vert q_{t}^{\epsilon}-q_{t}^{0}\vert^{2}\big)=0.
	\]
\end{proposition}
\begin{remark}
	By a refined analysis, it is possible to get information on the rate of convergence; see, e.g.~\cite{PavlSt03,PavSt05a}.
\end{remark}
The limiting SDE (\ref{eq:overdamped limit}) is nonreversible due to the term $-\mu J_1 \nabla_q V(q_t)\mathrm{d}t$ and also because the
matrix $(\nu J_{2}+\Gamma)^{-1}$ is in general neither symmetric
nor antisymmetric.
This result, together with the fact that nonreversible perturbations
of overdamped Langevin dynamics of the form \eqref{eq:nonreversible_overdamped} are by now well-known to have improved
performance properties, motivates further investigation of the dynamics
(\ref{eq:perturbed_underdamped}).

\begin{remark}
	The limit we described in this section respects the invariant distribution,
	in the sense that the limiting dynamics (\ref{eq:overdamped limit})
	is ergodic with respect to the measure $\pi(dq)=\frac{1}{Z}e^{-V}\mathrm{d}q.$
	To see this, we have to check that (we are using the notation $\nabla$ instead of $\nabla_q$) 
	\[
	\mathcal{L}^{\dagger}(e^{-V})=-\nabla\cdot\big((\nu J_{2}+\Gamma)^{-1}\nabla e^{-V}\big)+\nabla\cdot(\mu J_{1}\nabla e^{-V})+\nabla\cdot\big((\nu J_{2}+\Gamma)^{-1}\Gamma(-\nu J_{2}+\Gamma)^{-1}\nabla e^{-V}\big)=0,
	\]
	where $\mathcal{L}^{\dagger}$ refers to the $L^{2}(\mathbb{R}^{d})$-adjoint
	of the generator of (\ref{eq:overdamped limit}), i.e. to the associated Fokker-Planck operator. Indeed, the term
	$\nabla\cdot(\mu e^{-V}J_{1}\nabla V)$ vanishes because of the
	antisymmetry of $J_{1}.$ Therefore, it remains to show that 
	\[
	\nabla\cdot\big((\nu J_{2}+\Gamma)^{-1}\Gamma(-\nu J_{2}+\Gamma)^{-1}-(\nu J_{2}+\Gamma)^{-1}\big)\nabla e^{-V}\big)=0,
	\]
	i.e. that the matrix $(\nu J_{2}+\Gamma)^{-1}\Gamma(-\nu J_{2}+\Gamma)^{-1}-(\nu J_{2}+\Gamma)^{-1}$
	is antisymmetric. Clearly, the first term is symmetric and furthermore
	it turns out to be equal to the symmetric part of the second term:
		\begin{eqnarray*}
 \frac{1}{2}\big((\nu J_{2}+\Gamma)^{-1}+(-\nu J_{2}+\Gamma)^{-1}\big) & = &
	  =\frac{1}{2}\big((\nu J_{2}+\Gamma)^{-1}(-\nu J_{2}+\Gamma)(-\nu J_{2}+\Gamma)^{-1}  \\ && + (\nu J_{2}+\Gamma)^{-1}(\nu J_{2}+\Gamma)(-\nu J_{2}+\Gamma)^{-1}\big)\\
	& = & (\nu J_{2}+\Gamma)^{-1}\Gamma(-\nu J_{2}+\Gamma)^{-1},
		\end{eqnarray*}
	so $\pi$ is indeed invariant under the limiting dynamics (\ref{eq:overdamped limit}).
\end{remark}

%% file: gaussian.tex
In this section we study in detail the performance of the Langevin sampler~\eqref{eq:perturbed_underdamped} for Gaussian target densities, first considering the case of unit covariance. In particular, we study the optimal choice for the parameters in the sampler, the exponential decay rate and the asymptotic variance. We then extend our results to Gaussian target densities with arbitrary covariance matrices.

\subsection{Unit covariance - small perturbations}
\label{sec:small perturbations}

In our study of the dynamics given by \eqref{eq:perturbed_underdamped}
we first consider the simple case when $V(q)=\frac{1}{2}\vert q\vert^{2}$,
i.e. the task of sampling from a Gaussian measure with unit covariance.
We will assume $M=I$, $\Gamma=\gamma I$ and $J_{1}=J_{2}=:J$
(so that the $q-$ and $p-$dynamics are perturbed in the same way,
albeit posssibly with different strengths $\mu$ and $\nu$). Using
these simplifications, (\ref{eq:perturbed_underdamped})
reduces to the linear system 
\begin{align}
\mathrm{d}q_{t} & =p_{t}\mathrm{d}t-\mu Jq_{t}\mathrm{d}t\nonumber, \\
\mathrm{d}p_{t} & =-q_{t}\mathrm{d}t-\nu Jp_{t}\mathrm{d}t-\gamma p_{t}\mathrm{d}t+\sqrt{2\gamma}\mathrm{d}W_{t}.\label{eq:unit covariance}
\end{align}
The above dynamics are of Ornstein-Uhlenbeck type, i.e. we can write
\begin{equation}
\mathrm{d}X_{t}=-BX_{t}\mathrm{d}t+\sqrt{2Q}\mathrm{d}\bar{W}_{t}\label{eq:OU process}
\end{equation}
with $X=(q,p)^{T}$, 
\begin{equation}
B=\left(\begin{array}{cc}
\mu J & -I\\
I & \gamma I+\nu J
\end{array}\right),\label{eq:drift matrix}
\end{equation}
\begin{equation}
Q=\left(\begin{array}{cc}
\boldsymbol{0} & \boldsymbol{0}\\
\boldsymbol{0} & \gamma I
\end{array}\right)\label{eq:diffusion matrix}
\end{equation}
and $(\bar{W}_{t})_{t\ge0}$ denoting a standard Wiener process on
$\mathbb{R}^{2d}$. The generator of (\ref{eq:OU process}) is then
given by 
\begin{equation}
\mathcal{L}=-Bx\cdot\nabla+\nabla^{T}Q\nabla.\label{eq:OU generator}
\end{equation}
We will consider quadratic observables of the form 
\[
f(q)=q\cdot Kq+l\cdot q+C,
\]
with $K\in\mathbb{R}_{sym}^{d\times d}$, $l\in\mathbb{R}^{d}$ and
$C\in\mathbb{R}$, however it is worth recalling that the asymptotic variance $\sigma^2_f$ does not depend on $C$. We also stress that $f$ is assumed to be independent of
$p$ as those extra degrees of freedom are merely auxiliary. Our
aim will be to study the associated asymptotic variance $\sigma_{f}^{2}$, see equation (\ref{eq:asymptoticvariance}), in particular its
dependence on the parameters $\mu$ and $\nu$.  This dependence is encoded in the
function 
\begin{alignat*}{1}
\Theta:\quad\mathbb{R}^{2} & \rightarrow\mathbb{R}\\
(\mu,\nu) & \mapsto\sigma_{f}^{2},
\end{alignat*}
assuming a fixed observable $f$ and perturbation matrix $J$. In
this section we will focus on small perturbations, i.e. on the behaviour
of the function $\Theta$ in the neighbourhood of the origin. Our
main theoretical tool will be the Poisson equation \eqref{eq:poisson_general}, see the proofs in Appendix \ref{app:Gaussian_proofs}. Anticipating the forthcoming analysis, let us already state our main result, showing that in the neighbourhood of the origin, the function $\Theta$ has favourable properties along the diagonal $\mu=\nu$ (note that the perturbation strengths in the first and second line of \eqref{eq: unit covariance-1-1} coincide):

\begin{theorem}
	\label{cor:small pert unit var}Consider the dynamics 
\begin{align}
\mathrm{d}q_{t} & =p_{t}\mathrm{d}t-\mu Jq_{t}\mathrm{d}t,\nonumber \\
\mathrm{d}p_{t} & =-q_{t}\mathrm{d}t-\mu Jp_{t}\mathrm{d}t-\gamma p_{t}\mathrm{d}t+\sqrt{2\gamma}\mathrm{d}W_{t},\label{eq: unit covariance-1-1}
\end{align}
with $\gamma>\sqrt{2}$ and an observable of the form $f(q)=q\cdot Kq+l\cdot q+C$.
If at least one of the conditions $[J,K]\neq0$ and $l\notin\ker J$
is satisfied, then the asymptotic variance of the unperturbed sampler
is at a local maximum independently of $K$ and $J$ (and $\gamma$,
as long as $\gamma>\sqrt{2}$), i.e. 
\[
\left. \partial_{\mu}\sigma_{f}^{2} \right\rvert_{\mu =0}=0
\]
and 
\[
\left. \partial_{\mu}^{2}\sigma_{f}^{2}\right\rvert_{\mu = 0}<0.
\]
\end{theorem}

\subsubsection{\label{sub:Purely-quadratic-observables}Purely quadratic observables}

Let us start with the case $l=0$, i.e. $f(q)=q\cdot Kq+C$. The following
holds: 
\begin{proposition}
	\label{thm: local quadratic observable}The function $\Theta$ satisfies
	\begin{equation}
	\left. \nabla\Theta\right\rvert_{(\mu,\nu)=(0,0)}=0\label{eq:gradTheta}
	\end{equation}
	and 
	\begin{equation}
	\left. \Hess\Theta\right\rvert_{(\mu,\nu)=(0,0)}=\left(\begin{array}{cc}
	-(\gamma+\frac{1}{\gamma^{3}}+\gamma^{3})\left(\Tr(JKJK)-\Tr(J^{2}K^{2})\right) & (\frac{1}{\gamma^{3}}+\frac{1}{\gamma}-\gamma)\Tr(J^{2}K^{2})\\
	-\frac{2}{\gamma}\Tr(JKJK) & +(-\frac{1}{\gamma^{3}}+\frac{1}{\gamma}+\gamma)\Tr(JKJK)\\
	(\frac{1}{\gamma^{3}}+\frac{1}{\gamma}-\gamma)\Tr(J^{2}K^{2}) & (\frac{1}{\gamma^{3}}-\frac{1}{\gamma})\Tr(J^{2}K^{2})\\
	+(-\frac{1}{\gamma^{3}}+\frac{1}{\gamma}+\gamma)\Tr(JKJK) & -(\frac{1}{\gamma^{3}}+\frac{1}{\gamma})\Tr(JKJK)
	\end{array}\right).\label{eq:HessTheta}
	\end{equation}
\end{proposition}
\begin{proof}
	See Appendix \ref{app:Gaussian_proofs}.
	\qed
\end{proof}
The above proposition shows that the unperturbed dynamics represents a
critical point of $\Theta$, independently of the choice of $K$,
$J$ and $\gamma$. In general though, $\Hess\Theta\vert_{(\mu,\nu)=(0,0)}$
can have both positive and negative eigenvalues. In particular, this
implies that an unfortunate choice of the perturbations will actually
increase the asymptotic variance of the dynamics (in contrast to the situation
of perturbed \emph{overdamped }Langevin dynamics, where any nonreversible
perturbation leads to an improvement in asymptotic variance as detailed
in \cite{asvar_Hwang} and \cite{duncan2016variance}). Furthermore, the nondiagonality
of $\Hess\Theta\vert_{(\mu,\nu)=(0,0)}$ hints at the fact that the
interplay of the perturbations $J_{1}$ and $J_{2}$ (or rather their
relative strengths $\mu$ and $\nu$) is crucial for the performance
of the sampler and, consequently, the effect of these perturbations cannot be satisfactorily
studied independently. 
\begin{example}
	Assuming $J^{2}=-I$ and $[J,K]=0$ it follows that
	\[
	\left. \partial_{\mu}^{2}\Theta\right\rvert_{\mu=0}=\left. \partial_{\nu}^{2}\Theta\right\rvert_{\mu=0}=\frac{1}{\gamma}\Tr(K^{2})>0,
	\]
	for all nonzero $K$. Therefore in this case, a small perturbation of $J_{1}$ only or $J_{2}$ only will increase  the asymptotic variance, uniformly over all choices of $K$ and $\gamma$.
\end{example}
However, it turns out that it is possible to construct an improved sampler
by combining both perturbations in a suitable way. Indeed,
the function $\Theta$ can be seen to have good properties along $\mu=\nu$. We set $\mu(s)=s$, $\nu(s):=s$ and compute

\begin{align*}
\left. \frac{\mathrm{d}^{2}}{\mathrm{d}s^{2}}\Theta\right\rvert_{s=0} & =(1,1)\cdot\Hess\Theta\vert_{(\mu,\nu)=(0,0)}(1,1)\\
& =-(\gamma+\frac{1}{\gamma^{3}}+\gamma^{3})\left(\Tr(JKJK)-\Tr(J^{2}K^{2})\right)-\frac{2}{\gamma}\Tr(JKJK)\\
& +2\cdot\left((\frac{1}{\gamma^{3}}+\frac{1}{\gamma}-\gamma)\Tr(J^{2}K^{2})+(-\frac{1}{\gamma^{3}}+\frac{1}{\gamma}+\gamma)\Tr(JKJK)\right)\\
& +(\frac{1}{\gamma^{3}}-\frac{1}{\gamma})\Tr(J^{2}K^{2})-(\frac{1}{\gamma^{3}}+\frac{1}{\gamma})\Tr(JKJK)\\
& =\big(\gamma-\frac{4}{\gamma^{3}}-\gamma^{3}-\frac{1}{\gamma}\big)\cdot(\Tr(JKJK)-\Tr(J^{2}K^{2}))\le0.
\end{align*}
The last inequality follows from 
\[
\gamma-\frac{4}{\gamma^{3}}-\gamma^{3}-\frac{1}{\gamma}<0
\]
and 
\[
\Tr(JKJK)-\Tr(J^{2}K^{2})\ge0
\]
(both inequalities are proven in the Appendix, Lemma \ref{lem:basic_inequalities}), where the last inequality
is strict if $[J,K]\neq0$. Consequently, choosing both perturbations
to be of the same magnitude ($\mu=\nu$) and assuring that $J$ and
$K$ do not commute always leads to a smaller asymptotic variance,
independently of the choice of $K$, $J$ and $\gamma$. We state
this result in the following corrolary:
\begin{corollary}
	\label{cor:unit covariance quadratic obs}Consider the dynamics 
	\begin{align}
	\mathrm{d}q_{t} & =p_{t}\mathrm{d}t-\mu Jq_{t}\mathrm{d}t\nonumber, \\
	\mathrm{d}p_{t} & =-q_{t}\mathrm{d}t-\mu Jp_{t}\mathrm{d}t-\gamma p_{t}\mathrm{d}t+\sqrt{2\gamma}\mathrm{d}W_{t},\label{eq: unit covariance-1}
	\end{align}
 and a quadratic observable $f(q)=q\cdot Kq+C$. If $[J,K] \neq 0,$
	then the asymptotic variance of the unperturbed sampler is at a local
	maximum independently of K, $J$ and $\gamma$, i.e. 
	\[
	\left. \partial_{\mu}\sigma_{f}^{2}\right\rvert_{\mu=0}=0
	\]
	and 
	\[
	\left. \partial_{\mu}^{2}\sigma_{f}^{2}\right\rvert_{\mu=0}<0.
	\]
\end{corollary}
\begin{remark}
	As we will see in Section \ref{sec:large perturbations}, more precisely Example \ref{ex:commutation quadratic observables}, if $[J,K]=0,$
	the asymptotic variance is constant as a function of $\mu$, i.e.
	the perturbation has no effect.\end{remark}
\begin{example}
\label{ex:opposed perturbation}
	Let us set $\mu(s):=s$ and $\nu(s):=-s$ (this corresponds to a small
	perturbation with $J\nabla V(q_{t})\mathrm{d}t$ in $q$ and $-Jp_{t}\mathrm{d}t$
	in $p$). In this case we get 
	\[
	\left. \frac{\mathrm{d}^{2}\Theta}{\mathrm{d}s^{2}} \right\rvert_{s=0}=\underbrace{-\frac{1}{2}\cdot\frac{\gamma^{4}+3\gamma^{2}+5}{\gamma}\left(\Tr(JKJK)-\Tr(J^{2}K^{2})\right)}_{\le0}\underbrace{-4\frac{\Tr(J^{2}K^{2})}{\gamma}}_{\ge0},
	\]
	which changes its sign depending on $J$ and $K$ as the first term
	is negative and the second is positive. Whether the perturbation improves the performance of the sampler in terms of asymptotic variance therefore depends on the specifics of the observable and the perturbation in this case.
	
\end{example}

\subsubsection{Linear observables}

Here we consider the case $K=0$, i.e. $f(q)=l\cdot q+C$, where
again $l\in\mathbb{R}^{d}$ and $C\in\mathbb{R}$. We have the following
result: 
\begin{proposition}
	\label{thm:linear_full_J}The function $\Theta$ satisfies 
	\[
	\nabla\Theta\vert_{(\mu,\nu)=(0,0)}=0
	\]
	and 
	\[
	\Hess\Theta\vert_{(\mu,\nu)=(0,0)}=\left(\begin{array}{cc}
	-2\gamma^{3}\vert Jl\vert^{2} & 2\gamma\vert Jl\vert^{2}\\
	2\gamma\vert Jl\vert^{2} & 0
	\end{array}\right).
	\]
\end{proposition}
\begin{proof}
	See Appendix \ref{app:Gaussian_proofs}.
	\qed
\end{proof}
	Let us assume that $l\notin\ker J$. Then $\partial_{\mu}^{2}\Theta\vert_{\mu,\nu=0}<0$,
	and hence Theorem \ref{thm:linear_full_J} shows that a small perturbation
	by $\mu J\nabla V(q_{t})\mathrm{d}t$ alone always results in an improvement
	of the asymptotic variance. However, if we combine both perturbations
	$\mu J\nabla V(q_{t})\mathrm{d}t$ and $\nu Jp_{t}\mathrm{d}t$, then
	the effect depends on the sign of 
	\[
	\left(\begin{array}{cc}
	\mu & \nu\end{array}\right)\left(\begin{array}{cc}
	-2\gamma^{3}\vert Jl\vert^{2} & 2\gamma\vert Jl\vert^{2}\\
	2\gamma\vert Jl\vert^{2} & 0
	\end{array}\right)\left(\begin{array}{c}
	\mu\\
	\nu
	\end{array}\right)=-(2\mu^{2}\gamma^{3}-4\mu\nu\gamma)\vert Jl\vert^{2}.
	\]
	This will be negative if $\mu$ and $\nu$ have different signs, and
	also if they have the same sign and $\gamma$ is big enough.
Following Section \ref{sub:Purely-quadratic-observables}, we require
$\mu=\nu$. We then end up with the requirement 
\[
2\mu^{2}\gamma^{3}-4\mu\nu\gamma>0,
\]
which is satisfied if $\gamma>\sqrt{2}$

Summarizing the results of this section, for observables of the form
$f(q)=q\cdot Kq+l\cdot q+C$, choosing equal perturbations ($\mu=\nu$)
with a sufficiently strong damping $(\gamma>\sqrt{2}$) always leads
to an improvement in asymptotic variance under the conditions $[J,K]\neq0$
and $l\notin\ker J$. This is finally the content of Theorem \ref{cor:small pert unit var}.
\begin{figure}
	\begin{subfigure}[b]{0.5 \textwidth}
		\includegraphics[width=\textwidth]{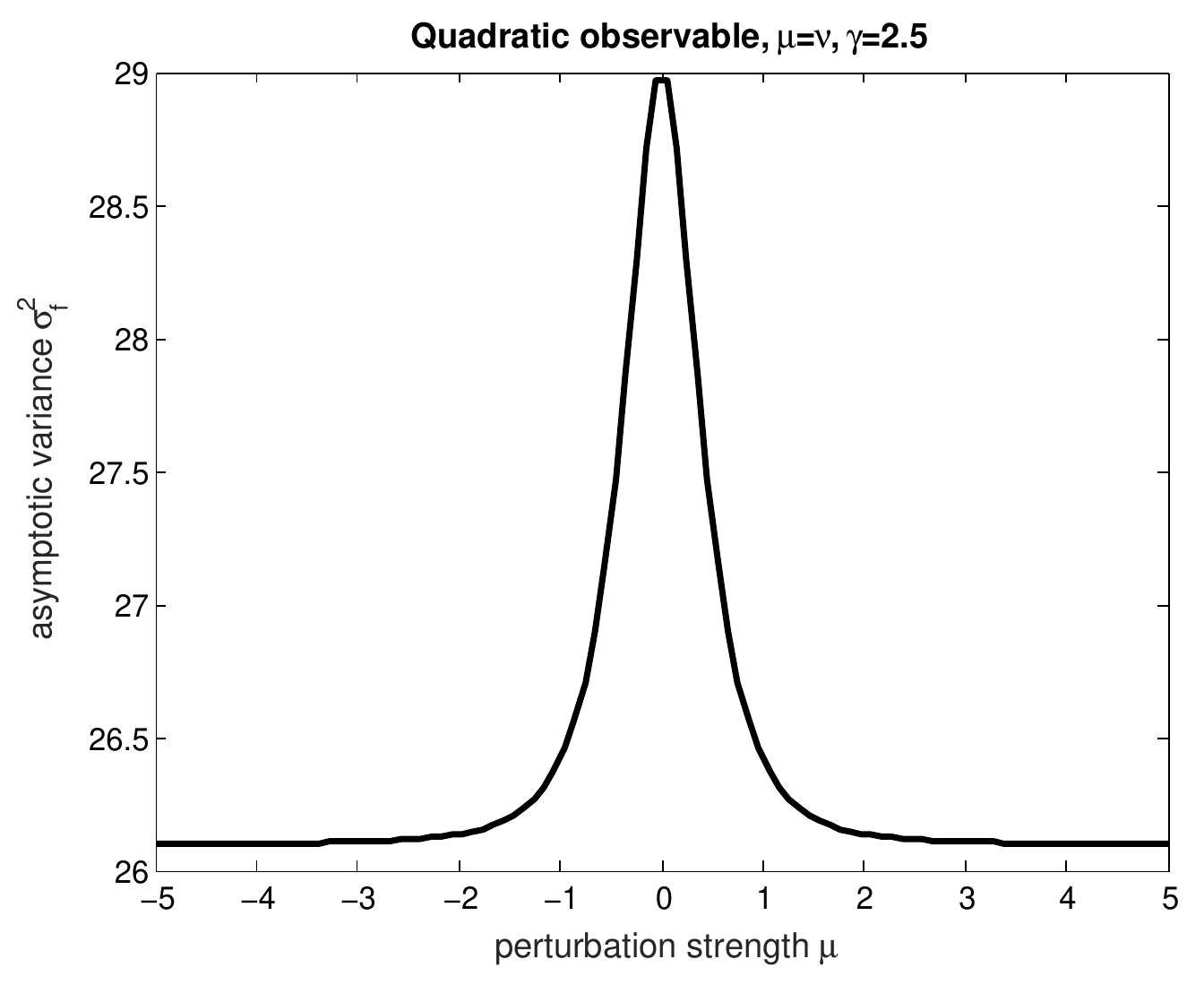}
		\caption{Equal perturbations: $\mu=\nu$}
		\label{fig:asym_quad}
	\end{subfigure}
	\hfill
	\begin{subfigure}[b]{0.5 \textwidth}
		\includegraphics[width=\textwidth]{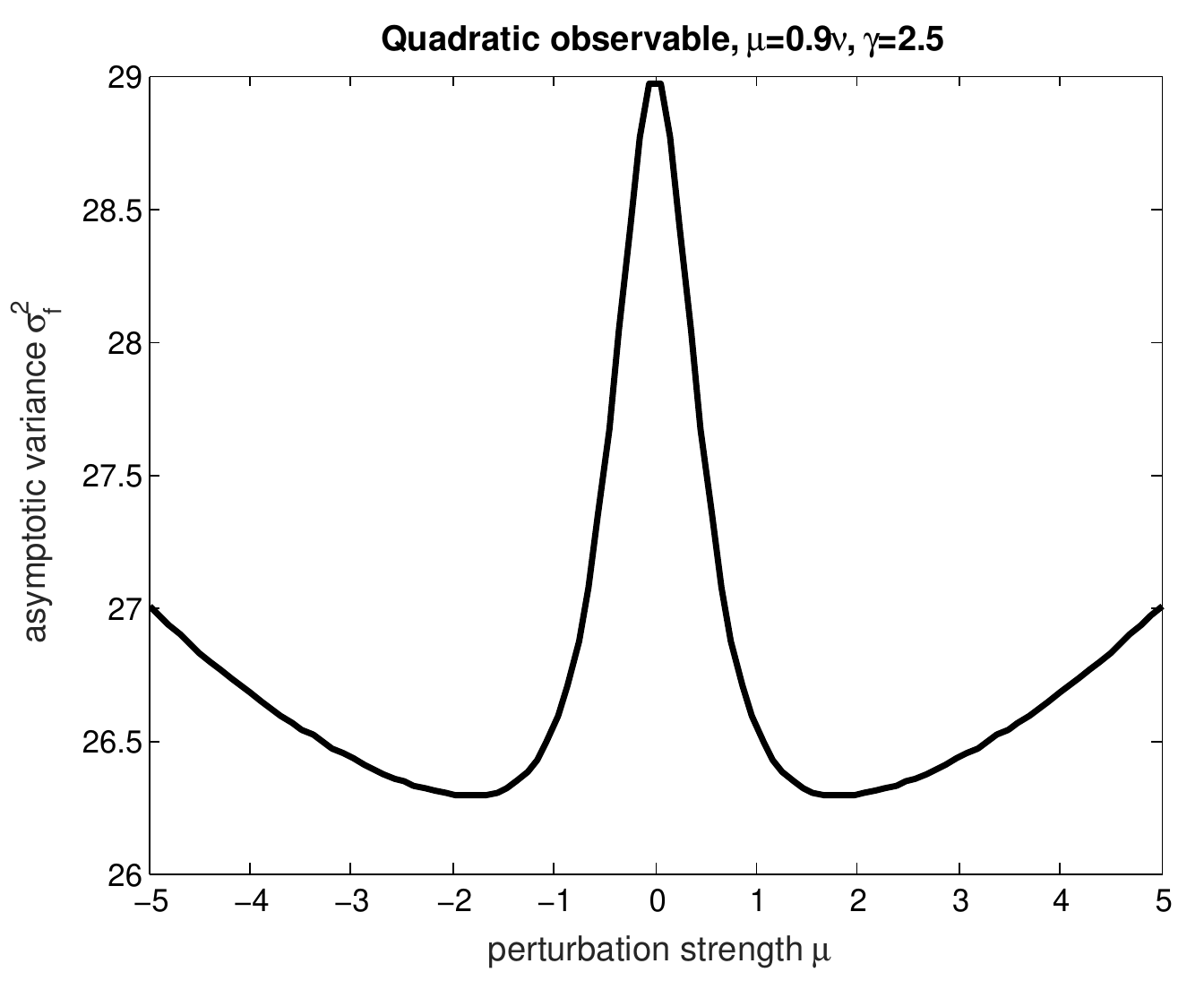}
		\caption{Approximately equal perturbations: $\mu=0.9\nu$}
		\label{fig:no_limit1}
	\end{subfigure}
	\hfill
	\begin{subfigure}[b]{0.5 \textwidth}
		\includegraphics[width=\textwidth]{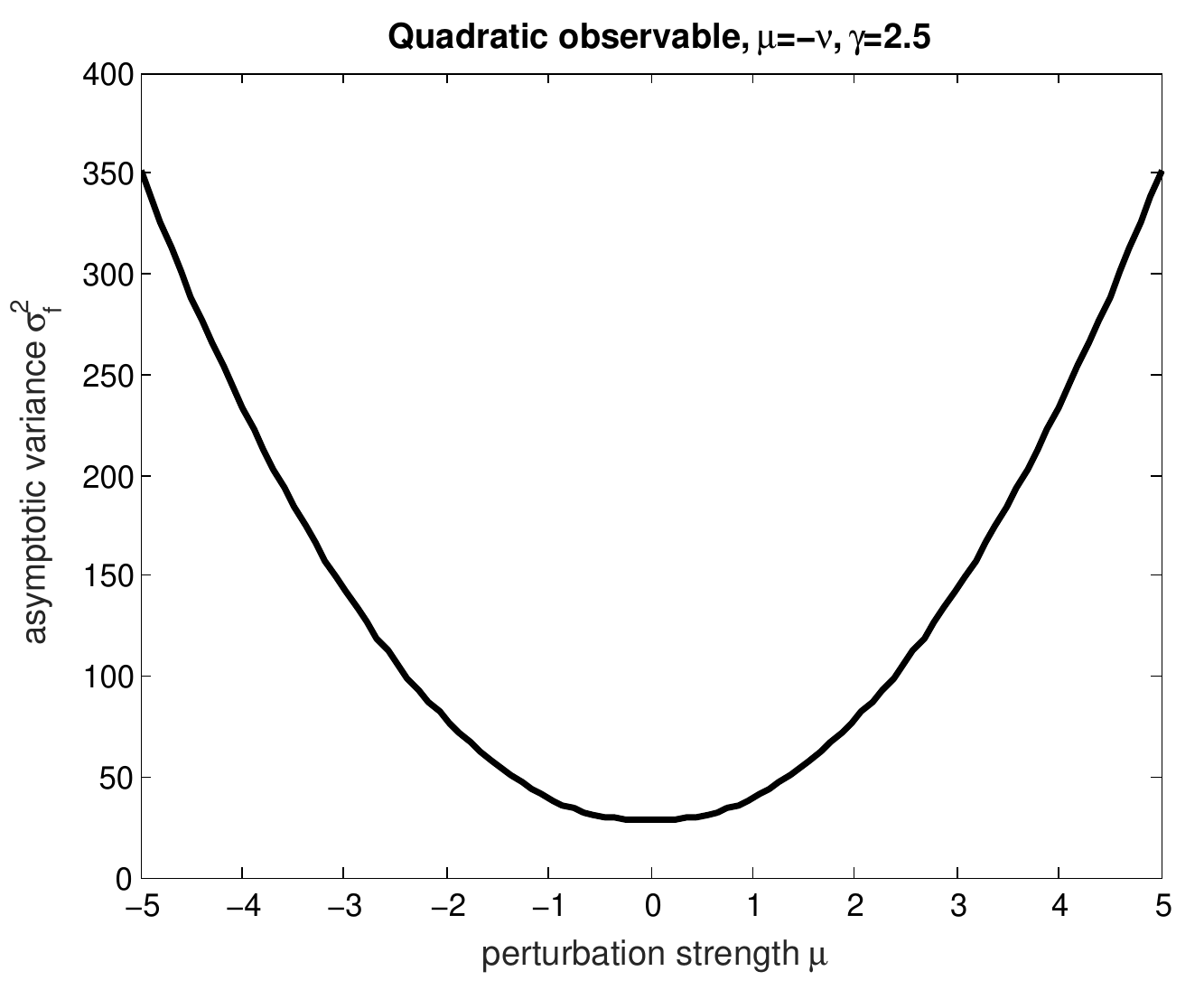}
		\caption{Opposing perturbations: $\mu=-\nu$ 
			\label{fig:no_limit2}}
	\end{subfigure}
	\hfill
	\begin{subfigure}[b]{0.5 \textwidth}
		\includegraphics[width=\textwidth]{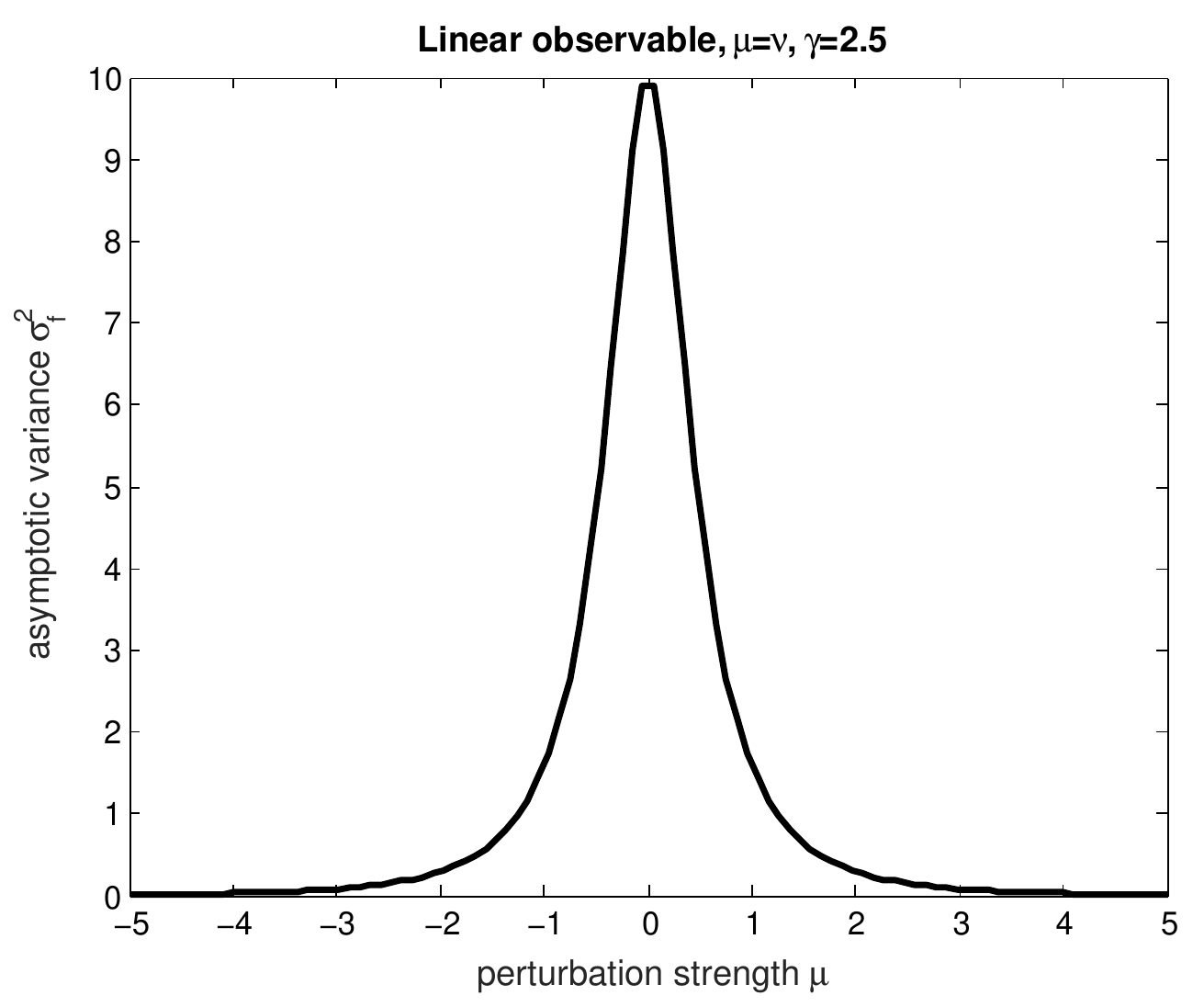}
		\caption{Equal perturbations: $\mu=\nu$ (sufficiently large friction $\gamma$)}
		\label{fig:lin_large_friction}
	\end{subfigure}
	\hfill
	\begin{subfigure}[b]{0.5 \textwidth}
		\includegraphics[width=\textwidth]{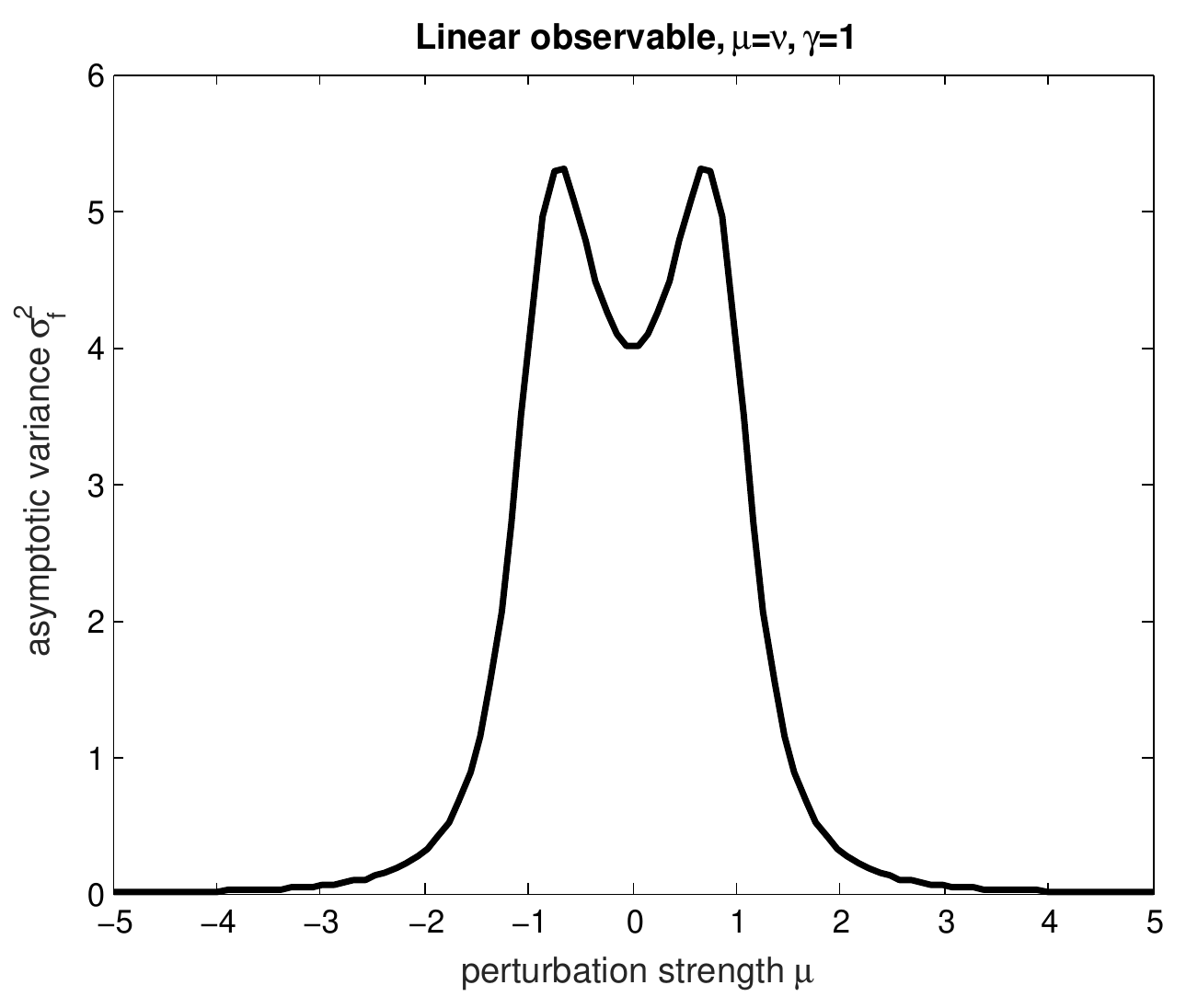}
		
		\caption{Equal perturbations: $\mu=\nu$ (small friction $\gamma$)}
		\label{fig:lin_small_friction}
	\end{subfigure}
	\caption{Asymptotic variance for linear and quadratic observables, depending on relative perturbation and friction strengths}
	\label{fig:linear and quadratic observables}
\end{figure}

Let us illustrate the results of this section by plotting the asymptotic variance as a function of the perturbation strength $\mu$ (see Figure \ref{fig:linear and quadratic observables}), making the choices $d=2$, $l=(1,1)^{T}$,
\begin{equation}
K=\left(\begin{array}{cc}
2 & 0\\
0 & 1
\end{array}\right)
\quad \text{and} \quad
J=\left(\begin{array}{cc}
0 & 1\\
-1 & 0
\end{array}\right).
\end{equation}
The asymptotic variance has been computed according to \eqref{eq:Gaussian asymvar}, using \eqref{eq:Lyapunov equation} and \eqref{eq:linear condition} from Appendix \ref{app:Gaussian_proofs}. The graphs confirm the results summarized in Corollary \ref{cor:small pert unit var} concerning the asymptotic variance in the neighbourhood of the unperturbed dynamics ($\mu = 0$). Additionally, they give an impression of the global behaviour, i.e. for larger values of $\mu$.

Figures \ref{fig:asym_quad}, \ref{fig:no_limit1} and \ref{fig:no_limit2}  show the asymptotic variance associated with the quadratic observable $f(q)=q\cdot K q$. In accordance with Corollary \ref{cor:unit covariance quadratic obs}, the asymptotic variance is at a  local maximum at zero perturbation in the case $\mu=\nu$ (see Figure \ref{fig:asym_quad}). For increasing perturbation strength, the graph shows that it decays monotonically
and reaches a limit for $\mu\rightarrow\infty$ (this limiting behaviour will be explored analytically in Section \ref{sec:large perturbations}). If the condition $\mu=\nu$ is only approximately satisfied (Figure \ref{fig:no_limit1}), our numerical examples still exhibits decaying asymptotic variance in the neighbourhood of the critical point. In this case, however, the asymptotic variance diverges for growing values of the perturbation $\mu$. If the perturbations are opposed ($\mu=-\nu$) as in Example \ref{ex:opposed perturbation}, it is possible for certain observables that the unperturbed dynamics represents a global minimum. Such a case is observed in Figure \ref{fig:no_limit2}. In Figures \ref{fig:lin_large_friction} and \ref{fig:lin_small_friction} the observable $f(q)=l\cdot q$ is considered. If the damping is sufficiently strong ($\gamma > \sqrt{2}$), the unperturbed dynamics is at a local maximum of the asymptotic variance (Figure \ref{fig:lin_large_friction}). Furthermore, the asymptotic variance approaches zero as $\mu \rightarrow \infty$ (for a theoretical explanation see again Section \ref{sec:large perturbations}). The graph in Figure \ref{fig:lin_small_friction} shows that the assumption of $\gamma$ not being too small cannot be dropped from Corollary \ref{cor:small pert unit var}. Even in this case though the example shows decay of the asymptotic variance for large values of $\mu$.    
\subsection{Exponential decay rate}
\label{sec:exp_decay}
Let us denote by $\lambda^{*}$ the \emph{optimal exponential decay rate} in \eqref{eq:hypocoercive estimate}, i.e.
\begin{equation}
\lambda^{*}=\sup\{\lambda > 0 \, \vert \, \text{There exists } C\ge 1 \text{ such that } \eqref{eq:hypocoercive estimate} \text{ holds}\}.
\end{equation}
Note that $\lambda^{*}$ is well-defined and positive by Theorem \ref{theorem:Hypocoercivity}. We also define the \emph{spectral bound} of the generator $\gen$ by
\begin{equation}
s(\gen)=\inf(\text{Re}\,\sigma(-\gen)\setminus\{0\}).
\end{equation} 
In \cite{Metafune_formula} it is proven that the Ornstein-Uhlenbeck semigroup $(P_t)_{t\ge0}$ considered in this section is differentiable (see Proposition 2.1). In this case (see Corollary 3.12 of \cite{Engel2000Semigroup}), it is known that the exponential decay rate and the spectral bound coincide, i.e. $\lambda^{*}=s(\gen)$, whereas in general only $\lambda^{*}\le s(\gen)$ holds.
In this section we will therefore analyse the spectral properties of the generator
(\ref{eq:OU generator}). In particular, this leads to some intuition
of why choosing equal perturbations ($\mu=\nu$) is crucial for the
performance of the sampler.

In \cite{Metafune_formula} (see also \cite{OPP12}), it was proven that
the spectrum of $\mathcal{L}$ as in (\ref{eq:OU generator}) in $L^{2}(\widehat{\pi})$
is given by 
\begin{equation}
\sigma(\mathcal{L})=\left\{-\sum_{j=1}^{r}n_{j}\lambda_{j}:\, n_{j}\in\mathbb{N},\lambda_{j}\in \sigma(B)\right\}.\label{eq:Metafune formula}
\end{equation}
Note that $\sigma(\mathcal{L})$ only depends on the drift matrix
$B$. In the case where $\mu=\nu$,
the spectrum of $B$ can be computed explicitly. 
\begin{lemma}
	\label{lem:drift matrix properties}Assume $\mu=\nu$. Then the spectrum
	of $B$ is given by
	\begin{equation}
	\sigma(B)=\left\{\mu\lambda+\sqrt{\big(\frac{\gamma}{2}\big)^{2}-1}+\frac{\gamma}{2}\vert\lambda\in\sigma(J)\}\cup\{\mu\lambda-\sqrt{\big(\frac{\gamma}{2}\big)^{2}-1}+\frac{\gamma}{2}\vert\lambda\in\sigma(J)\right\}.\label{eq:spectrum of B}
	\end{equation}
\end{lemma}
\begin{proof}
	We will compute $\sigma\big(B-\frac{\gamma}{2}I\big)$ and then use
	the identity
	\begin{equation}
	\sigma(B)=\left\{\lambda+\frac{\gamma}{2}\vert\lambda\in\sigma\left(B-\frac{\gamma}{2}I\right)\right\}.\label{eq:shift spectrum}
	\end{equation}
	We have 
	\begin{align*}
	\det\left(B-\frac{\gamma}{2}I-\lambda I\right) & =\det\left(\left(\mu J-\frac{\gamma}{2}I-\lambda I\right)\left(\mu J+\frac{\gamma}{2}I-\lambda I\right)+I\right)\\
	& =\det\left((\mu J-\lambda I)^{2}-\left(\frac{\gamma}{2}\right)^{2}I+I\right)\\
	& =\det\left(\left(\mu J-\lambda I+\sqrt{\left(\frac{\gamma}{2}\right)^{2}-1} I\right)\cdot\left(\mu J-\lambda I-\sqrt{\left(\frac{\gamma}{2} \right)^{2}-1} I\right)\right)\\
	& =\det\left(\mu J-\lambda I+\sqrt{\left(\frac{\gamma}{2}\right)^{2}-1} I\right)\cdot\det\left(\mu J-\lambda I-\sqrt{\left(\frac{\gamma}{2}\right)^{2}-1} I\right),
	\end{align*}
	where $I$ is understood to denote the identity matrix of appropriate dimension.
	The above quantity is zero if and only if 
	\[
	\lambda-\sqrt{\left(\frac{\gamma}{2}\right)^{2}-1}\in\sigma(\mu J)
	\]
	or 
	\[
	\lambda+\sqrt{\left(\frac{\gamma}{2}\right)^{2}-1}\in\sigma(\mu J).
	\]
	Together with (\ref{eq:shift spectrum}), the claim follows.
	\qed
\end{proof}
Using formula \eqref{eq:Metafune formula}, in Figure \ref{fig:good_spectrum} we show a sketch of the spectrum $\sigma(-\mathcal{L}$)
for the case of equal perturbations ($\mu=\nu)$ with the convenient
choices $n=1$ and $\gamma=2.$ Of course, the eigenvalue at $0$ is
associated to the invariant measure since $\sigma(-\mathcal{L})=\sigma(-\mathcal{L}^{\dagger})$
and $\mathcal{L}^{\dagger}\widehat{\pi}=0$, where $\mathcal{L}^{\dagger}$ denotes the Fokker-Planck operator, i.e. the $L^2(\mathbb{R}^{2d})$-adjoint of $\mathcal{L}$. The arrows indicate the movement
of the eigenvalues as the perturbation $\mu$ increases in accordance
with Lemma \ref{lem:drift matrix properties}. Clearly, the spectral
bound of $\gen$ is not affected by the perturbation. 
Note that the eigenvalues on the real axis stay invariant under the
perturbation. The subspace of $L_{0}^{2}(\widehat{\pi})$ associated to
those will turn out to be crucial for the characterisation of the
limiting asymptotic variance as $\mu\rightarrow\infty$.

To illustrate the suboptimal properties of the perturbed dynamics
when the perturbations are not equal, we plot the spectrum of the
drift matrix $\sigma(B)$ in the case when the dynamics is only perturbed
by the term $\nu J_{2}p\mathrm{d}t$ (i.e. $\mu=0$) for $n=2$, $\gamma=2$ and
\begin{equation}
J_2=\left(\begin{array}{cc}
0 & -1\\
1 & 0
\end{array}\right),
\end{equation} 
(see Figure \ref{fig:bad_spectrum}). Note that the full spectrum $\sigma(-\mathcal{L})$
can be inferred from (\ref{eq:Metafune formula}). For $\nu=0$ we have that the spectrum $\sigma(B)$
only consists of the (degenerate) eigenvalue $1$. For increasing
$\nu$, the figure shows that the degenerate eigenvalue splits up
into four eigenvalues, two of which get closer to the imaginary axis as $\nu$ increases, leading to a smaller spectral
bound and therefore to a decrease in the speed of convergence to equilibrium.
Figures (\ref{fig:good_spectrum}) and (\ref{fig:bad_spectrum}) give an intuitive explanation
of why the fine-tuning of the perturbation strengths is crucial.

\begin{figure}
	\begin{subfigure}[b]{0.45 \textwidth}
		\includegraphics[width=\textwidth]{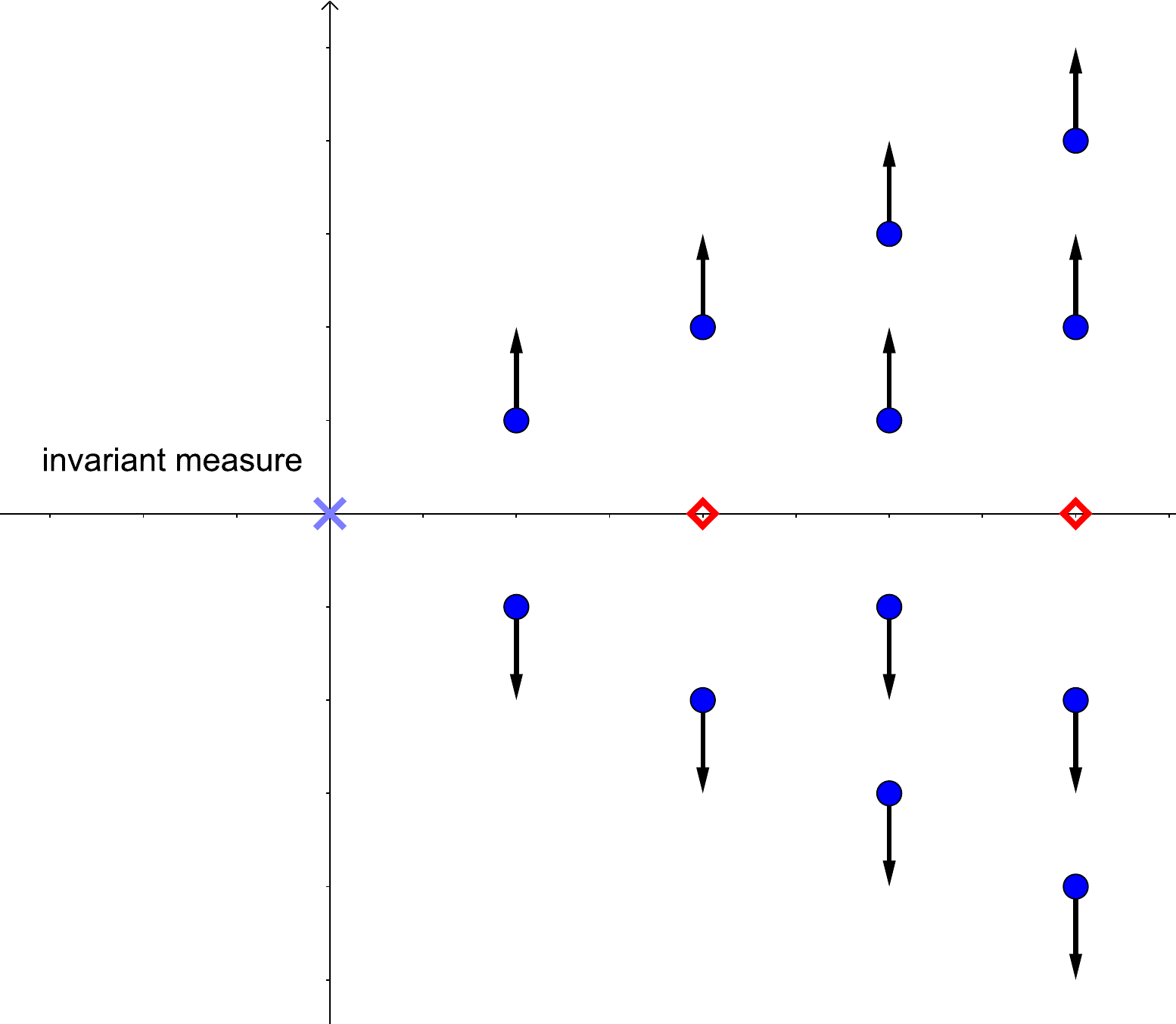}
		\caption{$\sigma(-\gen)$ in the case $\mu=\nu$. The arrows indicate the movement of the spectrum as the perturbation strength $\mu$ increases.\label{fig:good_spectrum}}
	\end{subfigure}
	\hfill
	\begin{subfigure}[b]{0.45 \textwidth}
		\includegraphics[width=\textwidth]{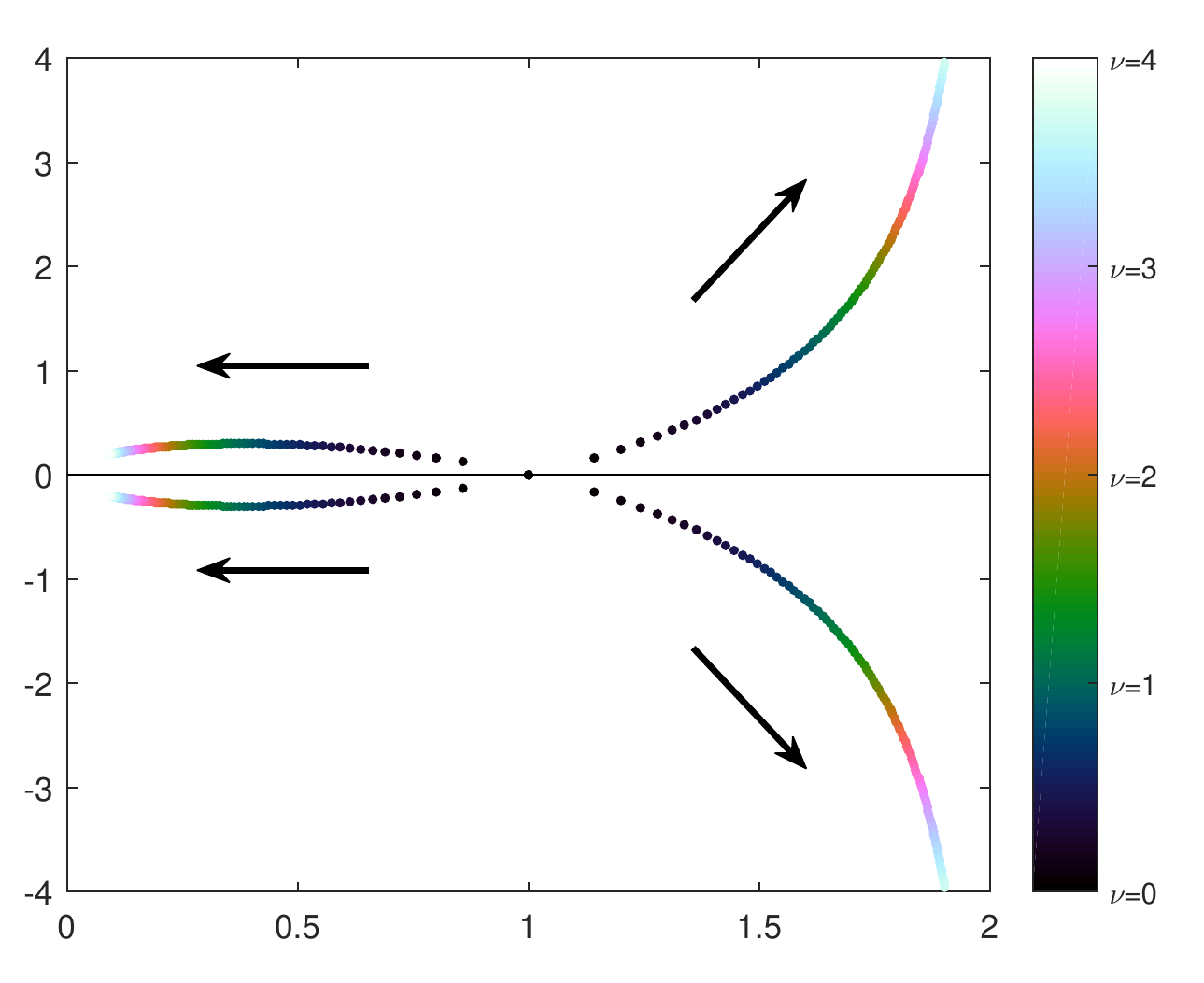}
		\caption{$\sigma(B)$ in the case $J_{1}=0$, i.e. the dynamics is only perturbed
			by $-\nu J_{2}p\mathrm{d}t$. The arrows indicate the movement of
			the eigenvalues as $\nu$ increases.\label{fig:bad_spectrum}}
	\end{subfigure}
	.	\caption{Effects of the perturbation on the spectra of $-\gen$ and $B$.}
	\label{fig:examples-introduction}
\end{figure}

\subsection{Unit covariance - large perturbations}
\label{sec:large perturbations}

In the previous subsection we observed that for the particular perturbation $J_1 = J_2$ and $\mu = \nu$, i.e. 
\begin{align}
\mathrm{d}q_{t} & =p_{t}\mathrm{d}t-\mu Jq_{t}\mathrm{d}t\nonumber \\
\mathrm{d}p_{t} & =-q_{t}\mathrm{d}t-\mu Jp_{t}\mathrm{d}t-\gamma p_{t}\mathrm{d}t+\sqrt{2\gamma}\,\mathrm{d}W_{t},\label{eq: unit covariance perfect perturbation}
\end{align}
the perturbed Langevin dynamics demonstrated an improvement in performance for $\mu$ in a neighbourhood of $0$, when the observable is linear or quadratic.  Recall that this dynamics is ergodic with respect to a standard Gaussian measure $\widehat{\pi}$ on $\mathbb{R}^{2d}$ with marginal $\pi$  with respect to the $q$--variable.  In the following we shall consider only observables that do not depend on $p$. Moreover, we assume without loss of generality that $\pi(f)=0$. For such an observable we will write $f\in L^2_0(\pi)$ and assume the canonical embedding $L^2_0(\pi)\subset L^2(\widehat{\pi})$.  The infinitesimal generator of (\ref{eq: unit covariance perfect perturbation})
is given by 
\begin{equation}
\label{eq:generator_equal}
\mathcal{L}=\underbrace{p\cdot\nabla_{q}-q\cdot\nabla_{p}+\gamma(-p\cdot\nabla_{p}+\Delta_{p})}_{\mathcal{L}_{0}}+\mu\underbrace{(-Jq\cdot\nabla_{q}-Jp\cdot\nabla_{p})}_{\mathcal{A}}=:\mathcal{L}_{0}+\mu\mathcal{A},
\end{equation}
where we have introduced the notation $\mathcal{L}_{pert}=\mu \mathcal{A}$. In the sequel, the adjoint of an operator $B$ in $L^2(\widehat{\pi})$ will be denoted by $B^{*}$. In the rest of this section we will make repeated use of the Hermite polynomials
\begin{equation}
g_{\alpha}(x)=(-1)^{\vert\alpha\vert}e^{\frac{\vert x\vert^{2}}{2}}\nabla^{\alpha}e^{-\frac{\vert x\vert^{2}}{2}},\quad\alpha\in\mathbb{N}^{2d},\label{eq: Hermite polynomials}
\end{equation}
invoking the notation $x=(q,p)\in\mathbb{R}^{2d}$. For $m\in\mathbb{N}_{0}$
define the spaces 
\[
\label{eq:Hermite spaces}
H_{m}=\Span\{g_{\alpha}:\,\vert\alpha\vert=m\},
\]
with induced scalar product 
\[
\langle f,g\rangle_{m}:=\langle f,g\rangle_{L^2(\widehat{\pi})},\quad f,g\in H_{m}.
\]
The space $(H_{m},\langle\cdot,\cdot\rangle_{m})$ is then a real Hilbert
space with (finite) dimension
\[
\dim H_{m}=\left(\begin{array}{c}
m+2d-1\\
m
\end{array}\right).
\]
The following result (Theorem \ref{thm:L2 decomposition}) holds for operators of the form
\begin{equation}
\label{eq:OU_operator}
\mathcal{L}=-Bx\cdot\nabla+\nabla^{T}Q\nabla,
\end{equation}
where the quadratic drift and diffusion matrices $B$ and $Q$ are such that $\mathcal{L}$ is the generator of an ergodic stochastic process (see \cite[Definition 2.1]{Arnold2014} for precise conditions on $B$ and $Q$ that ensure ergodicity). The generator of the SDE \eqref{eq: unit covariance perfect perturbation} is given by \eqref{eq:OU_operator} with $B$ and $Q$  as in equations \eqref{eq:drift matrix}
and \eqref{eq:diffusion matrix}, respectively.  The following result provides an orthogonal decomposition of $L^{2}(\widehat{\pi})$ into invariant subspaces of the operator $\mathcal{L}$.
\begin{theorem}{\cite[Section 5]{Arnold2014}.}
	\label{thm:L2 decomposition}The following holds:
	\begin{enumerate}[label=(\alph*)]
		\item The space $L^{2}(\widehat{\pi})$ has a decomposition into mutually orthogonal
		subspaces:
		\[
		L^{2}(\widehat{\pi})=\bigoplus_{m\in\mathbb{N}_{0}}H_{m}.
		\]
		
		\item For all $m\in\mathbb{N}_{0}$, $H_{m}$ is invariant under $\mathcal{L}$
		as well as under the semigroup $(e^{-t\mathcal{L}})_{t\ge0}$. 
		\item The spectrum of $\mathcal{L}$ has the following decomposition:
		\[
		\sigma(\mathcal{L})=\bigcup_{m\in\mathbb{N}_{0}}\sigma(\mathcal{L}\vert_{H_{m}}),
		\]
		where 
		\begin{equation}
		\sigma(\mathcal{L}\vert_{H_{m}})=\left\lbrace\sum_{j=1}^{2d}\alpha_{j}\lambda_{j}:\,\vert\alpha\vert=m,\,\lambda_{j}\in\sigma(B)\right\rbrace.\label{eq:spectrum on subspaces}
		\end{equation}
		
	\end{enumerate}
\end{theorem}
\begin{remark}
	Note that by the ergodicity of the dynamics, $\ker\mathcal{L}$ consists of constant functions and so $\ker\mathcal{L}=H_{0}$. Therefore, $L^2_0(\widehat{\pi})$ has the decomposition
	\[
	L_{0}^{2}(\widehat{\pi})=L^{2}(\widehat{\pi})/\ker\mathcal{L}=\bigoplus_{m\ge1}H_{m}.
	\]
	\end{remark}
Our first main result of this section is an expression for the asymptotic
variance in terms of the unperturbed operator $\mathcal{L}_{0}$ and
the perturbation $\mathcal{A}$:
\begin{proposition}
	\label{prop:asymvar_op_formula}
	Let $f\in L_{0}^{2}(\pi)$  (so in particular $f=f(q)$).
	Then the associated asymptotic variance is given by 
	\begin{equation}
	\label{eq:asymvar_op_formula}
	\sigma_{f}^{2}=\langle f,-\mathcal{L}_{0}(\mathcal{L}_{0}^{2}+\mu^{2}\mathcal{A}^{*}\mathcal{A})^{-1}f\rangle_{L^{2}(\widehat{\pi})}.
	\end{equation}
	
\end{proposition}
\begin{remark}
The proof of the preceding Proposition will show that $\mathcal{L}_{0}^{2}+\mu^{2}\mathcal{A}^{*}\mathcal{A}$ is invertible on $L^2_0(\widehat{\pi})$ and that $(\mathcal{L}_{0}^{2}+\mu^{2}\mathcal{A}^{*}\mathcal{A})^{-1}f \in \mathcal{D}(\mathcal{L}_0)$ for all $f \in L^2_0(\widehat{\pi})$. 
\end{remark}
To prove Proposition \ref{prop:asymvar_op_formula} we will make use of the \emph{generator
	with reversed perturbation} 
\[
\mathcal{L}_{-}=\mathcal{L}_{0}-\mu\mathcal{A}
\]
and the \emph{momentum flip operator} 
\begin{align*}
P:L_{0}^{2}(\widehat{\pi}) & \rightarrow L_{0}^{2}(\widehat{\pi})\\
\phi(q,p) & \mapsto\phi(q,-p).
\end{align*}
Clearly, $P^{2}=I$ and $P^{*}=P$. Further properties of $\mathcal{L}_{0}$,
$\mathcal{A}$ and the auxiliary operators $\mathcal{L}_{-}$ and
$P$ are gathered in the following lemma:
\begin{lemma}
	\label{operator lemma}
	For all $\phi, \psi \in C^{\infty}(\mathbb{R}^{2d})\cap L^2(\widehat{\pi})$ the following holds:
	\begin{enumerate}[label=(\alph*)]
		\item \label{it:oplem1} The generator $\mathcal{L}_{0}$ is symmetric in $L^2(\widehat{\pi})$ with respect to $P$:
		\[
		\langle  \phi, P\mathcal{L}_{0}P \psi\rangle_{L^2(\widehat{\pi})}=\langle \mathcal{L}_{0} \phi, \psi \rangle_{L^2(\widehat{\pi})}.
		\]
		
		\item \label{it:oplem2} The perturbation $\mathcal{A}$ is skewadjoint in $L^{2}(\widehat{\pi})$:
		\[ 
		\mathcal{A}^{*} = -\mathcal{A}.
		\]
		
		\item \label{it:oplem3} The operators $\mathcal{L}_{0}$ and $\mathcal{A}$ commute:
		\[
		[\mathcal{L}_{0},\mathcal{A}]\phi=0.
		\]
		
		\item \label{it:oplem4} The perturbation $\mathcal{A}$ satisfies
		\[
		P\mathcal{A}P\phi=\mathcal{A}\phi.
		\]
		
		\item \label{it:oplem5} $\mathcal{L}$ and $\mathcal{L}_{-}$ commute,
		\begin{equation*}
		[\mathcal{L},\mathcal{L}_{-}]\phi = 0,
		\end{equation*}
		
		 and the following relation holds:
		\begin{equation}
		\langle \phi ,P\mathcal{L}P\psi\rangle_{L^{2}(\widehat{\pi})}=\langle\mathcal{L}_{-}\phi,\psi\rangle_{ L^{2}(\widehat{\pi})}.\label{eq:L+L-}
		\end{equation}
		\item \label{it:oplem6} 
		The operators $\mathcal{L}$, $\mathcal{L}_0$, $\mathcal{L}_{-}$, $\mathcal{A}$ and $P$ leave the Hermite spaces $H_m$ invariant.
	\end{enumerate}
\end{lemma}
\begin{remark}
	The claim \ref{it:oplem3} in the above lemma is crucial for our approach, which
	itself rests heavily on the fact that the $q-$ and $p-$perturbations
	match ($J_{1}=J_{2}$).
\end{remark}
\begin{proof}[of Lemma \ref{operator lemma}]
	To prove \ref{it:oplem1}, consider the following
	decomposition of $\mathcal{L}_{0}$ as in (\ref{eq:generator}):
	\[
	\mathcal{L}_{0}=\underbrace{p\cdot\nabla_{q}-q\cdot\nabla_{p}}_{\mathcal{L}_{ham}}+\underbrace{\gamma\left(- p\cdot\nabla_{p}+ \Delta_{p}\right)}_{\mathcal{L}_{therm}}.
	\]
	By partial integration it is straightforward to see that 
	\begin{equation*}
	\langle\phi,\mathcal{L}_{ham}\psi\rangle_{ L^{2}(\widehat{\pi})}=-\langle\mathcal{L}_{ham}\phi,\psi\rangle_{ L^{2}(\widehat{\pi})}
	\end{equation*}
	and
	\begin{equation*}
	 \langle \phi,\mathcal{L}_{therm}\psi\rangle_{ L^{2}(\widehat{\pi})}=\langle\mathcal{L}_{therm}\phi,\psi\rangle_{ L^{2}(\widehat{\pi})},
	 \end{equation*}
	 for all $\phi,\psi \in C^{\infty}(\mathbb{R}^{2d})\cap L^2(\widehat{\pi})$,
	  i.e. $\mathcal{L}_{ham}$ and $\mathcal{L}_{therm}$
	are antisymmetric and symmetric in $L^{2}(\widehat{\pi})$ respectively.
	Furthermore, we immediately see that $P\mathcal{L}_{ham}P\phi=-\mathcal{L}_{ham}\phi$ and $P\mathcal{L}_{therm}P\phi = \mathcal{L}_{therm}\phi$, so that
	\[
	\langle \phi,P\mathcal{L}_{0}P\psi\rangle_{ L^{2}(\widehat{\pi})}=\langle\phi,-\mathcal{L}_{ham}\psi+\mathcal{L}_{therm}\psi\rangle_{ L^{2}(\widehat{\pi})}=\langle\mathcal{L}_{0}\phi,\psi\rangle_{ L^{2}(\widehat{\pi})}.
	\]
	We note that this result holds in the more general setting of Section \ref{sec:perturbed_langevin} for the infinitesimal generator \eqref{eq:generator}.  The claim \ref{it:oplem2} follows by noting that the flow vector field $b(q,p)=(-Jq,-Jp)$ associated to $\mathcal{A}$ is divergence-free with respect to $\widehat{\pi}$, i.e. $\nabla \cdot(\widehat{\pi}b)=0$. Therefore, $\mathcal{A}$ is the generator of a strongly continuous unitary semigroup on $L^2(\widehat{\pi})$ and hence skewadjoint by Stone's Theorem.
  To prove \ref{it:oplem3} we use the decomposition $\mathcal{L}_{0}=\mathcal{L}_{ham}+\mathcal{L}_{therm}$ to obtain
	\begin{equation}
	\label{eq:oplemmac_proof}
	[\mathcal{L}_{0},\mathcal{A}]\phi=[\mathcal{L}_{ham},\mathcal{A}]\phi+[\mathcal{L}_{therm},\mathcal{A}]\phi,\quad \phi \in C^\infty(\mathbb{R}^{2d})\cap L^2(\widehat{\pi}).
	\end{equation}
	The first term of \eqref{eq:oplemmac_proof} gives 
	\begin{align*}
	[p\cdot\nabla_{q}-q\cdot\nabla_{p}&,-Jq\cdot\nabla_{q} -Jp\cdot\nabla_{p}]\phi\\
	& =\big([p\cdot\nabla_{q},-Jq\cdot\nabla_{q}]+[p\cdot\nabla_{q},-Jp\cdot\nabla_{p}]+[-q\cdot\nabla_{p},-Jq\cdot\nabla_{q}] \\
	& \qquad +[-q\cdot\nabla_{p},-Jp\cdot\nabla_{p}]\big)\phi\\
	&= Jp\cdot\nabla_{q}\phi-Jp\cdot\nabla_{q}\phi+Jq\cdot\nabla_{p}\phi-Jq\cdot\nabla_{p}\phi=0.
	\end{align*}
	The second term of \eqref{eq:oplemmac_proof} gives 
	\begin{equation}
	\label{eq:term1}
	[-p\cdot\nabla_{p}+\Delta_{p},\mathcal{A}]\phi =[-p\cdot\nabla_{p},-Jp\cdot\nabla_{p}]\phi+[\Delta_{p},-Jp\cdot\nabla_{p}]\phi,
	\end{equation}
	since $Jq\cdot\nabla_{q}$ commutes with $p\cdot\nabla_{p}+\Delta_{p}$. Both  terms in \eqref{eq:term1} are clearly zero due the antisymmetry of $J$ and the symmetry of the Hessian $D^2_p \phi$. 
	\\\\
	The claim \ref{it:oplem4} follows from a short calculation similar to the proof of  \ref{it:oplem1}.  To prove \ref{it:oplem5}, note that the fact that $\mathcal{L}$ and $\mathcal{L}_{-}$ commute follows from \ref{it:oplem3}, as 
	\[
	[\mathcal{L},\mathcal{L}_{-}]\phi=[\mathcal{L}_{0}+\mu\mathcal{A},\mathcal{L}_{0}-\mu\mathcal{A}]\phi=-2\mu[\mathcal{L}_{0},\mathcal{A}]\phi=0,\quad \phi \in C^{\infty}\cap L^2(\widehat{\pi}),
	\]
	while the property $\langle \phi ,P\mathcal{L}_{0}P\psi\rangle_{L^{2}(\widehat{\pi})}=\langle\mathcal{L}_{-}\phi,\psi\rangle_{ L^{2}(\widehat{\pi})}$ follows from properties \ref{it:oplem1}, \ref{it:oplem2} and \ref{it:oplem4}. Indeed,
	\begin{subequations}
	\begin{eqnarray*}
	\langle \phi,P\mathcal{L}P\psi\rangle_{ L^{2}(\widehat{\pi})}& = & \langle \phi, P(\mathcal{L}_{0}+\mu\mathcal{A})P\psi\rangle_{ L^{2}(\widehat{\pi})}=\langle\phi,\left(P\mathcal{L}_{0}P+\mu\mathcal{A}\right)\psi\rangle_{ L^{2}(\widehat{\pi})} \\	
	 & = & \langle (\mathcal{L}_{0}-\mu\mathcal{A})\phi,\psi\rangle_{ L^{2}(\widehat{\pi})}=\langle\mathcal{L}_{-}\phi,\psi\rangle_{ L^{2}(\widehat{\pi})}, 
\end{eqnarray*}
\end{subequations}
	as required. To prove \ref{it:oplem6} first notice that $\mathcal{L}$, $\mathcal{L}_0$ and $\mathcal{L}_{-}$ are of the form \eqref{eq:OU_operator} and therefore leave the spaces $H_m$ invariant by Theorem \ref{thm:L2 decomposition}. It follows immediately that also $\mathcal{A}$ leaves those spaces invariant. The fact that $P$ leaves the spaces $H_m$ invariant follows directly by inspection of \eqref{eq: Hermite polynomials}.
	\qed
\end{proof}
Now we proceed with the proof of Proposition  \ref{prop:asymvar_op_formula}:
\begin{proof}[of Proposition \ref{prop:asymvar_op_formula}] Since the potential $V$ is quadratic, Assumption \ref{ass:bounded+Poincare} clearly holds and thus Lemma \ref{lemma:variance} ensures that $\mathcal{L}$ and $\mathcal{L}_{-}$ are invertible on $L^2_{0}(\widehat{\pi})$ with 
\begin{equation}
\label{eq:Laplace transform}
\mathcal{L}^{-1}=\int_0^\infty e^{-t\mathcal{L}}\mathrm{d}t,
\end{equation}
	and analogously for $\mathcal{L}_{-}^{-1}$.
	 In particular, the asymptotic variance can be written as 
	 \begin{equation*}
	 \sigma_{f}^{2}=\langle f,(-\mathcal{L})^{-1}f\rangle_{L^{2}(\widehat{\pi})}.
	 \end{equation*}
	  Due to the respresentation \eqref{eq:Laplace transform} and Theorem \ref{thm:L2 decomposition}, the inverses of $\mathcal{L}$ and $\mathcal{L}_{-}$ leave the Hermite spaces $H_m$ invariant. We will prove the claim from Proposition \ref{prop:asymvar_op_formula} under the assumption that $Pf=f$ which includes the case 
	$f=f(q)$. For the following calculations we will assume $f\in H_m$ for fixed $m \ge 1$. Combining statement \ref{it:oplem6} with \ref{it:oplem1} and \ref{it:oplem5} of Lemma \ref{operator lemma} (and noting that $H_m \subset C^\infty(\mathbb{R}^{2d})\cap L^2(\widehat{\pi})$) we see that 
	\begin{equation}
	\label{eq:PLPL-}
	P\mathcal{L}P=\mathcal{L}_{-}^{*}
	\end{equation}
	 and 
	 \begin{equation}
	 P\mathcal{L}_{0}P=\mathcal{L}_{0}^{*}
	 \end{equation}
	  when restricted to $H_m$. Therefore, the following calculations are justified:
	\begin{align*}
	\langle f,(-\mathcal{L})^{-1}f\rangle_{L^{2}(\widehat{\pi})} &=\frac{1}{2}\langle f,(-\mathcal{L})^{-1}f\rangle_{L^{2}(\widehat{\pi})}+\langle f,(-\mathcal{L}^{*})^{-1}f\rangle_{L^{2}(\widehat{\pi})}\\
	&=\frac{1}{2}\langle f,(-\mathcal{L})^{-1}f\rangle_{L^{2}(\widehat{\pi})}+\langle Pf,(-\mathcal{L}^{*})^{-1}Pf\rangle_{L^{2}(\widehat{\pi})}\\
	&=\frac{1}{2}\langle f,(-\mathcal{L})^{-1}f\rangle_{L^{2}(\widehat{\pi})}+\langle f,(-\mathcal{L}_{-})^{-1}f\rangle_{L^{2}(\widehat{\pi})}\\
	&=\frac{1}{2}\langle f,\left((-\mathcal{L})^{-1}+(-\mathcal{L}_{-})^{-1}\right)f\rangle_{L^{2}(\widehat{\pi})},
	\end{align*}
	where in the third line we have used the assumption $Pf=f$ and in
	the fourth line the properties $P^{2}=I$, $P^{*}=P$ and equation
	(\ref{eq:PLPL-}).   Since $\mathcal{L}$ and $\mathcal{L}_{-}$ commute on $H_m$ according to Lemma
	\ref{operator lemma}\ref{it:oplem5},\ref{it:oplem6} we can write
	\begin{equation*}
	(-\mathcal{L})^{-1}+(-\mathcal{L}_{-})^{-1}  =\mathcal{L}_{-}(-\mathcal{L}\mathcal{L}_{-})^{-1}+\mathcal{L}(-\mathcal{L}\mathcal{L}_{-})^{-1}
	=-2\mathcal{L}_{0}(\mathcal{L}\mathcal{L}_{-})^{-1}
	\end{equation*}
	for the restrictions on $H_m$, 
	using $\mathcal{L}+\mathcal{L}_{-}=2\mathcal{L}_{0}$. We also have
	\begin{alignat*}{1}
	\mathcal{L}\mathcal{L}_{-} & =(\mathcal{L}_{0}+\mu\mathcal{A})(\mathcal{L}_{0}-\mu\mathcal{A}) =\mathcal{L}_{0}^{2}+\mu^{2}\mathcal{A}^{*}\mathcal{A},
	\end{alignat*}
	since $\mathcal{L}_{0}$ and $\mathcal{A}$ commute. We thus arrive at the formula
	\begin{equation}
	\label{eq:av_formula_Hm}
	\sigma_{f}^{2}=\langle f,-\mathcal{L}_{0}(\mathcal{L}_{0}^{2}+\mu^{2}\mathcal{A}^{*}\mathcal{A})^{-1}f\rangle_{L^{2}(\widehat{\pi})}, \quad f\in H_m.
	\end{equation}
	Now since $(\mathcal{L}_{0}^{2}+\mu^{2}\mathcal{A}^{*}\mathcal{A})^{-1}f = (\mathcal{L}\mathcal{L}_{-})^{-1}f \in \mathcal{D}(\mathcal{L}_{0})$ for all $f\in L^2(\widehat{\pi})$, it follows that the operator $-\mathcal{L}_{0}(\mathcal{L}_{0}^{2}+\mu^{2}\mathcal{A}^{*}\mathcal{A})^{-1}$ is bounded. We can therefore extend formula \eqref{eq:av_formula_Hm} to the whole of $L^2(\widehat{\pi})$ by continuity, using the fact that $L^2_0(\widehat{\pi})=\bigoplus_{m\ge 1}H_m$. 
	\qed
\end{proof}
Applying Proposition \ref{prop:asymvar_op_formula} we can analyse the behaviour
of $\sigma_{f}^{2}$ in the limit of large perturbation strength $\mu\rightarrow\infty$.
To this end, we introduce the orthogonal decomposition
\begin{equation}
\label{eq:kernel decomposition}
L_{0}^{2}(\pi)=\ker (Jq\cdot \nabla_q) \oplus\ker (Jq\cdot \nabla_q)^{\perp},
\end{equation}
where $Jq\cdot\nabla_q$ is understood as an unbounded operator acting on $L_0^2(\pi)$, obtained as the smallest closed extension of $Jq\cdot \nabla_q$ acting on $C^{\infty}_c(\mathbb{R}^d)$. In particular, $\ker (Jq\cdot \nabla_q)$ is a closed linear subspace of $L^2_0(\pi)$.   
Let $\Pi$ denote the $L_{0}^{2}(\pi)$-orthogonal projection onto
$\ker (Jq\cdot \nabla_q)$. We will write $\sigma_{f}^{2}(\mu)$ to
stress the dependence of the asymptotic variance on the perturbation
strength. The following result shows that for large perturbations,
the limiting asymptotic variance is always smaller than the asymptotic
variance in the unperturbed case. Furthermore, the limit is given as
the asymptotic variance of the projected observable $\Pi f$ for the
unperturbed dynamics.
\begin{theorem}
	\label{prop:large pert}
	Let $f\in L_{0}^{2}(\pi)$, then
	\[
	\lim_{\mu\rightarrow\infty}\sigma_{f}^{2}(\mu)=\sigma_{\Pi f}^{2}(0)\le\sigma_{f}^{2}(0).
	\]
\end{theorem}
\begin{remark}
	Note that the fact that the limit exists and is finite is nontrivial.
	In particular, as Figures \ref{fig:no_limit1} and \ref{fig:no_limit2} demonstrate, it is often
	the case that $\lim_{\mu\rightarrow\infty}\sigma_{f}^{2}(\mu)=\infty$
	if the condition $\mu=\nu$ is not satisfied.
\end{remark}
\begin{remark}
	\label{rem:projection}
	The projection $\Pi$ onto $\ker(Jq\cdot\nabla_q)$ can be understood in terms of Figure \ref{fig:good_spectrum}. Indeed, the eigenvalues on the real axis (highlighted by diamonds) are not affected by the perturbations. Let us denote by $\tilde{\Pi}$ the projection onto the span of the eigenspaces of those eigenvalues. As $\mu \rightarrow \infty$, the limiting asymptotic  variance is given as the asymptotic variance associated to the unperturbed dynamics of the projection $\tilde{\Pi}f$. If we denote by $\Pi_0$ the projection of $L^2(\widehat{\pi})$ onto $L^2_0(\pi)$, then we have that $\Pi\Pi_0=\Pi_0\tilde{\Pi}$. 
\end{remark}
\begin{proof}[of Theorem \ref{prop:large pert}]
	Note that $\mathcal{L}_{0}$ and $\mathcal{A}^{*}\mathcal{A}$ leave the Hermite spaces $H_m$ invariant and their restrictions to those spaces commute 
	(see Lemma \ref{operator lemma}, \ref{it:oplem2}, \ref{it:oplem3} and \ref{it:oplem6}). Furthermore, as the Hermite spaces $H_m$ are finite-dimensional, those operators have discrete spectrum. As $\mathcal{A}^{*}\mathcal{A}$
	is nonnegative self-adjoint, there exists an orthogonal
	decomposition $L_{0}^{2}(\pi)=\bigoplus_{i}W_{i}$  into eigenspaces of the operator $-\mathcal{L}_{0}(\mathcal{L}_{0}^{2}+\mu^{2}\mathcal{A}^{*}\mathcal{A})^{-1}$,
	the decomposition $\bigoplus W_i$ being finer then $\bigoplus H_m$ in the sense that every $W_i$ is a subspace of some $H_m$. 
	 Moreover,
	\[
	-\mathcal{L}_{0}(\mathcal{L}_{0}^{2}+\mu^{2}\mathcal{A}^{*}\mathcal{A})^{-1}\vert_{W_{i}}=-\mathcal{L}_{0}(\mathcal{L}_{0}^{2}+\mu^{2}\lambda_{i})^{-1}\vert_{W_i},
	\]
	where $\lambda_{i}\ge0$ is the eigenvalue of $\mathcal{A}^{*}\mathcal{A}$
	associated to the subspace $W_{i}$. Consequently, formula (\ref{eq:asymvar_op_formula})
	can be written as 
	\begin{equation}
	\label{eq:asymvar_spectral}
	\sigma_{f}^{2}=\sum_{i}\langle f_{i},-\mathcal{L}_{0}(\mathcal{L}_{0}^{2}+\mu^{2}\lambda_{i})^{-1}f_{i}\rangle_{L^{2}(\widehat{\pi})},
	\end{equation}
	where $f=\sum_{i}f_{i}$ and $f_{i}\in W_{i}$. Let us assume now
	without loss of generality that $W_{0}=\ker\mathcal{A}^{*}\mathcal{A}$,
	so in particular $\lambda_{0}=0$. Then clearly 
	\[
	\lim_{\mu\rightarrow\infty}\sigma_{f}^{2}=2\langle f_{0},-\mathcal{L}_{0}(\mathcal{L}_{0}^{2})^{-1}f_{0}\rangle_{L^{2}(\widehat{\pi})}=2\langle f_{0},(-\mathcal{L}_{0})^{-1}f_{0}\rangle_{L^{2}(\widehat{\pi})}=\sigma_{f_{0}}^{2}(0).
	\]
	Now note that $W_{0}=\ker\mathcal{A}^{*}\mathcal{A}=\ker\mathcal{A}$ due
	to $\ker\mathcal{A}^{*}=(\im\mathcal{A})^{\perp}$.  It remains to show that  $\sigma_{\Pi f}^{2}(0)\le\sigma_{f}^{2}(0)$.  To see this, we write 
	\begin{align*}
	\sigma_{f}^{2}(0) & =2\langle f,(-\mathcal{L}_{0})^{-1}f\rangle_{L^{2}(\widehat{\pi})}=2\langle\Pi f+(1-\Pi)f,(-\mathcal{L}_{0})^{-1}\big(\Pi f+(1-\Pi)f\big)\rangle_{L^{2}(\widehat{\pi})}\\
	& =\sigma_{\Pi f}^{2}(0)+\sigma_{(1-\Pi)f}^{2}(0)+R,
	\end{align*}
	where 
	\[
	R=2\langle\Pi f,(-\mathcal{L}_{0})^{-1}(1-\Pi)f\rangle_{L^{2}(\widehat{\pi})}+2\langle(1-\Pi)f,(-\mathcal{L}_{0})^{-1}\Pi f\rangle_{L^{2}(\widehat{\pi})}.
	\]
	Note that since we only consider observables that do not depend on $p$, $\Pi f\in \ker (Jq\cdot \nabla_q)$ and $(1-\Pi)f\in\bigoplus_{i\ge1}W_{i}$.
	Since $\mathcal{L}_{0}$ commutes with $\mathcal{A}$, it follows
	that $(-\mathcal{L}_{0})^{-1}$ leaves both $W_{0}$ and $\bigoplus_{i\ge1}W_{i}$
	invariant. Therefore, as the latter spaces are orthogonal to each
	other, it follows that $R=0$, from which the result follows. 
	\qed
\end{proof}
From Theorem \ref{prop:large pert} it follows that in the limit as $\mu \rightarrow \infty$, the asymptotic variance $\sigma_f^2(\mu)$ is not decreased by the perturbation if $f \in \ker(Jq \cdot \nabla_q)$. In fact, this result also holds true non-asymptotically, i.e. observables in $\ker(Jq \cdot \nabla_q)$ are not affected at all by the perturbation:
\begin{lemma}
	\label{lem:invariant observables}
	Let $f\in \ker (Jq\cdot \nabla_q)$. Then
	\begin{equation*}
	\sigma^2_f(\mu) = \sigma^2_f(0)
	\end{equation*}
	for all $\mu \in \mathbb{R}$.
\end{lemma}
\begin{proof}
	From $f \in  \ker (Jq\cdot \nabla_q)$ it follows immediately that $f \in \ker \mathcal{A}^{*}\mathcal{A}$. Then the claim follows from the expression \eqref{eq:asymvar_spectral}.
	\qed
\end{proof}
\begin{example}
	\label{ex:commutation quadratic observables}
	Recall the case of observables of the form $f(q)=q\cdot Kq+l\cdot q+C$
	with $K\in\mathbb{R}_{sym}^{d\times d}$, $l\in\mathbb{R}^{d}$ and
	$C\in\mathbb{R}$ from Section \ref{sec:small perturbations}. If $[J,K]=0$
	and $l\in\ker J$, then $f\in\ker (Jq\cdot \nabla_q)$ as 
	\[
	Jq\cdot\nabla_{q}(q\cdot Kq+l\cdot q+C)=2Jq\cdot Kq+Jq\cdot l=q\cdot(KJ-JK)q-q\cdot Jl=0.
	\]
	From the preceding lemma it follows that $\sigma_{f}^{2}(\mu)=\sigma_{f}^{2}(0)$
	for all $\mu\in\mathbb{R},$ showing that the assumption in Theorem
	\ref{cor:small pert unit var} does not exclude nontrivial cases.
\end{example}
The following result shows that the dynamics (\ref{eq: unit covariance perfect perturbation})
is particularly effective for antisymmetric observables (at least
in the limit of large perturbations):
\begin{proposition}
	\label{prop:antisymmetric observables}Let $f\in L_{0}^{2}(\pi)$
	satisfy $f(-q)=-f(q)$ and assume that $\ker J=\{0\}$.
	Furthermore, assume that the eigenvalues of $J$ are rationally independent, i.e. 
	\begin{equation}
	\label{eq:rat_indp_spectrum}
	\sigma(J)=\{\pm i\lambda_{1},\pm i\lambda_{2},\ldots,\pm i\lambda_d\}
	\end{equation}
	with $\lambda_{i}\in\mathbb{R}_{>0}$ and  $\sum_i k_i \lambda_i \neq 0$ for all $(k_1,\ldots,k_d)\in\mathbb{Z}^d\setminus(0,\ldots,0)$. Then $\lim_{\mu\rightarrow\infty}\sigma_{f}^{2}(\mu)=0$.
\end{proposition} 
\begin{proof}
	[of Proposition \ref{prop:antisymmetric observables}]
	The claim would
	immediately follow from $f\in\ker(Jq\cdot\nabla)^{\perp}$ according to Theorem \ref{prop:large pert}, but that does not seem to be so easy to prove directly. Instead, we again make
	use of the Hermite polynomials.
	
	Recall from the proof of Proposition \ref{prop:asymvar_op_formula} that $\gen$ is invertible on $L_{0}^{2}(\widehat{\pi})$ and its inverse leaves the Hermite spaces $H_m$ invariant. Consequently, the
	asymptotic variance of an observable $f\in L_{0}^{2}(\widehat{\pi})$ can be written as  
	\begin{subequations}
	\begin{eqnarray}
	\sigma_{f}^{2} & = & \langle f,(-\mathcal{L})^{-1}f\rangle_{L^{2}(\widehat{\pi})} \\
	& = & \sum_{m=1}^{\infty}\langle\Pi_{m}f,(-\mathcal{L}\vert_{H_{m}})^{-1}\Pi_{m}f\rangle_{L^2(\widehat{\pi})},\label{eq:asymvar decomposition} 
	\end{eqnarray}
	\end{subequations}
	where $\Pi_{m}:L_{0}^{2}(\widehat{\pi})\rightarrow H_{m}$ denotes the orthogonal
	projection onto $H_{m}$. From (\ref{eq: Hermite polynomials}) it
	is clear that $g_{a}$ is symmetric for $\vert\alpha\vert$ even and
	antisymmetric for $\vert\alpha\vert$ odd. Therefore, from $f$ being
	antisymmetric it follows that 
	\[
	f\in\bigoplus_{m\ge1,m\,\text{odd}}H_{m}.
	\]
	In view of (\ref{eq:spectrum of B}), (\ref{eq:spectrum on subspaces}) and (\ref{eq:rat_indp_spectrum})
	the spectrum of $\mathcal{L}_{\vert H_{m}}$ can be written as 
	\begin{subequations}
	\begin{eqnarray}
	\sigma(\mathcal{L}\vert_{H_{m}}) & = &\left\lbrace \mu\sum_{j=1}^{2d}\alpha_{j}\beta_{j}+C_{\alpha,\gamma}:\,\vert\alpha\vert=m,\,\beta_{j}\in\sigma(J)\right\rbrace  \nonumber \\
	& = & \left\lbrace i\mu\sum_{j=1}^{d}(\alpha_{j}-\alpha_{j+d})\lambda_{j}+C_{\alpha,\gamma}:\,\vert\alpha\vert=m\right\rbrace  \label{eq:spec_grow}
	\end{eqnarray}
	\end{subequations}
	with appropriate real constants $C_{\alpha,\gamma}\in\mathbb{R}$ that depend
	on $\alpha$ and $\gamma$, but not on $\mu$. For $\vert\alpha\vert=\sum_{j=1}^{2d} \alpha_j=m$ odd, we have that
	\begin{equation}
	\label{eq:nonzero}
	\sum_{j=1}^{d}(\alpha_{j}-\alpha_{j+d})\lambda_{j} \neq 0.
	\end{equation}
	Indeed, assume to the contrary that the above expression is zero. Then it follows that $\alpha_j = \alpha_{j+d}$ for all $j=1,\ldots,d$ by rational independence of $\lambda_1,\ldots,\lambda_d$.
	From \eqref{eq:spec_grow} and \eqref{eq:nonzero} it is clear that
	\begin{equation*}
	\sup \left\lbrace r>0 : B(0,r) \cap \sigma(\mathcal{L}\vert_{H_m}) = \emptyset \right\rbrace \xrightarrow{\mu \rightarrow \infty} \infty,
	\end{equation*}
	where $B(0,r)$ denotes the ball of radius $r$ centered at the origin in $\mathbb{C}$.
	Consequently, the spectral radius of $(-\mathcal{L}\vert_{H_m})^{-1}$ and hence $(-\mathcal{L}\vert_{H_m})^{-1}$ itself converge to zero as $\mu \rightarrow \infty$. The result then follows from (\ref{eq:asymvar decomposition}). \qed\end{proof}
\begin{remark}
	The idea of the preceding proof can be explained using Figure \ref{fig:good_spectrum} and Remark \ref{rem:projection}. Since the real eigenvalues correspond to Hermite polynomials of even order, antisymmetric observables are orthogonal to the associated subspaces. The rational independence condition on the eigenvalues of $J$ prevents cancellations  that would lead to further eigenvalues on the real axis.
\end{remark}
The following corollary gives a version of the converse of Proposition \ref{prop:antisymmetric observables} and provides further intuition into the mechanics of the variance reduction achieved by the perturbation.
\begin{corollary}
	Let $f\in L_{0}^{2}(\pi)$ and assume that $lim_{\mu\rightarrow\infty}\sigma_{f}^{2}(\mu)=0$. Then 
	\[
	\int_{B(0,r)}f\mathrm{dq=0}
	\]
	for all $r\in(0,\infty)$, where $B(0,r)$ denotes the ball centered at $0$ with radius $r$.
\end{corollary}
\begin{proof}
	According to Theorem \ref{prop:large pert},  $\lim_{\mu\rightarrow\infty}\sigma_{f}^{2}(\mu)=0$  implies $\sigma_{\Pi f}^{2}(0)=0$. We can write 
	\begin{subequations}
	\begin{eqnarray}
	\sigma_{\Pi f}^{2}(0) & = & \langle \Pi f, (-\mathcal{L}_0)^{-1}\Pi f \rangle_{ L^{2}(\widehat{\pi})} \nonumber \\
	& = & \frac{1}{2}\langle \Pi f, \left((-\mathcal{L}_0)^{-1}+(-\mathcal{L}^{*}_0)^{-1}\right)\Pi f \rangle_{ L^{2}(\widehat{\pi})} \nonumber
	\end{eqnarray}
	\end{subequations}
	and recall from the proof of Proposition \ref{prop:asymvar_op_formula} that $(-\mathcal{L}_0)^{-1}$ and $(-\mathcal{L}^{*}_0)^{-1}$ leave the Hermite spaces $H_m$ invariant. Therefore  
	\begin{equation}
	\ker \left((-\mathcal{L}_0)^{-1}+(-\mathcal{L}^{*}_0)^{-1}\right) = {0}
	\end{equation}
	in $L^2_0(\widehat{\pi})$, and in particular $\sigma_{\Pi f}^{2}(0)=0$ implies $\Pi f = 0$, which in turn shows that
	  $f\in\ker(Jq\cdot\nabla)^{\perp}$. Using $\ker(Jq\cdot\nabla)^{\perp}=\overline{\im(Jq\cdot\nabla)}$,
	it follows that there exists a sequence $(\phi_n)_n\in C_c^{\infty}(\mathbb{R}^d)$ such that $Jq\cdot\nabla\phi_n \rightarrow f$ in $L^2(\pi)$. Taking a subsequence if necessary, we can assume that the convergence is pointwise $\pi$-almost everywhere and that the sequence is pointwise bounded by a function in $L^1(\pi)$. 
	Since $J$ is antisymmetric, we have that $Jq\cdot\nabla\phi_n=\nabla\cdot(\phi_n Jq)$.
	Now Gauss's theorem yields
	\[
	\int_{B(0,r)}f\mathrm{d}q=\int_{B(0,r)}\nabla\cdot(\phi Jq)\mathrm{d}q=\int_{\partial B(0,r)}\phi Jq\cdot\mathrm{d}n,
	\]
	where $n$ denotes
	the outward normal to the sphere $\partial B(0,r)$. This quantity
	is zero due to the orthogonality of $Jq$ and $n$, and so the result
	follows from Lebesgue's dominated convergence theorem.\qed
\end{proof}
\subsection{Optimal Choices of $J$ for Quadratic Observables}

Assume $f\in L_{0}^{2}(\pi)$ is given by $f(q)=q\cdot Kq+l\cdot q -\Tr K$,  with $K\in\mathbb{R}_{sym}^{d\times d}$ and $l\in\mathbb{R}^{d}$ (note that the constant term is chosen such that $ \pi(f)=0 $). Our objective is to choose $J$ in such a way that $\lim_{\mu\rightarrow\infty}\sigma_{f}^{2}(\mu)$ becomes as small as possible. To stress the dependence on the choice
of $J$, we introduce the notation $\sigma_{f}^{2}(\mu,J)$. Also, we denote the orthogonal projection onto $(\ker J)^{\perp}$ by $\Pi^{\perp}_{\ker J}$.
\begin{lemma}
	\label{lem:lin_observables}{(Zero variance limit for linear observables).} Assume $K=0$ and $\Pi^{\perp}_{\ker J}l=0$. Then 
	\[
	\lim_{\mu\rightarrow\infty}\sigma_{f}^{2}(\mu,J)=0.
	\]
\end{lemma}
\begin{proof}
	According to Proposition \ref{prop:large pert}, we have to show that
	$\Pi f=0$, where $\Pi$ is the $L^{2}(\pi)$-orthogonal projection
	onto $\ker(Jq\cdot\nabla)$. Let us thus prove that 
	\[
	f\in\ker(Jq\cdot\nabla)^{\perp}=\overline{\im(Jq\cdot\nabla)^{*}}=\overline{\im(Jq\cdot\nabla)},
	\]
	where the second identity uses the fact that $(Jq\cdot\nabla)^{*}=-Jq\cdot\nabla$.
	Indeed, since $\Pi^{\perp}_{\ker J}=0$, by Fredholm's alternative there exists $u \in \mathbb{R}^d$ such that $Ju=l$. Now define $\phi\in L_{0}^{2}(\pi)$
	by $\phi(q)=-u\cdot q,$ leading to 
	\[
	f=Jq\cdot\nabla\phi,
	\]
	so the result follows.\qed
\end{proof}

\begin{lemma}
	\label{lem:optimal_perturbation}{(Zero variance limit for purely quadratic observables.)} Let $l=0$ and consider the decomposition $K=K_{0}+K_{1}$ into the traceless part $K_{0}=K-\frac{\Tr K}{d}\cdot I$ and the
	trace-part $K_{1}=\frac{\Tr K}{d}\cdot I.$ For the corresponding
	decomposition of the observable 
	\[
	f(q)=f_{0}(q)+f_{1}(q)=q\cdot K_{0}q+q\cdot K_{1}q-\Tr K
	\]
	the following holds:
	\begin{enumerate}[label=(\alph*)]
		\item There exists an antisymmetric matrix $J$ such that  $\lim_{\mu\rightarrow\infty}\sigma_{f_{0}}^{2}(\mu,J)=0,$
		and there is an algorithmic way (see Algorithm \ref{alg:optimal J}) to compute an appropriate $J$ in terms
		of $K$.
		\item The trace-part is not effected by the perturbation, i.e. $\sigma_{f_{1}}^{2}(\mu,J)=\sigma_{f_{1}}^{2}(0)$ for all $\mu\in\mathbb{R}$.
	\end{enumerate}
\end{lemma}
\begin{proof}
	To prove the first claim, according to Theorem \ref{prop:large pert}
	it is sufficient to show that $f_{0}\in\ker(Jq\cdot\nabla)^{\perp}=\overline{\im(Jq\cdot\nabla)}$.
	Let us consider the function $\phi(q)=q\cdot Aq$, with $A\in\mathbb{R}_{sym}^{d\times d}$. It holds that
	\begin{equation*}
	Jq\cdot\nabla\phi=q\cdot(J^{T}Aq)=q\cdot[A,J]q.
	\end{equation*}
	The task of finding an antisymmetric matrix $J$ such that 
	\begin{equation}
	\label{eq:quad_var_reduction}
	\lim_{\mu\rightarrow\infty}\sigma_{f_{0}}^{2}(\mu,J)=0
	\end{equation}
	can therefore be accomplished by constructing an antisymmetric matrix
	$J$ such that there exists a symmetric matrix $A$ with the property
	$K_{0}=[A,J]$.  Given any traceless matrix $K_{0}$ there exists
	an orthogonal matrix $U\in O(\mathbb{R}^{d})$ such that $UK_{0}U^{T}$
	has zero entries on the diagonal, and that $U$ can be obtained in
	an algorithmic manner (see for example \cite{alg_zero_diag} or \cite[Chapter 2, Section 2, Problem 3]{Horn_Johnson_Matrix_Analysis}; for the reader's convenience we have summarised the algorithm in Appendix \ref{tracefree}.) Assume
	thus that such a matrix $U\in O(\mathbb{R}^{d})$ has been found and choose real numbers $a_1,\ldots,a_d \in \mathbb{R}$ such that $a_{i}\neq a_{j}$ if $i\neq j$.
	We now set
	\begin{equation}
	\bar{A}=\diag(a_{1},\ldots,a_{n}),
	\end{equation}
	and 
	\begin{equation}
	\bar{J}_{ij}= 
	\begin{cases}
	\frac{(UK_{0}U^{T})_{ij}}{a_{i}-a_{j}} & \text{if } i\neq j, \\
	0 & \text{if } i=j. \\ 
	\end{cases}
	\end{equation}
	Observe that since $UK_{0}U^{T}$ is symmetric, $\bar{J}$ is antisymmetric. 
	A short calculation shows that $[\bar{A},\bar{J}]= UK_{0}U^{T}$.
	We can thus define $A=U^{T}\bar{A}U$ and $J=U^{T}\bar{J}U$ to obtain $[A,J]=K_0$. Therefore, the $J$ constructed in this way indeed satisfies \eqref{eq:quad_var_reduction}.  For the second claim, note that $f_{1}\in\ker(Jq\cdot\nabla)$, since
	\begin{equation}
	\label{eq:constant_trace}
	Jq\cdot\nabla\left(q\cdot\frac{\Tr K}{d}q\right)=2\frac{\Tr K}{d}q\cdot Jq=0
	\end{equation}
	because of the antisymmetry of $J$. The result then follows from
	Lemma \ref{lem:invariant observables}.\qed
\end{proof}
We would like to stress that the perturbation $J$ constructed in the previous lemma is far from unique due to the freedom of choice of $U$ and $a_1,\ldots,a_d \in \mathbb{R}$ in its proof. However, it is asymptotically optimal:
\begin{corollary}
	\label{cor:optimality}
	In the setting of Lemma \ref{lem:optimal_perturbation} the following holds:
	\[
	\min_{J^T=-J}\left(\lim_{\mu\rightarrow\infty} \sigma^2_{f}(\mu,J)\right)=\sigma^2_{f_1}(0).
	\]
\end{corollary}
\begin{proof}
	The claim follows immediately since $f_{1}\in\ker(Jq\cdot\nabla)$ for arbitrary antisymmetric $J$ as shown in \eqref{eq:constant_trace}, and therefore the contribution of the trace part $f_1$ to the asymptotic variance cannot be reduced by any choice of $J$ according to Lemma \ref{lem:invariant observables}.	
\end{proof}
As the proof of Lemma \ref{lem:optimal_perturbation} is constructive, we obtain the following algorithm for determining optimal perturbations for quadratic observables:
\begin{algorithm}
	\label{alg:optimal J}
	Given $K\in\mathbb{R}_{sym}^{d\times d}$,
	determine an optimal antisymmetric perturbation $J$ as follows:
	\begin{enumerate}
		\item Set $K_{0}=K-\frac{\Tr K}{d}\cdot I.$
		\item Find $U\in O(\mathbb{R}^{d})$ such that $UK_{0}U^{T}$ has zero entries
		on the diagonal (see Appendix \ref{tracefree}).
		\item Choose $a_{i}\in\mathbb{R},$ $i=1,\ldots d$ such that $a_{i}\neq a_{j}$
		for $i\neq j$ and set 
		\[
		\bar{J}_{ij}=\frac{(UK_{0}U^{T})_{ij}}{a_{i}-a_{j}}
		\]
		for $i\ne j$ and $\bar{J}_{ii}=0$ otherwise.
		\item
		Set $J=U^{T}\bar{J}U$.
	\end{enumerate}
\end{algorithm}
\begin{remark}
	In \cite{duncan2016variance}, the authors consider the task of finding optimal perturbations $J$ for the nonreversible overdamped Langevin dynamics given in \eqref{eq:nonrev_overdamped_J}. In the Gaussian case this optimization problem turns out be equivalent to the one considered in this section. Indeed, equation (39) of \cite{duncan2016variance} can be rephrased as 
	\begin{equation*}
	f \in \ker(Jq\cdot \nabla)^{\perp}.
	\end{equation*}
	Therefore, Algorithm \ref{alg:optimal J} and its generalization Algorithm \ref{alg:optimal J general} (described in Section~\ref{sec:arbitrary covariance})  can be used without modifications to find optimal perturbations of overdamped Langevin dynamics.
\end{remark}

\subsection{Gaussians with Arbitrary Covariance and Preconditioning}
\label{sec:arbitrary covariance}

In this section we  extend the results of the preceding sections to the case
when the target measure $\pi$ is given by a Gaussian with arbitrary
covariance, i.e. $V(q)=\frac{1}{2}q\cdot Sq$ with $S\in\mathbb{R}_{sym}^{d\times d}$ symmetric and positive definite. The
dynamics (\ref{eq:perturbed_underdamped}) then takes the
form 

\begin{align}
\mathrm{d}q_{t} & =M^{-1}p_{t}\mathrm{d}t-\mu J_{1}Sq_{t}\mathrm{d}t\nonumber, \\
\mathrm{d}p_{t} & =-Sq_{t}\mathrm{d}t-\nu J_{2}M^{-1}p_{t}\mathrm{d}t-\Gamma M^{-1}p_{t}\mathrm{d}t+\sqrt{2\Gamma}\mathrm{d}W_{t}.\label{eq:Underdamped Langevin Gaussian}
\end{align}
The key observation is now that the choices $M=S$ and $\Gamma=\gamma S$
together with the transformation $\widetilde{q}=S^{1/2}q$ and $\widetilde{p}=S^{-1/2}p$
lead to the dynamics
\begin{align}
\mathrm{d}\widetilde{q}_{t} & =\widetilde{p}_{t}\mathrm{d}t-\mu S^{1/2}J_{1}S^{1/2}\widetilde{q}_{t}\mathrm{d}t,\nonumber \\
\mathrm{d}\widetilde{p}_{t} & =-\widetilde{q}_{t}\mathrm{d}t-\mu S^{-1/2}J_{2}S^{-1/2}\widetilde{p}_{t}\mathrm{d}t-\gamma\widetilde{p}_{t}\mathrm{d}t+\sqrt{2\gamma}\mathrm{d}W_{t},\label{eq:Underdamped Langevin transformed}
\end{align}
which is of the form (\ref{eq:unit covariance}) if $J_{1}$ and
$J_{2}$ obey the condition $SJ_{1}S=J_{2}$ (note that both $S^{1/2}J_{1}S^{1/2}$
and $S^{-1/2}J_{2}S^{-1/2}$ are of course antisymmetric). Clearly
the dynamics (\ref{eq:Underdamped Langevin transformed}) is ergodic
with respect to a Gaussian measure with unit covariance, in the following
denoted by $\widetilde{\pi}$. The connection between the asymptotic variances
associated to (\ref{eq:Underdamped Langevin Gaussian}) and (\ref{eq:Underdamped Langevin transformed})
is as follows: 
\\\\
For an observable $f\in L_{0}^{2}(\pi)$ we can write 
\[
\sqrt{T}\bigg(\frac{1}{T}\int_{0}^{T}f(q_{s})\mathrm{d}s-\pi(f)\bigg)=\sqrt{T}\bigg(\frac{1}{T}\int_{0}^{T}\widetilde{f}(\widetilde{q}_{s})\mathrm{d}s-\widetilde{\pi}(\widetilde{f})\bigg),
\]
where $\widetilde{f}(q)=f(S^{-1/2}q)$. Therefore, the asymptotic variances
satisfy
\begin{equation}
\sigma_{f}^{2}=\widetilde{\sigma}_{\widetilde{f}}^{2},\label{eq:asymvar transform}
\end{equation}
where $\widetilde{\sigma}_{\widetilde{f}}^{2}$ denotes the asymptotic variance
of the process $(\widetilde{q}_{t})_{t\ge0}$. Because of this, the results
from the previous sections generalise to (\ref{eq:Underdamped Langevin Gaussian}),
subject to the condition that the choices $M=S$, $\Gamma=\gamma S$
and $SJ_{1}S=J_{2}$ are made. We formulate our results in this general
setting as corollaries:
\begin{corollary}
	\label{cor:small_pert_general}
	Consider the dynamics 
	\begin{align}
	\mathrm{d}q_{t} & =M^{-1}p_{t}\mathrm{d}t-\mu J_{1}\nabla V(q_{t})\mathrm{d}t,\nonumber \\
	\mathrm{d}p_{t} & =-\nabla V(q_{t})\mathrm{d}t-\mu J_{2}M^{-1}p_{t}\mathrm{d}t-\Gamma M^{-1}p_{t}\mathrm{d}t+\sqrt{2\Gamma}\mathrm{d}W_{t},\label{eq: perturbed Langevin corollary}
	\end{align}
	with $V(q)=\frac{1}{2}q\cdot Sq$. Assume that $M=S$, $\Gamma=\gamma S$
	with $\gamma > \sqrt{2}$ and $SJ_{1}S=J_{2}$. Let $f\in L^{2}(\pi)$ be an observable of the form 
	\begin{equation}
	f(q)=q\cdot Kq+l\cdot q+C\label{eq:quadratic observable}
	\end{equation}
	with $K\in\mathbb{R}_{sym}^{d\times d}$, $l\in\mathbb{R}^{d}$ and
	$C\in\mathbb{R}$. If at least one of the conditions $KJ_{1}S\neq SJ_{1}K$
	and $l \notin \ker J$ is satisfied, then the asymptotic variance is at a local maximum for the unperturbed sampler, i.e.
	\[
	\left. \partial_{\mu}\sigma_{f}^{2}\right\rvert_{\mu=0}=0\qquad \mbox{ and } \qquad 	\left. \partial_{\mu}^{2}\sigma_{f}^{2}\right\rvert_{\mu=0}<0.
	\]
\end{corollary}
\begin{proof}
	Note that 
	\[
	\widetilde{f}(q)=f(S^{-1/2}q)=q\cdot S^{-1/2}KS^{-1/2}q+S^{-1/2}l\cdot q+C=q\cdot\widetilde{K}q+\widetilde{l}\cdot q+C
	\]
	is again of the form (\ref{eq:quadratic observable}) (where in the
	last equality, $\widetilde{K}=S^{-1/2}KS^{-1/2}$ and $\widetilde{l}=S^{-1/2}l$
	have been defined). From (\ref{eq:Underdamped Langevin transformed}),
	(\ref{eq:asymvar transform}) and Theorem \ref{cor:small pert unit var}
	the claim follows if at least one of the conditions $[\widetilde{K},S^{1/2}J_{1}S^{1/2}]\neq0$
	and $\widetilde{l}\notin\ker(S^{1/2}J_{1}S^{1/2})$ is satisfied. The
	first of those can easily seen to be equivalent to 
	\[
	S^{-1/2}(KJS-SJK)S^{-1/2}\neq0,
	\]
	which is equivalent to $KJ_{1}S\neq SJ_{1}K$ since $S$ is nondegenerate.
	The second condition is equivalent to 
	\[
	S^{1/2}J_{1}l\neq0,
	\]
	which is equivalent to $J_{1}l\neq0,$ again by nondegeneracy of $S$. \qed \end{proof}
\begin{corollary}
	\label{cor:limit_asym_var}
	Assume the setting from the previous corollary and denote by $\Pi$
	the orthogonal projection onto $\ker(J_{1}Sq\cdot\nabla)$. For $f\in L^{2}(\pi)$
	it holds that
	\[
	\lim_{\mu\rightarrow\infty}\sigma_{f}^{2}(\mu)=\sigma_{\Pi f}^{2}(0)\le\sigma_{f}^{2}(0).
	\]
\end{corollary}
\begin{proof}
	Theorem \ref{prop:large pert} implies 
	\[
	\lim_{\mu\rightarrow\infty}\widetilde{\sigma}_{\widetilde{f}}^{2}(\mu)=\widetilde{\sigma}_{\widetilde{\Pi}\widetilde{f}}^{2}(0)\le\widetilde{\sigma}_{\widetilde{f}}^{2}(0)
	\]
	for the transformed system (\ref{eq:Underdamped Langevin transformed}).
	Here $\widetilde{f}(q)=f(S^{-1/2}q)$ is the transformed observable and
	$\widetilde{\Pi}$ denotes $L^{2}(\pi)$-orthogonal projection onto $\ker(S^{1/2}J_{1}S^{1/2}q\cdot\nabla)$.
	According to (\ref{eq:asymvar transform}), it is sufficient to show
	that $(\Pi f)\circ S^{-1/2}=\widetilde{\Pi}\widetilde{f}$. This however follows
	directly from the fact that the linear transformation $\phi\mapsto\phi\circ S^{1/2}$
	maps $\ker(S^{1/2}J_{1}S^{1/2}q\cdot\nabla)$ bijectively onto $\ker(J_{1}Sq\cdot\nabla)$.\qed
\end{proof}
Let us also reformulate Algorithm \ref{alg:optimal J} for the case of a Gaussian with arbitrary covariance.
\begin{algorithm}
	\label{alg:optimal J general}Given $K,S\in\mathbb{R}_{sym}^{d\times d}$
	with $f(q)=q\cdot Kq$ and $V(q)=\frac{1}{2}q\cdot Sq$ (assuming $S$ is nondegenerate), determine optimal perturbations $J_{1}$
	and $J_{2}$ as follows:
	\begin{enumerate}
		\item Set $\widetilde{K}=S^{-1/2}KS^{-1/2}$ and $\widetilde{K}_{0}=\widetilde{K}-\frac{\Tr\widetilde{K}}{d}\cdot I$.
		\item Find $U\in O(\mathbb{R}^{d})$ such that $U\widetilde{K}_{0}U^{T}$ has
		zero entries on the diagonal (see Appendix \ref{tracefree}).
		\item Choose $a_{i}\in\mathbb{R}$, $i=1,\ldots,d$ such that $a_{i}\ne a_{j}$
		for $i\ne j$ and set 
		\[
		\bar{J}_{ij}=\frac{(U\widetilde{K}_{0}U^{T})_{ij}}{a_{i}-a_{j}}.
		\]
		\item
		Set $\widetilde{J}=U^{T}\bar{J}U$.
		\item Put $J_{1}=S^{-1/2}\widetilde{J}S^{-1/2}$ and $J_{2}=S^{1/2}JS^{1/2}$.
	\end{enumerate}
\end{algorithm}
Finally, we obtain the following optimality result from Lemma \ref{lem:lin_observables} and Corollary \ref{cor:optimality}.
\begin{corollary}
	Let $f(q)=q\cdot Kq+l\cdot q-\Tr K$ and assume that $\Pi^{\perp}_{\ker J}l=0$.
	Then 
	\[
	\min_{J_1^T=-J_1,\, J_2=SJ_1 S}\left(\lim_{\mu\rightarrow\infty} \sigma^2_{f}(\mu,J_1,J_2)\right)=\sigma^2_{f_1}(0),
	\]
	where $f_{1}(q)=q\cdot K_{1}q$, $K_{1}=\frac{\Tr(S^{-1}K)}{d}S$.
	Optimal choices for $J_{1}$ and $J_{2}$ can be obtained using Algorithm \ref{alg:optimal J general}.
\end{corollary}
\begin{remark}
	Since in Section \ref{sec:small perturbations} we analysed the case
	where $J_{1}$ and $J_{2}$ are proportional, we are not able to drop
	the restriction $J_{2}=SJ_{1}S$ from the above optimality
	result. Analysis of completely arbitrary perturbations will be the
	subject of future work. 
\end{remark}

\begin{remark}
	The choices $M=S$ and $\Gamma=\gamma S$ have been introduced to
	make the perturbations considered in this article lead to samplers that perform well in terms of reducing the asymptotic variance. However, adjusting
	the mass and friction matrices according to the target covariance
	in this way (i.e. $M=S$ and $\Gamma=\gamma S$) is a popular way of preconditioning the dynamics, see for instance \cite{GirolamiCalderhead2011} and, in particular mass-tensor molecular dynamics~\cite{Bennett1975267}. Here we will present an argument why such a preconditioning
	is indeed beneficial in terms of the convergence rate of the dynamics.
	Let us first assume that $S$ is diagonal, i.e. $S=\diag(s^{(1)},\ldots,s^{(d)})$
	and that $M=\diag(m^{(d)},\ldots,m^{(d)})$ and $\Gamma=\diag(\gamma^{(d)},\ldots,\gamma^{(d)})$
	are chosen diagonally as well. Then (\ref{eq:Underdamped Langevin Gaussian})
	decouples into one-dimensional SDEs of the following form: 
	\begin{align}
	\mathrm{d}q_{t}^{(i)} & =\frac{1}{m^{(i)}}p_{t}^{(i)}\mathrm{d}t,\nonumber \\
	\mathrm{d}p_{t}^{(i)} & =-s^{(i)}q_{t}^{(i)}\mathrm{d}t-\frac{\gamma^{(i)}}{m^{(i)}}p_{t}^{(i)}\mathrm{d}t+\sqrt{2\gamma^{(i)}}\mathrm{d}W_{t},\quad i=1,\ldots,d.\label{eq:decoupled Langevin}
	\end{align}
	Let us write those Ornstein-Uhlenbeck processes as 
	\begin{equation}
	\mathrm{d}X_{t}^{(i)}=-B^{(i)}X_{t}^{(i)}\mathrm{d}t+\sqrt{2Q^{(i)}}\mathrm{d}W_{t}^{(i)}\label{eq: decoupled OU process}
	\end{equation}
	with 
	\[
	B^{(i)}=\left(\begin{array}{cc}
	0 & -\frac{1}{m^{(i)}}\\
	s^{(i)} & \frac{\gamma^{(i)}}{m^{(i)}}
	\end{array}\right)\,\text{and }\,Q^{(i)}=\left(\begin{array}{cc}
	0 & 0\\
	0 & \gamma^{(i)}
	\end{array}\right).
	\]
	As in Section \ref{sec:exp_decay}, the rate of the exponential decay of (\ref{eq: decoupled OU process}) is equal to $\min\text{Re}\,\sigma(B^{(i)})$. A short calculation shows that the eigenvalues of $B^{(i)}$ are given by  
	\[
	\lambda_{1,2}^{(i)}=\frac{\gamma^{(i)}}{2m^{(i)}}\pm\sqrt{\bigg(\frac{\gamma^{(i)}}{2m^{(i)}}\bigg)^{2}-\frac{s^{(i)}}{m^{(i)}}}.
	\]
	Therefore, the rate of exponential decay is maximal when 
	\begin{equation}
	\bigg(\frac{\gamma^{(i)}}{2m^{(i)}}\bigg)^{2}-\frac{s^{(i)}}{m^{(i)}}=0,\label{eq:gm constraint}
	\end{equation}
	in which case it is given by 
	\[
	(\lambda^{(i)})^{*}=\sqrt{\frac{s^{(i)}}{m^{(i)}}}.
	\]
	Naturally, it is reasonable to choose $m^{(i)}$ in such a way that
	the exponential rate $(\lambda^{(i)})^{*}$ is the same for all $i$, leading
	to the restriction $M=cS$ with $c>0$. Choosing $c$ small will result in fast convergence to equilibrium,
	but also make the dynamics (\ref{eq:decoupled Langevin}) quite stiff,
	requiring a very small timestep $\Delta t$ in a discretisation scheme.
	The choice of $c$ will therefore need to strike a balance between
	those two competing effects. The constraint (\ref{eq:gm constraint})
	then implies $\Gamma=2cS$.	 By a coordinate transformation, the preceding argument also applies if $S$, $M$ and $\Gamma$ are diagonal in the same basis, and of course $M$ and $\Gamma$ can always be chosen that way.
	Numerical experiments show that it is possible to increase the rate of convergence to equilibrium even further by choosing $M$ and $\Gamma$ nondiagonally with respect to $S$ (although
	only by a small margin). A clearer understanding of this is a topic of further investigation.
\end{remark}

%% file: numerics.tex
\subsection{Numerical Scheme}
\label{sec:numerical_scheme}

In this section we introduce a splitting scheme for simulating
the perturbed underdamped Langevin dynamics given by equation (\ref{eq:perturbed_underdamped}).
In the unpertubed case, i.e. when $J_{1}=J_{2}=0$, the right-hand side
can be decomposed into parts $A$, $B$ and $C$ according to 

\[
\mathrm{d}\left(\begin{array}{c}
q_{t}\\
p_{t}
\end{array}\right)=\underbrace{\left(\begin{array}{c}
	M^{-1}p_{t}\\
	\boldsymbol{0}
	\end{array}\right)\mathrm{d}t}_{A}+\underbrace{\left(\begin{array}{c}
	\boldsymbol{0}\\
	-\nabla V(q_{t})
	\end{array}\right)\mathrm{d}t}_{B}+\underbrace{\left(\begin{array}{c}
	\boldsymbol{0}\\
	-\Gamma M^{-1}+\sqrt{2\Gamma}\mathrm{d}W_{t}
	\end{array}\right),}_{O}
\]
i.e. $O$ refers to the Ornstein-Uhlenbeck part of the dynamics, whereas $A$ and $B$ stand for the momentum and position updates, respectively.
\\\\
One particular splitting scheme which has proven to be efficient is the  $BAOAB$ scheme, (see \cite{LeimkuhlerMatthews2015} and references therein).  The string
of letters refers to the order in which the different parts are integrated, namely
\begin{subequations}
\begin{eqnarray}
p_{n+1/2} & = & p_{n}-\frac{1}{2}\Delta t\nabla V(q_{n}),\\
q_{n+1/2} & = & q_{n}+\frac{1}{2}\Delta t\cdot M^{-1}p_{n+1/2},\\
\hat{p}   & = & \exp(-\Delta t\Gamma M^{-1})p_{n+1/2}+\sqrt{I-e^{-2\Gamma\Delta t}}\mathcal{N}(0,I),\label{eq:OU-step-1}\\
q_{n+1}   & = & q{}_{n+1/2}+\frac{1}{2}\Delta t\cdot M^{-1}\hat{p},\\
p_{n+1}   & = & \hat{p}-\frac{1}{2}\Delta t\cdot\nabla V(q_{n+1}).
\end{eqnarray}
\end{subequations}
We note that many different discretisation schemes such as $ABOBA$,
$OABAO$, etc. are viable, but that analytical and numerical evidence
has shown that the $BAOAB$-ordering has particularly good properties
to compute long-time ergodic averages with respect to $q$-dependent observables.
Motivated by this, we introduce the following perturbed scheme,
introducing additional Runge-Kutta integration steps between the $A$, $B$ and
$O$ parts:

\begin{subequations}
\begin{eqnarray}
p_{n+1/2} 	& = & p_{n}-\frac{1}{2}\Delta t\nabla V(q_{n}),\\
q_{n+1/2}	& = & q_{n}+\frac{1}{2}\Delta t\cdot M^{-1}p_{n+1/2},\\
q'_{n+1/2}	& = & RK_{4}(\frac{1}{2}\Delta t,q_{n+1/2})\label{eq:q'},\\
\hat{p}		& = & \exp(-\Delta t(\Gamma M^{-1}+\nu J_{2}M^{-1}))p_{n+1/2}+\sqrt{I-e^{-2\Gamma\Delta t}}\mathcal{N}(0,1)\label{OU},\\
q''_{n+1/2}	& = & RK_{4}(\frac{1}{2}\Delta t,q'_{n+1/2})\label{eq:q''},\\
q_{n+1}		& = & q''_{n+1/2}+\frac{1}{2}\Delta t\cdot M^{-1}\hat{p},\\
p_{n+1}		& = & \hat{p}-\frac{1}{2}\Delta t\cdot\nabla V(q_{n+1}),
\end{eqnarray}
\end{subequations}
where $RK_{4}(\Delta t,q_{0})$
refers to fourth order Runge-Kutta integration of the ODE 
\begin{equation}
\dot{q}=-J_{1}\nabla V(q), \quad q(0)=q_0\label{eq:J1 ODE}
\end{equation}
up until time $\Delta t$.
We remark that the $J_{2}$-perturbation is linear and can therefore be
included in the $O$-part without much computational overhead. Clearly,
other discretisation schemes are possible as well, for instance one could use a symplectic integrator for the ODE \eqref{eq:J1 ODE}, noting that it is of Hamiltonian type. However, since $V$ as the Hamiltonian for \eqref{eq:J1 ODE} is not separable in general, such a symplectic integrator would have to be implcit. Moreover, (\ref{eq:q'}) and (\ref{eq:q''})
could be merged since (\ref{eq:q''}) commutes with (\ref{OU}). In
this paper, we content ourselves with the above scheme for our numerical
experiments.
\begin{remark}
	The aformentioned schemes lead to an error in the approximation
	for $\pi(f)$, since the invariant measure $\pi$ is not preserved
	exactly by the numerical scheme. In practice, the $BAOAB$-scheme
	can therefore be accompanied by an accept-reject Metropolis step as in \cite{BAOABMetropolis},
	leading to an unbiased estimate of $\pi(f)$, albeit with an inflated
	variance. In this case, after every rejection the momentum variable
	has to be flipped ($p\mapsto-p$) in order to keep the correct invariant
	measure. We note here that our perturbed scheme can be 'Metropolized'
	in a similar way by 'flipping the matrices $J_{1}$ and $J_{2}$ after
	every rejection ($J_{1}\mapsto-J_{1}$ and $J_{2}\mapsto-J_{2})$
	and using an appropriate (volume-preserving and time-reversible) integrator for the
	dynamics given by (\ref{eq:J1 ODE}). Implementations of this idea are the subject of ongoing work.
\end{remark}

\subsection{Diffusion Bridge Sampling}

To numerically test our analytical results, we will apply the dynamics
(\ref{eq:perturbed_underdamped}) to sample a measure on
path space associated to a diffusion bridge. Specifically, consider
the SDE 

\[
\mathrm{d}X_{s}=-\nabla U(X_{s})\mathrm{d}s+\sqrt{2\beta^{-1}}\mathrm{d}W_{s},
\]
with $X_{s}\in\mathbb{R}^{n}$, $\beta>0$ and the potential $U:\mathbb{R}^{n}\rightarrow\mathbb{R}$ obeying adequate growth and smoothness conditions (see \cite{HairerStuartVoss2007}, Section 5 for precise statements). The law of the solution to this SDE conditioned on the events $X(0)=x_{-}$ and $X(s_{+})=x_{+}$ is a probability measure $\pi$ on $L^{2}([0,s_{+}],\mathbb{R}^{n})$
which poses a challenging and important sampling problem, especially if $U$ is multimodal. This
setting has been used as a test case for sampling probability measures
in high dimensions (see for example \cite{BeskosPinskiSanz-SernaEtAl2011} and
\cite{OttobrePillaiPinskiEtAl2016}). For a more detailed introduction (including applications)
see \cite{BeskosStuart2009} and for a rigorous theoretical treatment
the papers \cite{HairerStuartVossEtAl2005,HairerStuartVoss2007,HairerStuartVos2009,BeskosStuart2009} .

In the case $U\equiv0$, it can be shown that the law of the conditioned
process is given by a Gaussian measure $\pi_{0}$ with mean zero and
precision operator $\mathcal{S}=-\frac{\beta}{2}\Delta$ on the Sobolev
space $H^{1}([0,s_{+}],\mathbb{R}^{d})$ equipped with appropriate
boundary conditions. The general case can then be understood as a
perturbation thereof: The measure $\pi$ is absolutely continuous
with respect to $\pi_{0}$ with Radon-Nikodym derivative 
\begin{equation}
\frac{\mathrm{d}\pi}{\mathrm{d}\pi_{0}}\propto\exp\big(-\Psi\big),\label{eq:Radon-Nikodyn for diffusion bridges}
\end{equation}
where
\[
\Psi(x)=\frac{\beta}{2}\int_{0}^{s_{+}}G(x(s),\beta)\mathrm{d}s
\]
and 
\[
G(x,\beta)=\frac{1}{2}\vert\nabla U(x)\vert^{2}-\frac{1}{\beta}\Delta U(x).
\]
We will make the choice $x_{-}=x_{+}=0$, which is possible without
loss of generality as explained in \cite[Remark 3.1]{BeskosRobertsStuartEtAl2008}, leading to Dirichlet boundary conditions on $[0,s_+]$ for the precision operator $\mathcal{S}$. Furthermore, we choose $s_{+}=1$ and discretise the ensuing
$s$-interval $[0,1]$ according to 
\[
[0,1]=[0,s_{1})\cup[s_{1},s_{2})\cup\ldots\cup[s_{n-1},s_{n})\cup[s_{n},1]
\]
in an equidistant way with stespize $s_{j+1}-s_{j}\equiv\delta=\frac{1}{d+1}$. Functions on this grid are determined by the values $x(s_1)=x_1,\ldots,x(s_n)=x_n$, recalling that $x(0)=x(1)=0$ by the Dirichlet boundary conditions. We discretise the functional
$\Psi$ as
\begin{align*}
\tilde{\Psi}(x_{1},\ldots,x_{n}) & =\frac{\beta}{2}\delta\sum_{i=1}^{d}G(x_{i},\beta)\\
& =\frac{\beta}{2}\delta \sum_{i=1}^{d}\big((U'(x_{i})^{2}-\frac{1}{\beta}U''(x_{i})\big),
\end{align*}
such that its gradient is given by 
\[
(\nabla\tilde{\Psi})_{i}=\frac{\beta}{2}\delta\big(2U'(x_{i})U''(x_{i})-\frac{1}{\beta}U'''(x_{i})\big),\quad i=1,\ldots,d.
\]
The discretised version $A$ of the Dirichlet-Laplacian $\Delta$ on $[0,1]$ is given by
\[
A=\delta^{-2}\left(\begin{array}{ccccc}
-2 & 1\\
1 & -2\\
&  & \ldots\\
&  &  &  & 1\\
&  &  & 1 & -2
\end{array}\right).
\]
Following (\ref{eq:Radon-Nikodyn for diffusion bridges}), the discretised
target measure $\widehat{\pi}$ has the form
\[
\widehat{\pi}=\frac{1}{Z}e^{-V}\mathrm{d}x,
\]
with 
\[
V(x)=\tilde{\Psi}(x)-\frac{\beta\delta }{4}x\cdot Ax,\quad x\in\mathbb{R}^{d}.
\]
In the following we will consider the case $n=1$ with potential $U:\mathbb{R}\rightarrow\mathbb{R}$
given by $U(x)=\frac{1}{2}(x^{2}-1)^{2}$ and set $\beta=1$. To test
our algorithm we adjust the parameters $M$, $\Gamma$, $J_{1}$ and
$J_{2}$ according to the recommended choice in the Gaussian case,
\begin{equation}
\label{eq:Gaussian_parameters}
M  = S, \quad
\Gamma =\gamma S, \quad
SJ_{1}S =J_{2}, \quad 
\mu =\nu,
\end{equation}
where we take $S=\frac{\beta}{2}\delta\cdot A$ as the precision
operator of the Gaussian target.  We will consider the linear observable
$f_{1}(x)=l\cdot x$ with $l=(1,\ldots,1)$ and the quadratic observable
$f_{2}(x)=\vert x\vert^{2}$. In a first experiment we adjust the
perturbation $J_{1}$ (and via (\ref{eq:Gaussian_parameters})
also $J_{2}$) to the observable $f_{2}$ according to Algorithm \ref{alg:optimal J general}.
The dynamics (\ref{eq:perturbed_underdamped}) is integrated
using the splitting scheme introduced in Section \ref{sec:numerical_scheme}
with a stepsize of $\Delta t=10^{-4}$ over the time interval $[0,T]$
with $T=10^{2}$. Furthermore, we choose initial conditions $q_0=(1,\ldots,1)$, $p_0=(0,\ldots,0)$  and introduce a burn-in time $T_{0}=1$,
i.e. we take the estimator to be 
\[
\hat{\pi}(f)\approx\frac{1}{T-T_{0}}\int_{T_{0}}^{T}f(q_{t})\mathrm{d}t.
\]
We compute the variance of the above estimator from $N=500$ realisations
and compare the results for different choices of the friction coefficient $\gamma$
and of the perturbation strength $\mu$. 

The numerical experiments show that the perturbed dynamics generally outperform
the unperturbed dynamics independently of the choice of $\mu$ and $\gamma$, both for linear and quadratic observables. One notable exception is the behaviour of the linear observable for small friction $\gamma = 10^{-3}$ (see Figure \ref{fig:lin_1}), where the asymptotic variance initially increases for small perturbation strengths $\mu$. However, this does not contradict our analytical results, since the small perturbation results from Section \ref{sec:small perturbations} generally require $\gamma$ to be sufficiently big (for example $\gamma\ge \sqrt{2}$ in Theorem \ref{cor:small pert unit var}). We remark here that the condition $\gamma\ge\sqrt{2}$, while necessary for the theoretical results from Section \ref{sec:small perturbations}, is not a very advisable choice in practice (at least in this experiment), since Figures \ref{fig:lin_2} and \ref{fig:quad2} clearly indicate that the optimal friction is around $\gamma \approx 10^{-1}$. Interestingly, the problem of choosing a suitable value for the friction coefficient coefficient $\gamma$ becomes mitigated by the introduction of the perturbation: While the performance of the unperturbed sampler depends quite sensitively on $\gamma$, the asymptotic variance of the perturbed dynamics is a lot more stable with respect to variations of $\gamma$. 

In the regime of growing values of $\mu$, the experiments confirm the results from Section \ref{sec:large perturbations}, i.e. the asymptotic variance approaches a limit that is smaller than the asymptotic variance of the unperturbed dynamics.

As a final remark we report our finding that the performance of the sampler for the linear observable is qualitatively independent of the coice of $J_1$ (as long as $J_2$ is adjusted according to \eqref{eq:Gaussian_parameters}). This result is in alignment with Propostion \ref{prop:antisymmetric observables} which predicts good properties of the sampler for antisymmetric observables. In contrast to this, a judicious choice of $J_1$ is critical for quadratic observables. In particular, applying Algorithm \ref{alg:optimal J general} significantly improves the performance of the perturbed sampler in comparison to choosing $J_1$ arbitrarily.   
\begin{figure}
	\begin{subfigure}[b]{0.45 \textwidth}
		\includegraphics[width=\textwidth]{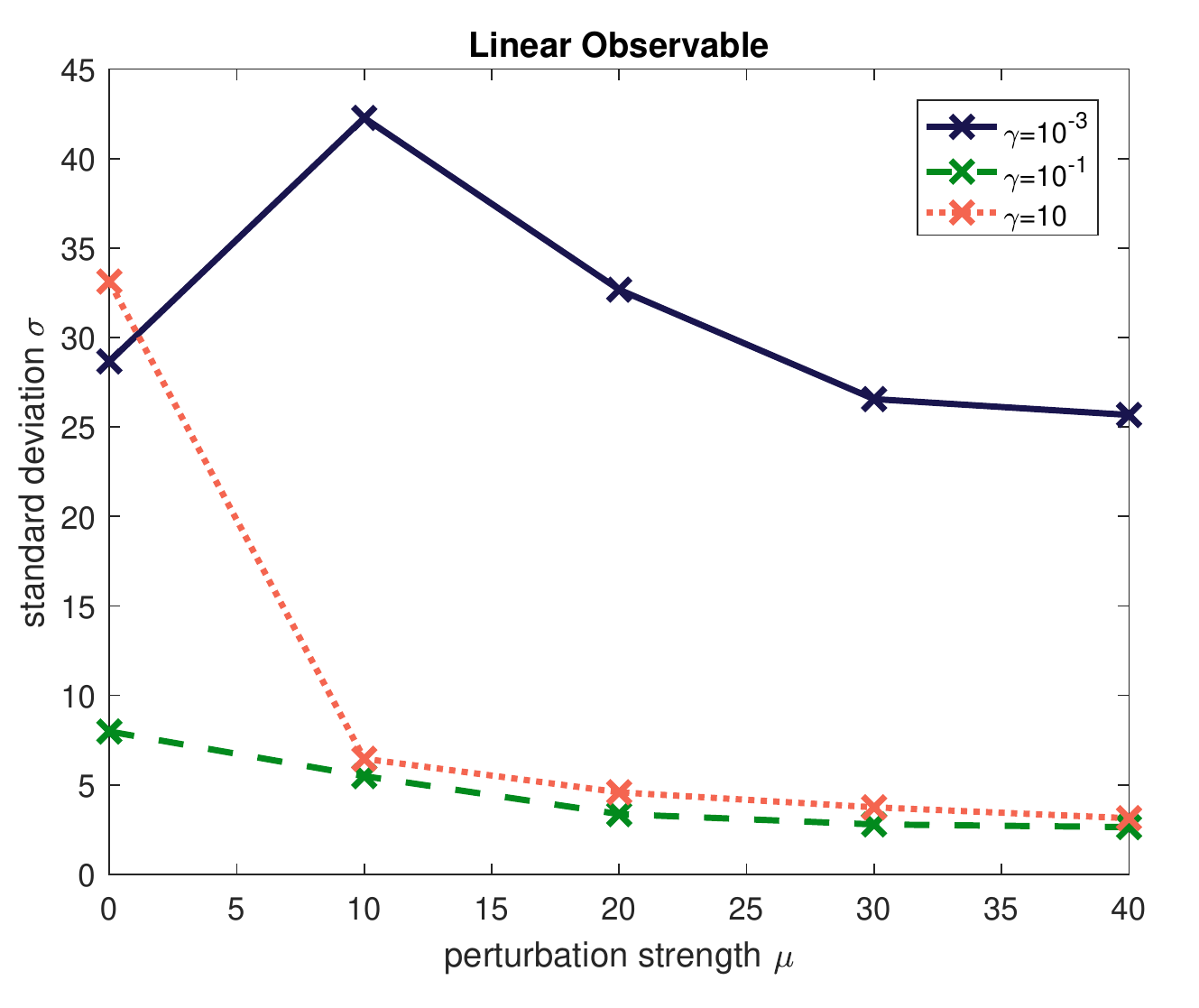}
		\caption{}
		\label{fig:lin_1}
	\end{subfigure}
	\hfill
	\begin{subfigure}[b]{0.45 \textwidth}
		\includegraphics[width=\textwidth]{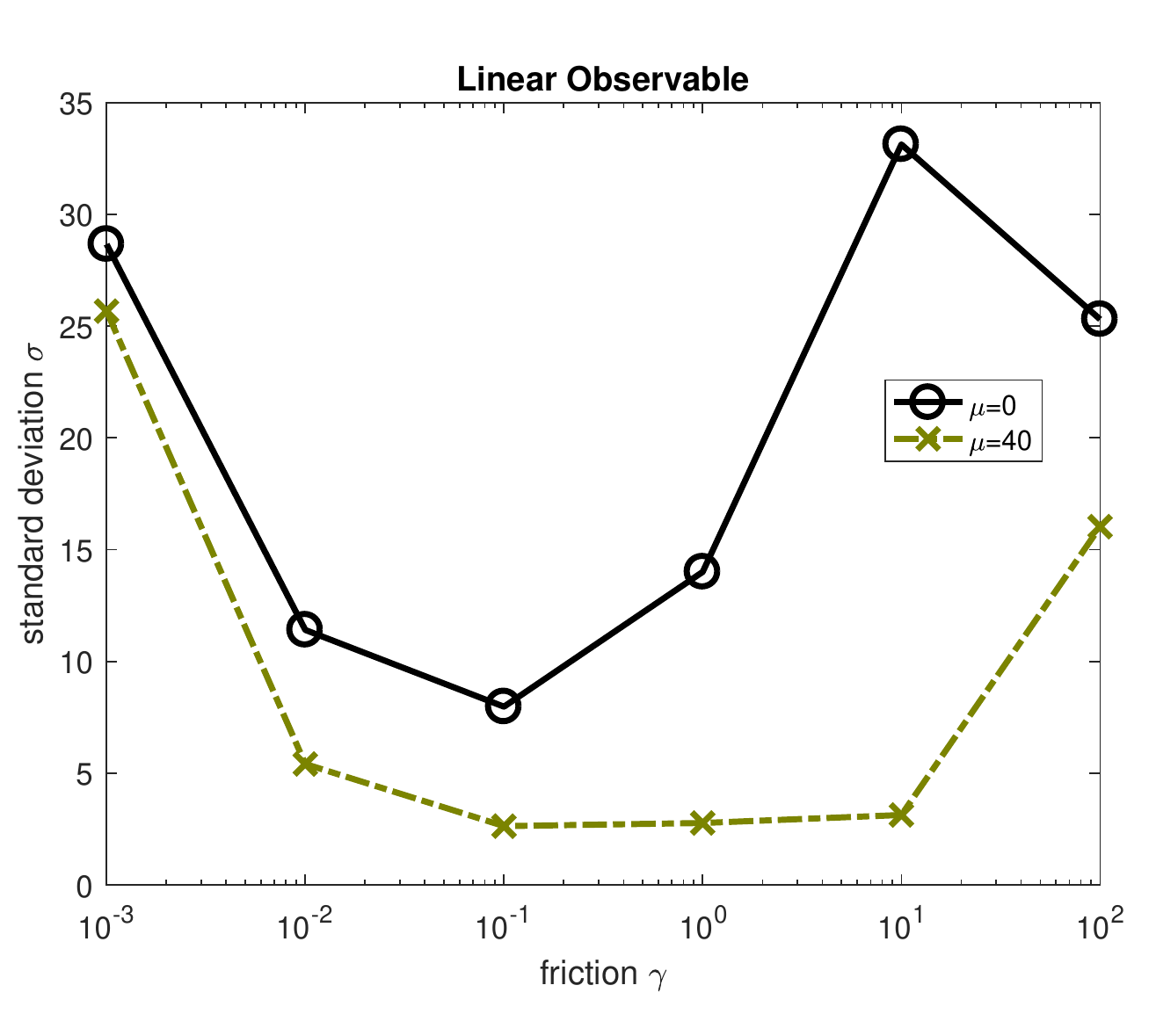}
		\caption{}
		\label{fig:lin_2}
	\end{subfigure}
	.	\caption{Standard deviation of $\hat{\pi}(f)$ for a linear observable as
			a function of friction $\gamma$ and perturbation strength $\mu$}
	\label{fig:linear_observable}
\end{figure}

\begin{figure}
	\begin{subfigure}[b]{0.45 \textwidth}
		\includegraphics[width=\textwidth]{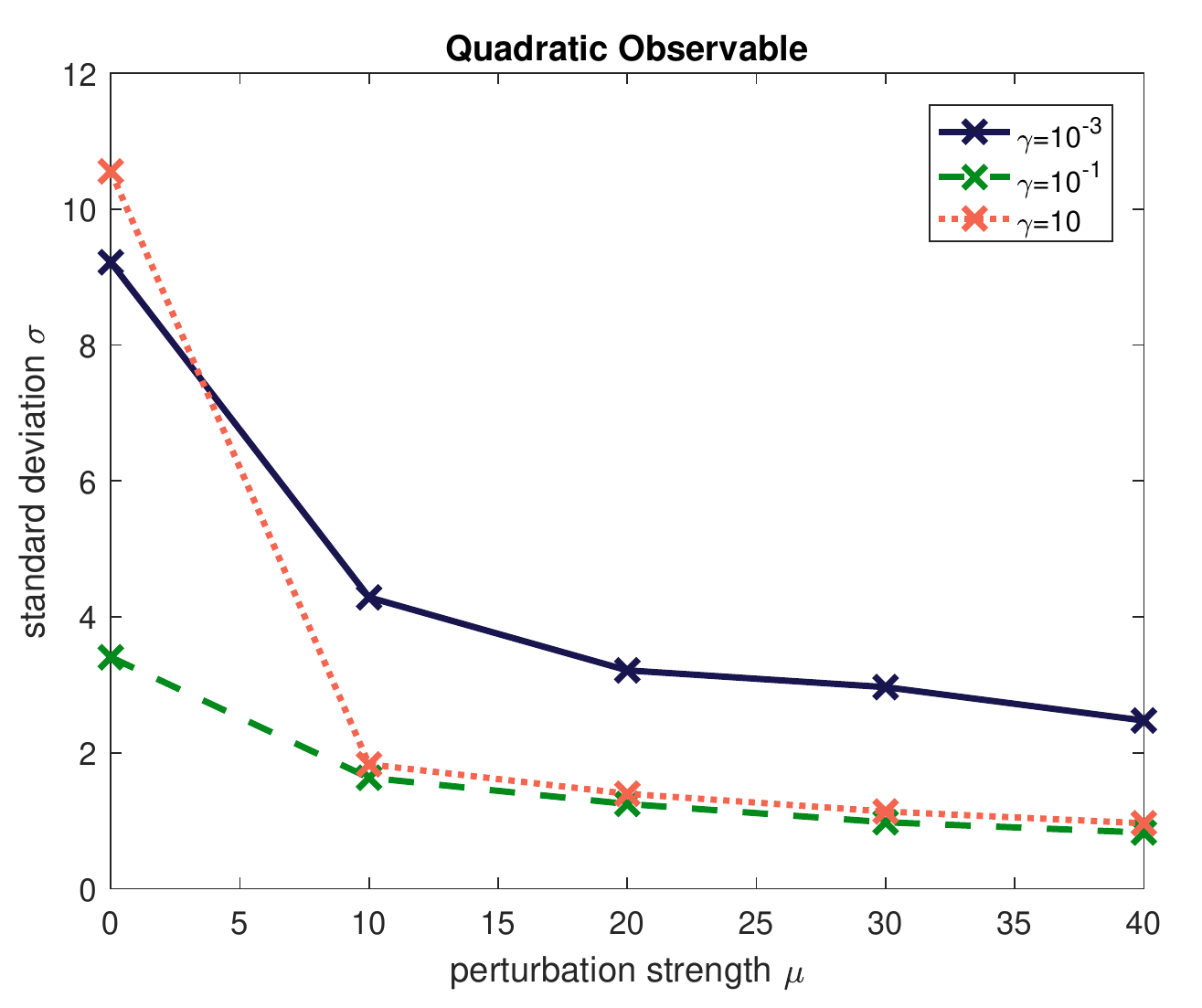}
		\caption{}
		\label{fig:quad1}
	\end{subfigure}
	\hfill
	\begin{subfigure}[b]{0.45 \textwidth}
		\includegraphics[width=\textwidth]{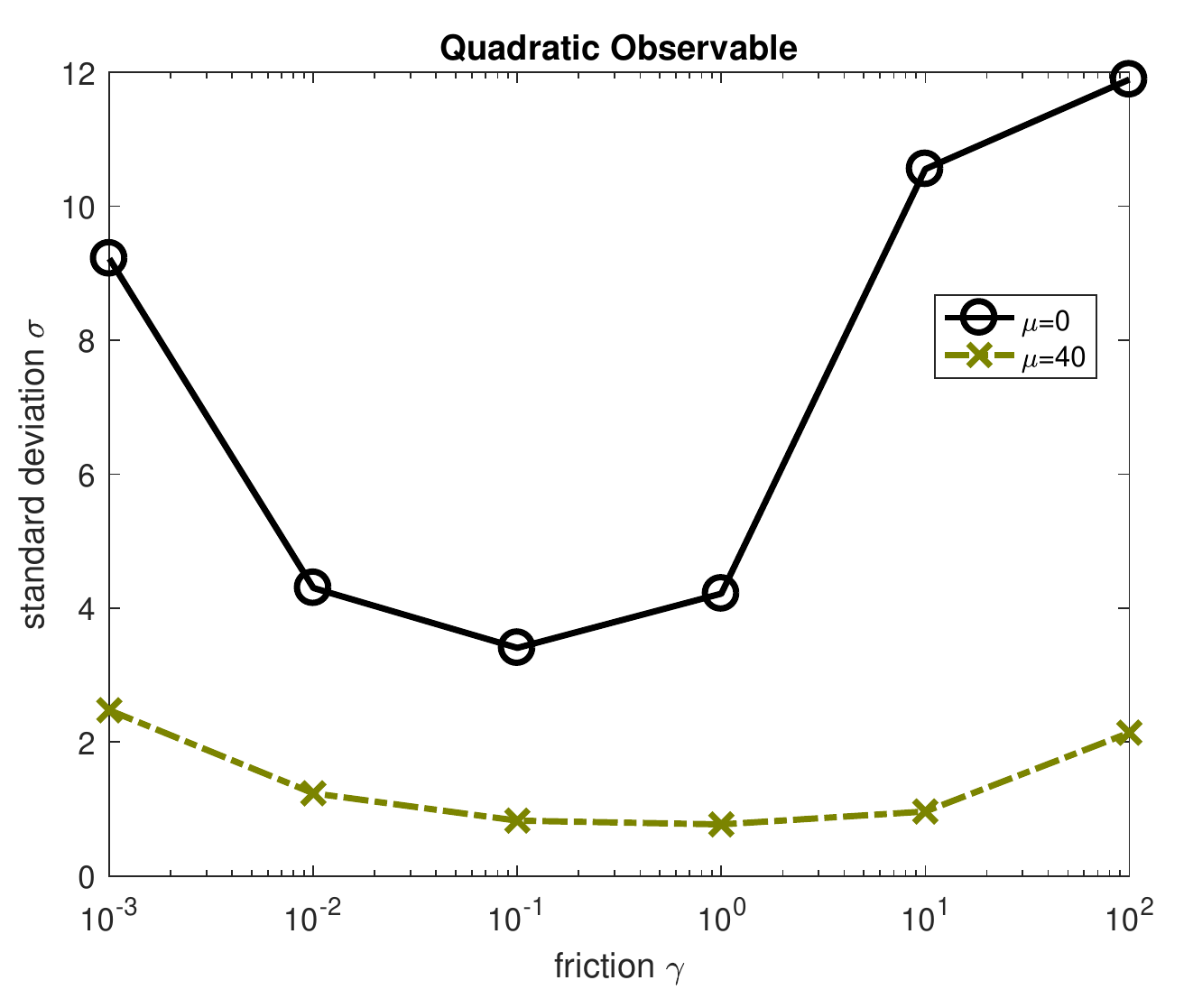}
		\caption{}
		\label{fig:quad2}
	\end{subfigure}
	.	\caption{Standard deviation of $\hat{\pi}(f)$ for a quadratic observable as
		a function of friction $\gamma$ and perturbation strength $\mu$}
	\label{fig:linear_observable}
\end{figure}

%% file: outlook.tex
A new family of Langevin samplers was introduced in this paper. These new SDE samplers consist of perturbations of the underdamped Langevin dynamics (that is known to be ergodic with respect to the canonical measure), where auxiliary drift terms in the equations for both the position and the momentum are added, in a way that the perturbed family of dynamics is ergodic with respect to the same (canonical) distribution. These new Langevin samplers were studied in detail for Gaussian target distributions where it was shown, using tools from spectral theory for differential operators, that an appropriate choice of the perturbations in the equations for the position and momentum can improve the performance of the Langvin sampler, at least in terms of reducing the asymptotic variance. The performance of the perturbed Langevin sampler to non-Gaussian target densities was tested numerically on the problem of diffusion bridge sampling.

The work presented in this paper can be improved and extended in several directions. First, a rigorous analysis of the new family of Langevin samplers for non-Gaussian target densities is needed. The analytical tools developed in~\cite{duncan2016variance} can be used as a starting point. Furthermore, the study of the actual computational cost and its minimization by an appropriate choice of the numerical scheme and of the perturbations in position and momentum would be of interest to practitioners. In addition, the analysis of our proposed samplers can be facilitated by using tools from symplectic and differential geometry. Finally, combining the new Langevin samplers with existing variance reduction techniques such as zero variance MCMC, preconditioning/Riemannian manifold MCMC can lead to sampling schemes that can be of interest to practitioners, in particular in molecular dynamics simulations. All these topics are currently under investigation.

%% file: appproofs.tex
\begin{proof}[of Lemma \ref{lemma:bias}]
Suppose that $(P_t)_{t \ge 0}$ satisfies \eqref{eq:hypocoercive estimate}.  Let $\pi_0$ be an initial distribution of $(X_t)_{t \ge 0}$ such that $\pi_0 \ll \pi$ and $h = \frac{d\pi_0}{d\pi} \in L^2(\pi)$.  Slightly abusing notation, we denote by $\pi_0 P_t$ the law of $X_t$ given $X_0 \sim \pi$.  Then
\begin{align*}
	\lVert \pi_0 P_t - \pi \rVert_{TV} = \left\lVert P_t^* h - 1 \right\rVert_{L^1(\pi)} \leq \left\lVert P_t^* \right\rVert_{L^2(\pi)\rightarrow L^2(\pi)} \left\lVert h - 1 \right\rVert_{L^2(\pi)} \leq Ce^{-\lambda t}\left\lVert h - 1 \right\rVert_{L^2(\pi)},
\end{align*}
where $P_t^*$ denotes the $L^2(\pi)$-adjoint of $P_t$.  Since $f$ is assumed to be bounded, we immediately obtain
$$
  \norm{\mathbb{E}[f(X_t) | X_0 \sim \pi_0] - \pi(f)} \leq C\Norm{f}_{L^\infty}e^{-\lambda t}\left(\mbox{Var}_{\pi}\left[\frac{d\pi_0}{d\pi}\right]\right)^{1/2},
$$
and so, for $X_0 \sim \pi_0$,
$$
	\norm{\pi_T(f) - \pi(f)} \leq \frac{C}{\lambda T}{\left(1 - e^{-\lambda t}\right)}\lVert f \rVert_{L^\infty}\left(\mbox{Var}_{\pi}\left[\frac{d\pi_0}{d\pi}\right]\right)^{1/2},
$$
as required.
  \qed 
\end{proof}

\begin{proof}[of Lemma \ref{lemma:variance}]
Given $f \in L^2(\pi)$, for fixed $T > 0$, 
\begin{equation}
  \chi_T(x) := \int_{0}^{T} \left(\pi(f) - P_t f(x)\right)\mathrm{d}t.
\end{equation}
Then we have that $\chi_T \in \mathcal{D}(\mathcal{L})$ and $\mathcal{L}\chi_T  = f - P_T f$, moreover
\begin{align*}
  \Norm{\chi_{T} - \chi_{T'}}_{L^2(\pi)} &= \Norm{\int^{T'}_{T} P_t(f)-\pi(f)\,\mathrm{d}t}_{L^2(\pi)}\\
  &\leq C\Norm{f}_{L^2(\pi)}\int_{T}^{T'}e^{-\lambda t}\,\mathrm{d}t,
\end{align*}
so that $\lbrace \chi_T \rbrace_{T \geq 0}$ is a Cauchy sequence in $L^2(\pi)$ converging to $\chi = \int_0^\infty \left(\pi(f) - P_tf\right)\mathrm{d}t$.  Since $\mathcal{L}$ is closed and
$$
  (\mathcal{L}\chi_T, \chi_T) \rightarrow (f-\pi(f), \chi),\quad T \rightarrow \infty,
$$
in $L^2(\pi)$, it follows that $\chi\in\mathcal{D}(\mathcal{L})$ and $\mathcal{L}\chi = f - \pi(f)$.  Moreover,
$$
  \Norm{\chi}_{L^{2}(\pi)} \leq \int_0^\infty \Norm{P_t(f) - \pi(f)}_{L^2(\pi)}\,\mathrm{d}t \leq K_{\lambda}\Norm{f-\pi(f)}_{L^2(\pi)},\quad 
$$
where $K_{\lambda} = C\int_0^\infty e^{-\lambda t}\,\mathrm{d}t$. 
Since we assume that $f$ is smooth, the coefficients are smooth and $\gen$ is hypoelliptic, then $\mathcal{L}\chi = f-\pi(f)$ implies that  $\chi \in C^{\infty}(\R^d)$, and thus we can apply It\^{o}'s formula to $\chi(X_t)$ to obtain:
$$
\frac{1}{T}\int_0^T \left[f(X_t) - \pi(f)\right]\,\mathrm{d}t = \frac{1}{T}\left[\chi(X_0) - \chi(X_T)\right] + \frac{1}{T}\int_0^T \nabla\chi(X_t)\sigma(X_t)\,\mathrm{d}W_t.
$$
One can check that the conditions of \cite[Theorem 7.1.4]{EthierKu86} hold.  In particular, the following central limit theorem follows
$$
	\frac{1}{\sqrt{T}}\int_0^T \nabla\chi(X_t)\sigma(X_t)\,\mathrm{d}W_t \xrightarrow{d} \mathcal{N}(0,2\sigma^2_f),\quad \mbox{ as } T \rightarrow \infty.
$$
By Theorem \ref{theorem:invariance_theorem}, the generator $\gen$ has the form
$$
	\gen = \pi^{-1}\nabla\cdot\left(\pi\Sigma \nabla \cdot \right) + \gamma\cdot\nabla,
$$
where $\nabla\cdot(\pi \gamma) = 0$.  It follows that
\begin{equation}
\label{eq:variance_equation}
\sigma^2_f = \inner{\Sigma \nabla\chi}{\nabla\chi}_{L^2(\pi)} = -\inner{\gen \chi}{\chi}_{L^2(\pi)}  = \inner{\chi}{f}_{L^2(\pi)} < \infty.
\end{equation}
First suppose that $X_0 \sim \pi$.  Then $(\chi(X_t))_{t\geq 0}$ is a stationary process, and so 
$$
	\frac{1}{\sqrt{T}}\left(\chi(X_0) - \chi(X_T)\right) \rightarrow 0,\quad \mbox{a.s as } T \rightarrow \infty.
$$
From which \eqref{eq:CLT} follows.  More generally, suppose that $X_0 \sim \pi_0$, where $\pi_0(x) = h(x)\pi(x)$ for $h \in L^2(\pi)$.  If $f \in L^\infty(\pi)$, then by \eqref{lemma:bias},
\begin{align*}
	|\chi(x)| &\leq \int_0^\infty |\pi(f) - P_t f(x)|\,dt \\
			  &\leq \int_0^\infty \lVert f\rVert_{L^\infty}\lVert\pi - \pi_0 P_t\rVert_{TV}\,dt \\
			  & \leq \frac{C}{\lambda }\lVert f\rVert_{L^\infty}\left(\mbox{Var}_{\pi}\left[\frac{d\pi_0}{d\pi}\right]\right)^{1/2},
\end{align*} 
so that $\chi \in L^\infty(\pi)$.  Therefore $\frac{1}{\sqrt{T}}(\chi(X_0)- \chi(X_T)) \xrightarrow{p} 0$ as $T \rightarrow \infty$, and so \eqref{eq:CLT} holds in this case, similarly.
\end{proof}

%% file: apphypocoercivity.tex
\begin{proof} of Lemma \ref{lem:hypoellipticity}
	We first note that $\gen$ in \eqref{eq:generator} can be written in the ``sum of squares'' form:
	$$
	\mathcal{L}=A_{0}+\frac{1}{2}\sum_{k=1}^{d}A_{k}^{2},
	$$
	where 
	$$
	A_{0}=M^{-1}p\cdot\nabla_{q}-\nabla_{q}V\cdot\nabla_{p}-\mu J_{1}\nabla_{q}V\cdot\nabla_{q}-\nu J_{2}M^{-1}p\cdot\nabla_{p}-\Gamma M^{-1}p\cdot\nabla_{p}
	$$
	and
	$$
	A_{k}=e_{k}\cdot\Gamma^{1/2}\nabla_{p}, \quad k = 1,\ldots, d.
	$$
	Here $\lbrace e_k \rbrace_{k=1,\ldots, d}$ denotes the standard Euclidean basis and $\Gamma^{1/2}$ is the unique positive definite square root of the matrix $\Gamma$.   The relevant commutators turn out to be
	\[
	[A_{0},A_{k}]=e_{k}\cdot\Gamma^{1/2} M^{-1}(\Gamma\nabla_{p}-\nabla_{q}-\nu J_{2}\nabla_{p}), \quad k = 1,\ldots, k.
	\]
	Because $\Gamma$ has full rank on $\R^d$, it follows that 
	\[
	\Span\{A_{k}:k=1,\ldots d\}=\Span\{\partial_{p_{k}}:k=1,\ldots,d\}.
	\]
	Since
	\[
	e_{k}\cdot\Gamma^{1/2} M^{-1}(\Gamma\nabla_{p}-\nu J_{2}\nabla_{p})\in\Span\{A_{j}:j=1,\ldots d\},\quad k=1,\ldots,d,
	\]
	and $\Span(\lbrace \Gamma^{1/2}M^{-1}\nabla_q \,:\, k=1,\ldots, d \rbrace) = \Span\{\partial_{q_{k}}:k=1,\ldots,d\}$, it follows that
	\[
	\Span(\lbrace A_{k}:k=0,1,\ldots,d\rbrace\cup\{[A_{0},A_{k}]:k=1,\ldots,d\}) = \R,
	\]
	so the assumptions of H\"{o}rmander's theorem hold.\qed
\end{proof}

\subsection{The overdamped limit}

The following is a technical lemma required for the proof of Proposition \ref{prop: overdamped limit}:
\begin{lemma}
	\label{lem:bounded p}Assume the conditions from Proposition \ref{prop: overdamped limit}.
	Then for every $T>0$ there exists $C>0$ such that 
	\[
	\mathbb{E} \left( \sup_{0\le t\le T}\vert p_{t}^{\epsilon}\vert^{2} \right) \le C.
	\]
	
\end{lemma}
\begin{proof}
	Using variation of constants, we can write the second line of (\ref{eq:rescaling})
	as 
	\[
	p_{t}^{\epsilon}=e^{-\frac{t}{\epsilon^{2}}(\nu J_{2}+\Gamma)M^{-1}}p_{0}-\frac{1}{\epsilon}\int_{0}^{t}e^{-\frac{(t-s)}{\epsilon^{2}}(\nu J_{2}+\Gamma)M^{-1}}\nabla_{q}V(q_{s}^{\epsilon})\mathrm{d}s+\frac{1}{\epsilon}\sqrt{2\Gamma}\int_{0}^{t}e^{-\frac{(t-s)}{\epsilon^{2}}(\nu J_{2}+\Gamma)M^{-1}}\mathrm{d}W_{s}.
	\]
	We then compute 
	\begin{align}
	\mathbb{E}\sup_{0\le t\le T}\vert p_{t}^{\epsilon}\vert^{2} & =\sup_{0\le t\le T}\left\lvert e^{-\frac{t}{\epsilon^{2}}(\nu J_{2}+\Gamma)M^{-1}}p_{0}\right\rvert^{2}+\frac{1}{\epsilon^{2}}\mathbb{E}\sup_{0\le t\le T}\left\lvert\int_{0}^{t}e^{-\frac{(t-s)}{\epsilon^{2}}(\nu J_{2}+\Gamma)M^{-1}}\nabla_{q}V(q_{s}^{\epsilon})\mathrm{d}s\right\rvert^{2}\nonumber \\
	& +\frac{1}{\epsilon^{2}}\mathbb{E}\sup_{0\le t\le T}\left\lvert\sqrt{2\Gamma}\int_{0}^{t}e^{-\frac{(t-s)}{\epsilon^{2}}(\nu J_{2}+\Gamma)M^{-1}}\mathrm{d}W_{s}\right\rvert^{2}\nonumber \\
	& -\frac{1}{\epsilon}\mathbb{E}\sup_{0\le t\le T}\left(e^{-\frac{t}{\epsilon^{2}}(\nu J_{2}+\Gamma)M^{-1}}p_{0}\cdot\int_{0}^{t}e^{-\frac{(t-s)}{\epsilon^{2}}(\nu J_{2}+\Gamma)M^{-1}}\nabla_{q}V(q_{s}^{\epsilon})\mathrm{d}s\right)\label{eq:P^2}\\
	& +\frac{1}{\epsilon}\mathbb{E}\sup_{0\le t\le T}\left(e^{-\frac{t}{\epsilon^{2}}(\nu J_{2}+\Gamma)M^{-1}}p_{0}\cdot\frac{1}{\epsilon}\sqrt{2\Gamma}\int_{0}^{t}e^{-\frac{(t-s)}{\epsilon^{2}}(\nu J_{2}+\Gamma)M^{-1}}\mathrm{d}W_{s}\right)\nonumber \\
	& -\frac{1}{\epsilon^{2}}\mathbb{E}\sup_{0\le t\le T}\left(\int_{0}^{t}e^{-\frac{(t-s)}{\epsilon^{2}}(\nu J_{2}+\Gamma)M^{-1}}\nabla_{q}V(q_{s}^{\epsilon})\mathrm{d}s\cdot\sqrt{2\Gamma}\int_{0}^{t}e^{-\frac{(t-s)}{\epsilon^{2}}(\nu J_{2}+\Gamma)M^{-1}}\mathrm{d}W_{s}\right).\nonumber 
	\end{align}
	Clearly, the first term on the right hand side of (\ref{eq:P^2})
	is bounded. For the second term, observe that 
	\begin{equation}
	\frac{1}{\epsilon^{2}}\mathbb{E}\sup_{0\le t\le T}\left\lvert\int_{0}^{t}e^{-\frac{(t-s)}{\epsilon^{2}}(\nu J_{2}+\Gamma)M^{-1}}\nabla_{q}V(q_{s}^{\epsilon})\mathrm{d}s\right\rvert^{2}\le\frac{1}{\epsilon^{2}}\sup_{0\le t\le T}\int_{0}^{t}\left\lVert e^{-\frac{(t-s)}{\epsilon^{2}}(\nu J_{2}+\Gamma)M^{-1}}\right\rVert^{2}\mathrm{d}s\label{eq:estimate1}
	\end{equation}
	since $V \in C^1(\mathbb{T}^d)$ and therefore $\nabla_{q}V$ is bounded. By the basic matrix exponential estimate
	$\Vert e^{-t(\nu J_{2}+\Gamma)M^{-1}}\Vert\le Ce^{-\omega t}$ for
	suitable $C$ and $\omega$, we see that (\ref{eq:estimate1}) can
	further be bounded by 
	\[
	\frac{1}{\epsilon^{2}}C\sup_{0\le t\le T}\int_{0}^{t}e^{-2\omega\frac{(t-s)}{\epsilon^{2}}}\mathrm{d}s=\frac{C}{2\omega}\left(1-e^{-2\omega\frac{T}{\epsilon^{2}}}\right),
	\]
	so this term is bounded as well. The third term is bounded by the
	Burkholder\textendash Davis\textendash Gundy inequality and a similar
	argument to the one used for the second term applies. The cross terms can
	be bounded by the previous ones, using the Cauchy-Schwarz inequality
	and the elementary fact that $\sup(ab)\le\sup a\cdot\sup b$ for $a,b>0$, so the
	result follows. \qed
\end{proof}

\begin{proof}
	[of Proposition \ref{prop: overdamped limit}] Equations (\ref{eq:rescaling})
	can be written in integral form as
	\[
	(\nu J_{2}+\Gamma)q_{t}^{\epsilon}=(\nu J_{2}+\Gamma)q_{0}^{\epsilon}+\frac{1}{\epsilon}\int_{0}^{t}(\nu J_{2}+\Gamma)M^{-1}p_{s}^{\epsilon}\mathrm{d}s-\mu\int_{0}^{t}(\nu J_{2}+\Gamma)J_{1}\nabla_{q}V(q_{s}^{\epsilon})\mathrm{d}s
	\]
	and 
	\begin{equation}
	-\int_{0}^{t}\nabla V(q_{s}^{\epsilon})\mathrm{d}s-\frac{1}{\epsilon}\int_{0}^{t}(\nu J_{2}+\Gamma)M^{-1}p_{s}^{\epsilon}\mathrm{d}s+\sqrt{2\Gamma}W(t)=\epsilon(p_{t}^{\epsilon}-p_{0}),\label{eq:rescaled p equation-1}
	\end{equation}
	where the first line has been multiplied by the matrix $\nu J_{2}+\Gamma$.
	Combining both equations yields
	\[
	q_{t}^{\epsilon}=q_{0}^{\epsilon}-\int_{0}^{t}(\nu J_{2}+\Gamma)\nabla_{q}V(q_{s}^{\epsilon})\mathrm{d}s-\epsilon(\nu J_{2}+\Gamma)^{-1}(p_{t}^{\epsilon}-p_{0})-\mu\int_{0}^{t}J_{1}\nabla_{q}V(q_{s})\mathrm{d}s+(\nu J_{2}+\Gamma)^{-1}\sqrt{2\Gamma}W_{t}.
	\]
	Now applying Lemma \ref{lem:bounded p} gives the desired result,
	since the above equation differs from the integral version of (\ref{eq:overdamped limit})
	only by the term $\epsilon(\nu J_{2}+\Gamma)^{-1}(p_{t}^{\epsilon}-p_{0})$
	which vanishes in the limit as $\epsilon\rightarrow0$. \qed
\end{proof}

\subsection{Hypocoercivity}

The objective of this section is to prove that the perturbed dynamics
(\ref{eq:perturbed_underdamped}) converges to equilibrium
exponentially fast, i.e. that the associated semigroup $(P_t)_{t\ge0}$ satisfies the estimate \eqref{eq:hypocoercive estimate}. We we will be using the theory of hypocoercivity outlined in
\cite{villani2009hypocoercivity} (see also the exposition in \cite[Section 6.2]{pavliotis2014stochastic}).
We provide a brief review of the theory of hypocoercivity.
\\\\
Let $(\mathcal{H},\langle\cdot,\cdot\rangle)$ be a real separable
Hilbert space and consider two unbounded operators $A$ and $B$ with
domains $D(A)$ and $D(B)$ respectively, $B$ antisymmetric. Let
$S\subset\mathcal{H}$ be a dense vectorspace such that $S\subset D(A)\cap D(B)$,
i.e. the operations of $A$ and $B$ are authorised on $S$. The theory
of hypocoercivity is concerned with equations of the form 
\begin{equation}
\label{eq:abstract fp equation}
\partial_{t}h+Lh=0,
\end{equation}
and the associated semigroup $(P_t)_{t\ge0}$ generated by $L=A^{*}A-B$. Let
us also introduce the notation $K=\ker L$. With the choices $\mathcal{H}=L^{2}(\widehat{\pi})$,
$A=\sigma\nabla_{p}$ and $B=M^{-1}p\cdot\nabla_{q}-\nabla_{q}V\cdot\nabla_{p}-\mu J_{1}\nabla_{q}V\cdot\nabla_{q}-\nu J_{2}M^{-1}p\cdot\nabla_{p},$
it turns out that $L$ is the (flat) $L^2(\mathbb{R}^{2d})$-adjoint of the generator $\mathcal{L}$ given in \eqref{eq:generator} and therefore equation \eqref{eq:abstract fp equation} is the Fokker-Planck equation associated to the dynamics \eqref{eq:perturbed_underdamped}. 
	In many situations of practical interest, the operator $A^{*}A$ is
	coercive only in certain directions of the state space, and therefore
	exponential return to equilibrium does not follow in general. In our
	case for instance, the noise acts only in the $p$-variables and therefore
	relaxation in the $q$-variables cannot be concluded a priori. However,
	intuitively speaking, the noise gets transported through the equations
	by the Hamiltonian part of the dynamics. This is what the theory of
	hypocoercivity makes precise. Under some conditions on the interactions
	between $A$ and $B$ (encoded in their iterated commutators), exponential
	return to equilibrium can be proved.
To state the main abstract theorem, we need the following definitions: 
\begin{definition}
	(Coercivity) Let $T$ be an unbounded operator on $\mathcal{H}$ with
	domain $D(T)$ and kernel $K$. Assume that there exists another Hilbert
	space $(\tilde{\mathcal{H}},\langle\cdot,\cdot\rangle_{\tilde{\mathcal{H}}})$,
	continuously and densely embedded in $K^{\perp}$. The operator is
	said to be $\lambda$-coercive if 
	\[
	\langle Th,h\rangle_{\tilde{\mathcal{H}}}\ge\lambda\Vert h\Vert_{\tilde{\mathcal{H}}}^{2}
	\]
	for all $h\in K^{\perp}\cap D(T)$. 
\end{definition}

\begin{definition}
	An operator $T$ on $\mathcal{H}$ is said to be relatively bounded
	with respect to the operators $T_{1},\ldots,T_{n}$ if the intersection
	of the domains $\cap D(T_{j})$ is contained in $D(T)$ and there
	exists a constant $\alpha>0$ such that 
	\[
	\Vert Th\Vert\le\alpha(\Vert T_{1}h\Vert+\ldots+\Vert T_{n}h\Vert)
	\]
	holds for all $h\in D(T)$. 
\end{definition}
We can now proceed to the main result of the theory.
\begin{theorem}{\cite[Theorem 24]{villani2009hypocoercivity}}
	\label{thm: hypocoercivity abstract}Assume there exists $N\in\mathbb{N}$
	and possibly unbounded operators $$C_{0},C_{1},\ldots,C_{N+1},R_{1},\ldots,R_{N+1},Z_{1},\ldots,Z_{N+1},$$
	such that $C_{0}=A$, 
	\begin{equation}
	[C_{j},B]=Z_{j+1}C_{j+1}+R_{j+1}\quad(0\le j\le N),\quad C_{N+1}=0,\label{eq:iterated commutators}
	\end{equation}
	and for all $k=0,1,\ldots,N$ 
	\begin{enumerate}[label=(\alph*)]
		\item \label{it:hypo1} $[A,C_{k}]$ is relatively bounded with respect to $\{C_{j}\}_{0\le j\le k}$
		and $\{C_{j}A\}_{0\le j\le k-1}$, 
		\item \label{it:hypo2} $[C_{k},A^{*}]$ is relatively bounded with respect to $I$ and $\{C_{j}\}_{0\le j\le k}$
		,
		\item \label{it:hypo3} $R_{k}$ is relatively bounded with respect to $\{C_{j}\}_{0\le j\le k-1}$
		and $\{C_{j}A\}_{0\le j\le k-1}$ and
		\item \label{it:hypo4} there are positive constants $\lambda_{i}$, $\Lambda_{i}$ such that
		$\lambda_{j}I\le Z_{j}\le\Lambda_{j}I$.
	\end{enumerate}
	Furthermore, assume that $\sum_{j=0}^{N}C_{j}^{*}C_{j}$ is $\kappa$-coercive
	for some $\kappa>0$.  Then, there exists $C\ge0$ and $\lambda>0$ such that 
	\begin{equation}
	\Vert P_t\Vert_{\mathcal{H}^{1}/K\rightarrow \mathcal{H}^{1}/K}\le Ce^{-\lambda t},\label{eq:hypocoercivity estimate}
	\end{equation}
	where $\mathcal{H}^{1}\subset\mathcal{H}$ is the subspace associated
	to the norm
	\begin{equation}
	\label{eq:abstractH1_norm}
	\Vert h\Vert_{\mathcal{H}^{1}}=\sqrt{\Vert h\Vert^{2}+\sum_{k=0}^{N}\Vert C_{k}h\Vert^{2}}
	\end{equation}
	and $K=\ker(A^{*}A-B)$. \end{theorem}
\begin{remark}
	Property (\ref{eq:hypocoercivity estimate}) is called \emph{hypocoercivity
		of $L$ on $\mathcal{H}^{1}:=(K^{\perp},\Vert\cdot\Vert_{\mathcal{H}^{1}})$.}
\end{remark}
If the conditions of the above theorem hold, we also get a regularization
result for the semigroup $e^{-tL}$ (see  \cite[Theorem A.12]{villani2009hypocoercivity}):
\begin{theorem}
	\label{thm:hypocoercive regularisation}Assume the setting and notation
	of Theorem \ref{thm: hypocoercivity abstract}. Then there exists
	a constant $C>0$ such that for all $k=0,1,\ldots,N$ and $t \in (0,1]$ the following
	holds:
	\[
	\Vert C_{k}P_{t}h\Vert\le C\frac{\Vert h\Vert}{t^{k+\frac{1}{2}}},\quad h\in\mathcal{H}.
	\]
\end{theorem}
\begin{proof}
	[of Theorem \ref{theorem:Hypocoercivity}]. We pove the claim by verifying
	the conditions of Theorem \ref{thm: hypocoercivity abstract}. Recall
	that $C_{0}=A=\sigma\nabla_{p}$ and 
	\[
	B=M^{-1}p\cdot\nabla_{q}-\nabla_{q}V\cdot\nabla_{p}-\mu J_{1}\nabla_{q}V\cdot\nabla_{q}-\nu J_{2}M^{-1}p\cdot\nabla_{p}.
	\]
	A quick calculation shows that 
	\[
	A^{*}=\sigma M^{-1}p-\sigma\nabla_{p},
	\]
	so that indeed 
	\[
	A^{*}A=\Gamma M^{-1}p\cdot\nabla_{p}-\nabla^{T}\Gamma\nabla=\mathcal{L}_{therm}
	\]
	and 
	\[
	A^{*}A-B=-\mathcal{L}^{*}.
	\]
	We make the choice $N=1$ and calculate the commutator 
	\[
	[A,B]=\sigma M^{-1}(\nabla_{q}+\nu J_{2}\nabla_{p}).
	\]
	Let us now set $C_{1}=\sigma M^{-1}\nabla_{q}$, $Z_{1}=1$ and $R_{1}=\nu\sigma M^{-1}J_{2}\nabla_{p}$,
	such that (\ref{eq:iterated commutators}) holds for $j=0$. Note
	that $[A,A]=0$\footnote{This is not true automatically, since $[A,A]$ stands for the array
		$([A_{j},A_{k}])_{jk}$.}, $[A,C_{1}]=0$ and $[A^{*},C_{1}]=0$. Furthermore, we have that
	\[
	[A,A^{*}]=\sigma M^{-1}\sigma.
	\]
	We now compute 
	\[
	[C_{1},B]=-\sigma M^{-1}\nabla^{2}V\nabla_{p}+\mu\sigma M^{-1}\nabla^{2}VJ_{1}\nabla_{q}
	\]
	and choose $R_{2}=[C_{1},B]$, $Z_{2}=1$ and recall that $C_{2}=0$
	by assumption (of Theorem \ref{thm: hypocoercivity abstract}). With those choices, assumptions \ref{it:hypo1}-\ref{it:hypo4} of Theorem \ref{thm: hypocoercivity abstract} are fulfilled. Indeed, assumption \ref{it:hypo1} holds trivially 
	since all relevant commutators are zero. Assumption \ref{it:hypo2} follows from the fact that $[A,A^{*}]=\sigma M^{-1}\sigma$ is clearly bounded
	relative to $I$. To verify assumption \ref{it:hypo3}, let us start with the
	case $k=1$. It is necessary to show that $R_{1}=\nu\sigma M^{-1}J_{2}\nabla_{p}$
	is bounded relatively to $A=\sigma\nabla_{p}$ and $A^{2}$.
	This is obvious since the $p$-derivatives appearing in $R_{1}$ can
	be controlled by the $p$-derivatives appearing in $A$. For $k=2,$
	a similar argument shows that $R_{2}=-\sigma M^{-1}\nabla^{2}V\nabla_{p}+\mu\sigma M^{-1}\nabla^{2}VJ_{1}\nabla_{q}$
	is bounded relatively to $A=\sigma\nabla_{p}$ and $C_{1}=\sigma M^{-1}\nabla_{q}$
	because of the assumption that $\nabla^{2}V$ is bounded. Note that it
	is crucial for the preceding arguments to assume that the matrices
	$\sigma$ and $M$ have full rank. Assumption \ref{it:hypo4} is trivially satisfied,
	since $Z_{1}$ and $Z_{2}$ are equal to the identity.  It remains to show that 
	\[
	T:=\sum_{j=0}^{N}C_{j}^{*}C_{j}
	\]
	is $\kappa$-coercive for some $\kappa>0$.  It is straightforward
	to see that the kernel of $T$ consists of constant functions and
	therefore 
	\[
	(\ker T)^{\perp}=\{\phi\in L^{2}(\mathbb{R}^{2d},\widehat{\pi}):\quad\widehat{\pi}(\phi)=0\}.
	\]
	Hence, $\kappa$-coercivity of $T$ amounts to the functional inequality
	\[
	\int_{\mathbb{R}^{2d}}\big(\vert\sigma M^{-1}\nabla_{q}\phi\vert^{2}+\vert\sigma\nabla_{p}\phi\vert^{2}\big)\mathrm{d}\widehat{\pi}\ge\kappa\bigg(\int_{\mathbb{R}^{2d}}\phi^{2}\mathrm{d}\widehat{\pi}-\left(\int_{\mathbb{R}^{2d}}\phi\mathrm{d}\widehat{\pi}\right)^{2}\bigg),\quad\phi\in H^{1}(\widehat{\pi}).
	\]
	Since the transformation $\phi\mapsto\psi$, $\psi(q,p)=\phi(\sigma^{-1}Mq,\sigma^{-1}p)$ is bijective on $H^{1}(\mathbb{R}^{2d},\widehat{\pi})$, the above is equivalent to 
	\[
	\int_{\mathbb{R}^{2d}}\big(\vert\nabla_{q}\psi\vert^{2}+\vert\nabla_{p}\psi\vert^{2}\big)\mathrm{d}\widehat{\pi}\ge\kappa\bigg(\int_{\mathbb{R}^{2d}}\psi^{2}\mathrm{d}\widehat{\pi}-\left(\int_{\mathbb{R}^{2d}}\psi\mathrm{d}\widehat{\pi}\right)^{2}\bigg),\quad\psi\in H^{1}(\widehat{\pi}),
	\]
	i.e. a Poincar\'{e} inequality for $\widehat{\pi}$. Since $\widehat{\pi}=\pi\otimes\mathcal{N}(0,M),$
	coercivity of $T$ boils down to a Poincar\'{e} inequality for $\pi$
	as in Assumption \ref{ass:bounded+Poincare}. This concludes the proof of the hypocoercive decay estimate
	(\ref{eq:hypocoercivity estimate}). Clearly, the abstract $\mathcal{H}^{1}$-norm from \eqref{eq:abstractH1_norm}
	is equivalent to the Sobolev norm $H^{1}(\widehat{\pi})$, and therefore it follows that there exist constants $C\ge0$ and $\lambda\ge0$ such that 
	\begin{equation}
	\label{eq:H1_decay}
	\Vert P_{t} f \Vert_{H^1(\widehat{\pi})}  \le C e^{-\lambda t} \Vert f \Vert _{H^1(\widehat{\pi})},
	\end{equation}
	for all $f \in H^1(\widehat{\pi})\setminus K$, where $K=\ker T$ consists of constant functions.  Let us now lift this estimate to $L^2(\widehat{\pi})$. There exist a constant $\tilde{C}\ge0$ such that 
	\begin{equation}
	\Vert h \Vert_{H^1(\widehat{\pi})} \le \tilde{C} \sum_{k=0}^2 \Vert C_k h \Vert_{L^2(\widehat{\pi})},
	\quad f \in H^1(\widehat{\pi}).
	\end{equation}
	Therefore, Theorem \ref{thm:hypocoercive regularisation} implies 
	\begin{equation}
	\label{eq:H1L2_reg}
	\Vert P_{1} f \Vert_{H^1(\widehat{\pi})} \le \tilde{C} \Vert f \Vert_{L^2(\widehat{\pi})},
	\quad f \in L^2(\widehat{\pi}), 
	\end{equation}
	for $t=1$ and a possibly different constant $\tilde{C}$. Let us now assume that $t\ge1$ and $f \in L^2(\widehat{\pi})\setminus K$.	It holds that
	\begin{equation}
	\Vert P_t f \Vert_{L^2(\widehat{\pi})} \le \Vert P_t f \Vert_{H^1(\widehat{\pi})} = \Vert P_{t-1}P_{1} f \Vert_{H^1(\widehat{\pi})}
	\le C e^{-\lambda (t-1)} \Vert P_{1} f \Vert_{H^1(\widehat{\pi})}, 
	\end{equation}
	where the last inequality follows from \eqref{eq:H1_decay}. Now applying \eqref{eq:H1L2_reg} and gathering constants results in 
	\begin{equation}
	\Vert P_t f\Vert_{L^2(\widehat{\pi})} \le C e^{-\lambda t}\Vert f \Vert_{L^2(\widehat{\pi})}, \quad f \in L^2(\widehat{\pi})\setminus K.
	\end{equation} 
	Note that although we assumed $t\ge1$, the above estimate also holds for $t\ge0$ (although possibly with a different constant $C$) since $\Vert P_t \Vert_{L^2(\widehat{\pi})\rightarrow L^2(\widehat{\pi})}$ is bounded on $[0,1]$. 
	\qed  
\end{proof}

%% file: app_Gaussian.tex
We begin by deriving a formula for the asymptotic variance of observables
of the form 
\[
f(q)=q\cdot Kq+l\cdot q-\Tr K,
\]
with $K\in\mathbb{R}_{sym}^{d\times d}$ and $l\in\mathbb{R}^{d}$.
Note that the constant term is chosen such that $\widehat{\pi}(f)=0$.
The following calculations are very much along the lines of \cite[Section  4]{duncan2016variance}. Since the Hessian of $V$ is bounded and the target measure $\pi$ is Gaussian, Assumption \ref{ass:bounded+Poincare} is satisfied and exponential decay of the semigroup $(P_t)_{t\ge0}$ as in \eqref{eq:hypocoercive estimate} follows by  Theorem \ref{theorem:Hypocoercivity}. According to Lemma \ref{lemma:variance}, the
asymptotic variance is then given by 
\begin{equation}
\sigma_{f}^{2}=\langle \chi,f\rangle_{L^{2}(\widehat{\pi})},
\end{equation}
where $\chi$ is the solution to the Poisson equation 
\begin{equation}
\label{eq:Poisson equation Gauss}
-\mathcal{L}\chi=f,\quad\widehat{\pi}(\chi)=0.
\end{equation}
Recall that 
\[
\mathcal{L}=-Bx\cdot\nabla+\nabla^{T}Q\nabla=-x\cdot A\nabla+\nabla^{T}Q\nabla
\]
is the generator as in (\ref{eq:OU generator}), where for later convenience
we have defined $A=B^{T}$, i.e.
\begin{equation}
A=\left(\begin{array}{cc}
-\mu J & I\\
-I & \gamma I -\nu J
\end{array}\right) \in \mathbb{R}^{2d \times 2d}.\label{eq:Amatrix}
\end{equation}
In the sequel we will solve (\ref{eq:Poisson equation Gauss}) analytically.  First, we introduce the notation 
\[
\bar{K}=\left(\begin{array}{cc}
K & \boldsymbol{0}\\
\boldsymbol{0} & \boldsymbol{0}
\end{array}\right)\in\mathbb{R}^{2d\times2d}
\]
and 
\[
\bar{l}=\left(\begin{array}{c}
l\\
\boldsymbol{0}
\end{array}\right)\in\mathbb{R}^{2d},
\]
such that by slight abuse of notation $f$ is given by 
\[
f(x)=x\cdot\bar{K}x+\bar{l}\cdot x-\Tr\bar{K}.
\]
By uniqueness (up to a constant) of the solution to the Poisson equation \eqref{eq:Poisson equation Gauss} and
linearity of $\mathcal{L}$, $g$ has to be a quadratic polynomial,
so we can write 
\[
g(x)=x\cdot Cx+D\cdot x-\Tr C,
\]
where $C\in\mathbb{R}_{sym}^{2d\times2d}$ and $D\in\mathbb{R}^{2d}$ (notice that $C$ can be chosen to be symmetrical since $x \cdot C x$ does not depend on the antisymmetric part of $C$).
Plugging this ansatz into (\ref{eq:Poisson equation Gauss}) yields
\[
-\mathcal{L}g(x)=x\cdot A\big(2Cx+D\big)-\gamma\Tr_{p}C=x\cdot\bar{K}x+\bar{l}\cdot x-\Tr\bar{K},
\]
where 
\[
\Tr_{p}C=\sum_{i=n+1}^{2n}C_{ii}
\]
denotes the trace of the momentum component of $C$. Comparing different
powers of $x$, this leads to the conditions
\begin{subequations}
\begin{eqnarray}
AC+CA^{T} & =\bar{K},
\label{eq:Lyapunov equation}\\
AD & =\bar{l},
\label{eq:linear condition}\\
\gamma\Tr_{p}C & =\Tr\bar{K}.
\label{eq:trace condition}
\end{eqnarray}
\end{subequations}
Note that (\ref{eq:trace condition}) will be satisfied eventually
by existence and uniqueness of the solution to (\ref{eq:Poisson equation Gauss}).
Then, by the calculations in \cite{duncan2016variance}, the asymptotic variance is given by 
\begin{equation}
\sigma_{f}^{2}=2\Tr(C\bar{K})+D\cdot\bar{l}.
\label{eq:Gaussian asymvar}
\end{equation}

\begin{proof}[of Proposition \ref{thm: local quadratic observable}]. According to
	(\ref{eq:Gaussian asymvar}) and (\ref{eq:Lyapunov equation}), the
	asymptotic variance satisfies 
	\[
	\sigma_{f}^{2}=2\Tr(C\bar{K}),
	\]
	where the matrix $C$ solves 
	\begin{equation}
	AC+CA^{T}=\bar{K}\label{eq: lyap equation}
	\end{equation}
	and $A$ is given as in (\ref{eq:Amatrix}). We will use the notation
	\[
	C(\mu,\nu)=\left(\begin{array}{cc}
	C_{1}(\mu,\nu) & C_{2}(\mu.\nu)\\
	C_{2}^{T}(\mu.\nu) & C_{3}(\mu,\nu)
	\end{array}\right)
	\]
	and the abbreviations $C(0):=C(0,0)$, $C^{\mu}(0):=\partial_{\mu}C\vert_{\mu,\nu=0}$
	and $C^{\nu}(0):=\partial_{\nu}C\vert_{\mu,\nu=0}$.  Let us first determine $C(0)$, i.e. the solution to the equation
	\[
	\left(\begin{array}{cc}
	\boldsymbol{0} & I\\
	-I & \gamma I
	\end{array}\right)C(0)+C(0)\left(\begin{array}{cc}
	\boldsymbol{0} & I\\
	-I & \gamma I
	\end{array}\right)^{T}=\left(\begin{array}{cc}
	K & \boldsymbol{0}\\
	\boldsymbol{0} & \boldsymbol{0}
	\end{array}\right).
	\]
	This leads to the following system of equations,
	\begin{subequations}
	\begin{eqnarray}
	C_{2}(0)+C_{2}(0)^{T} & =K,
	\label{eq:C(0) 1}\\
	-C_{1}(0)+\gamma C_{2}(0)+C_{3}(0) & =\boldsymbol{0},
	\label{eq:C(0) 2}\\
	-C_{1}(0)+\gamma C_{2}(0)^{T}+C_{3}(0) & =\boldsymbol{0},
	\label{eq:C(0) 3}\\
	-C_{2}(0)-C_{2}(0)^{T}+2\gamma C_{3}(0) & =\boldsymbol{0}.\\ \label{eq:C(0) 4}
	\end{eqnarray}
	\end{subequations}
	Note that equations (\ref{eq:C(0) 2}) and (\ref{eq:C(0) 3}) are
	equivalent by taking the transpose. Plugging (\ref{eq:C(0) 1}) into
	(\ref{eq:C(0) 4}) yields 
	\begin{equation}
	C_{3}(0)=\frac{1}{2\gamma}K.\label{eq:C3(0) result}
	\end{equation}
	Adding (\ref{eq:C(0) 2}) and (\ref{eq:C(0) 3}), together with (\ref{eq:C(0) 1})
	and (\ref{eq:C3(0) result}) leads to 
	\[
	C_{1}(0)=\frac{1}{2\gamma}K+\frac{\gamma}{2}K.
	\]
	Solving (\ref{eq:C(0) 2}) we obtain, 
	\[
	C_{2}(0)=\frac{1}{2}K,
	\]
	so that
	\begin{equation}
	C(0)=\left(\begin{array}{cc}
	\frac{1}{2\gamma}K+\frac{\gamma}{2}K & \frac{1}{2}K\\
	\frac{1}{2}K & \frac{1}{2\gamma}K
	\end{array}\right).\label{eq:C(0) result}
	\end{equation}
	Taking the $\mu$-derivative of (\ref{eq: lyap equation}) and setting
	$\mu=\nu=0$ yields 
	\begin{equation}
	A^{\mu}(0)C(0)+A(0)C^{\mu}(0)+C^{\mu}(0)A(0)^{T}+C(0)A^{\mu}(0)^{T}=\boldsymbol{0}.\label{eq:mu lyap}
	\end{equation}
	Notice that 
	\begin{align*}
	A^{\mu}(0)C(0)+C(0)A^{\mu}(0)^{T}\\
	=\left(\begin{array}{cc}
	-J & \boldsymbol{0}\\
	\boldsymbol{0} & \boldsymbol{0}
	\end{array}\right)C(0)+C(0)\left(\begin{array}{cc}
	J & \boldsymbol{0}\\
	\boldsymbol{0} & \boldsymbol{0}
	\end{array}\right)\\
	=\left(\begin{array}{cc}
	\big(\frac{1}{2\gamma}+\frac{\gamma}{2}\big)[K,J] & -\frac{1}{2}JK\\
	\frac{1}{2}KJ & \boldsymbol{0}
	\end{array}\right).
	\end{align*}
	With computations similar to those in the derivation of (\ref{eq:C(0) result})
	(or by simple substitution), equation (\ref{eq:mu lyap}) can be solved
	by 
	\begin{equation}
	C^{\mu}(0)=\left(\begin{array}{cc}
	-\big(\frac{\gamma^{2}}{4}+\frac{1}{4\gamma^{2}}+\frac{1}{4}\big)[K,J] & \frac{1}{2\gamma}JK-\frac{\gamma}{4}[K,J]\\
	-\frac{1}{2\gamma}KJ-\frac{\gamma}{4}[K,J] & -\big(\frac{1}{4\gamma^{2}}+\frac{1}{4}\big)[K,J]
	\end{array}\right).\label{eq:C^mu}
	\end{equation}
	We employ a similar strategy to determine $C^{\nu}(0)$: Taking the
	$\nu$-derivative in equation (\ref{eq: lyap equation}), setting
	$\mu=\nu=0$ and inserting $C(0)$ and $A(0)$ as in (\ref{eq:C(0) result})
	and (\ref{eq:Amatrix}) leads to the equation
	\[
	\left(\begin{array}{cc}
	\boldsymbol{0} & I\\
	-I & \gamma I
	\end{array}\right)C^{\nu}(0)+C^{\nu}(0)\left(\begin{array}{cc}
	\boldsymbol{0} & I\\
	-I & \gamma I
	\end{array}\right)=\left(\begin{array}{cc}
	\boldsymbol{0} & -\frac{1}{2}KJ\\
	\frac{1}{2}JK & -\frac{1}{2\gamma}[K,J]
	\end{array}\right),
	\]
	which can be solved by 
	\begin{equation}
	C^{\nu}(0)=\left(\begin{array}{cc}
	\big(-\frac{1}{4\gamma^{2}}+\frac{1}{4}\big)[K,J] & \frac{1}{\gamma}\big(-\frac{1}{2}KJ+\frac{1}{4}[K,J]\big)\\
	\frac{1}{\gamma}\big(\frac{1}{2}KJ-\frac{1}{4}[K,J]\big) & -\frac{1}{4\gamma^{2}}[K,J]
	\end{array}\right).\label{eq:C^nu}
	\end{equation}
	Note that $\Tr (C\bar{K})=\Tr(C_{1}K)$, and so 
	\begin{alignat*}{1}
	\partial_{\mu}\Theta\vert_{\mu,\nu=0} & =2\Tr(C_{1}^{\mu}(0)K)=\\
	& =-\big(\frac{\gamma^{2}}{4}+\frac{1}{4\gamma^{2}}+\frac{1}{4}\big)\cdot\Tr([K,J]K)=0,
	\end{alignat*}
	since clearly $\Tr([K,J],K)=\Tr(KJK)-\Tr(JK^{2})=0$. In the same
	way it follows that 
	\[
	\partial_{\nu}\Theta\vert_{\mu,\nu=0}=0,
	\]
	proving (\ref{eq:gradTheta}). 
	\\\\
	Taking the second $\mu$-derivative of (\ref{eq: lyap equation})
	and setting $\mu=\nu=0$ yields
	\[
	2A^{\mu}(0)C^{\mu}(0)+A(0)C^{\mu\mu}(0)+C^{\mu\mu}(0)A(0)^{T}+2C^{\mu}(0)A^{\mu}(0)^{T}=\boldsymbol{0},
	\]
	employing the notation $C^{\mu\mu}(0)=\partial_{\mu}^{2}C\vert_{\mu,\nu=0}$
	and noticing that $\partial_{\mu}^{2}A=0$. Using (\ref{eq:C^mu})
	we calculate
	\[
	A^{\mu}(0)C^{\mu}(0)+C^{\mu}(0)A^{\mu}(0)^{T}=\left(\begin{array}{cc}
	\big(\frac{\gamma^{2}}{4}+\frac{1}{4\gamma^{2}}+\frac{1}{4}\big)[J,[K,J]] & -\frac{1}{2\gamma}J^{2}K+\frac{\gamma}{4}J[K,J]\\
	-\frac{1}{2\gamma}KJ^{2}-\frac{\gamma}{4}[K,J]J & \boldsymbol{0}
	\end{array}\right).
	\]
	As before, we make the ansatz
	\[
	C^{\mu\mu}(0)=\left(\begin{array}{cc}
	C_{1}^{\mu\mu}(0) & C_{2}^{\mu\mu}(0)\\
	\left(C_{2}^{\mu\mu}(0)\right)^T & C_{3}^{\mu\mu}(0)
	\end{array}\right),
	\]
	leading to the equations
	\begin{subequations}
	\begin{eqnarray}
	C_{2}^{\mu\mu}(0)+C_{2}^{\mu\mu}(0)^{T} & = &-\big(\frac{\gamma^{2}}{4}+\frac{1}{4\gamma^{2}}+\frac{1}{4}\big)[J,[K,J]]\label{eq:C''1}\\
	-C_{1}^{\mu\mu}(0)+\gamma C_{2}^{\mu\mu}(0)+C_{3}^{\mu}(0) & =&\frac{1}{\gamma}J^{2}K-\frac{\gamma}{2}J[K,J]\label{eq:C''2}\\
	-C_{1}^{\mu\mu}(0)+\gamma C_{2}^{\mu\mu}(0)^{T}+C_{3}^{\mu}(0) & =&\frac{1}{\gamma}KJ^{2}+\frac{\gamma}{2}[K,J]J\label{eq:C''3}\\
	-C_{2}^{\mu\mu}(0)-C_{2}^{\mu\mu}(0)^{T}+2\gamma C_{3}^{\mu\mu}(0) & =&\boldsymbol{0}.
	\label{eq:C''4}
	\end{eqnarray}
	\end{subequations}
	Again, (\ref{eq:C''2}) and (\ref{eq:C''3}) are equivalent by taking
	the transpose. Plugging (\ref{eq:C''1}) into (\ref{eq:C''4}) and
	combing with (\ref{eq:C''2}) or (\ref{eq:C''3}) gives
	\[
	C_{1}^{\mu\mu}(0)=\big(\frac{\gamma}{4}+\frac{1}{4\gamma^{3}}+\frac{\gamma^{3}}{4}\big)(2JKJ-J^{2}K-KJ^{2})-\frac{1}{\gamma}JKJ.
	\]
	Now 
	\[
	\partial_{\mu}^{2}\Theta\vert_{\mu,\nu=0}=2\Tr(C_{1}^{\mu\mu}(0)K)=-(\gamma+\frac{1}{\gamma^{3}}+\gamma^{3})\big(\Tr(JKJK)-\Tr(J^{2}K^{2})\big)-\frac{2}{\gamma}\Tr(JKJK)
	\]
	gives the first part of (\ref{eq:HessTheta}). We proceed in the same
	way to determine $C_{1}^{\nu\nu}(0)$. Analogously, we get 
	\[
	A^{\nu}(0)C^{\nu}(0)+C^{\nu}(0)A^{\nu}(0)^{T}=\left(\begin{array}{cc}
	\boldsymbol{0} & \frac{1}{\gamma}(KJ^{2}-\frac{1}{2}[K,J]J)\\
	\frac{1}{\gamma}(JKJ-\frac{1}{2}J[K,J]) & \frac{1}{2\gamma^{2}}([K,J]J-J[K,J])
	\end{array}\right).
	\]
	Solving the resulting linear matrix system (similar to (\ref{eq:C''1})-(\ref{eq:C''4}))
	results in
	\[
	C_{1}^{\nu\nu}(0)=\big(\frac{1}{4\gamma^{3}}-\frac{1}{4\gamma}\big)(KJ^{2}+J^{2}K)-\big(\frac{1}{2\gamma^{3}}+\frac{1}{2\gamma}\big)JKJ,
	\]
	leading to 
	\[
	\partial_{\nu}^{2}\Theta\vert_{\mu,\nu=0}=2\Tr(C_{1}^{\nu\nu}(0)K)=\big(\frac{1}{\gamma^{3}}-\frac{1}{2\gamma}\big)\Tr(J^{2}K^{2})\big)-\big(\frac{1}{2\gamma^{3}}+\frac{1}{2\gamma}\big)\Tr(JKJK).
	\]
	To compute the cross term $C_{1}^{\mu\nu}(0)$ we take the mixed derivative
	$\partial_{\mu\nu}^{2}$ of (\ref{eq: lyap equation}) and set $\mu=\nu=0$
	to arrive at 
	\[
	A^{\mu}(0)C^{\nu}(0)+A^{\nu}(0)C^{\mu}(0)+A(0)C^{\mu\nu}(0)+C^{\mu\nu}(0)A(0)^{T}+C^{\mu}(0)A^{\nu}(0)^{T}+C^{\nu}(0)A^{\mu}(0)^{T}=\boldsymbol{0}.
	\]
	Using $\eqref{eq:C^mu}$ and (\ref{eq:C^nu}) we see that 
	\begin{multline*}
	A^{\mu}(0)C^{\nu}(0)+A^{\nu}(0)C^{\mu}(0)+C^{\mu}(0)A^{\nu}(0)^{T}+C^{\nu}(0)A^{\mu}(0)^{T}\\
	=\left(\begin{array}{cc}
	\big(\frac{1}{4\gamma^{2}}-\frac{1}{4}\big)[J,[K,J]] & \frac{1}{\gamma}JKJ-\frac{1}{4\gamma}J[K,J]-\frac{\gamma}{4}[K,J]J\\
	\frac{1}{2\gamma}JKJ+\frac{1}{2\gamma}KJ^{2}+\frac{\gamma}{4}J[K,J]-\frac{1}{4\gamma}[K,J]J & \big(\frac{1}{4\gamma^{2}}+\frac{1}{4}\big)[J,[K,J]]
	\end{array}\right).
	\end{multline*}
	The ensuing linear matrix system yields the solution
	\[
	C_{1}^{\mu\nu}(0)=\big(-\frac{1}{4\gamma^{3}}+\frac{\gamma}{4}-\frac{1}{4\gamma}\big)[J,[K,J]]+\frac{1}{\gamma}JKJ,
	\]
	leading to 
	\begin{equation}
	\partial_{\mu\nu}^{2}\Theta\vert_{\mu,\nu=0}=2\Tr(C_{1}^{\mu\nu}(0)K)=\big(\frac{1}{\gamma^{3}}+\frac{1}{\gamma}-\gamma\big)\Tr(J^{2}K^{2})+\big(-\frac{1}{\gamma^{3}}+\frac{1}{\gamma}+\gamma\big)\Tr(JKJK).
	\end{equation}
	This completes the proof.
	\qed 
    \end{proof}
\begin{proof}
	[Proof of Proposition \ref{thm:linear_full_J}] By (\ref{eq:linear condition})
	and (\ref{eq:Gaussian asymvar}) the function $\Theta$ satisfies
	\[
	\Theta(\mu,\nu)=\bar{l}\cdot A^{-1}\bar{l}.
	\]
	Recall the following formula for blockwise inversion of matrices using the Schur complement:
	\begin{equation}
	\left(\begin{array}{cc}
	U & V\\
	W & X
	\end{array}\right)^{-1}=\left(\begin{array}{cc}
	(U-VX^{-1}W)^{-1} & \ldots\\
	\ldots & \ldots
	\end{array}\right),\label{eq: blockwise inversion}
	\end{equation}
	provided that $X$ and $U-VX^{-1}W$ are invertible. Using this, we obtain 
	\[
	\Theta(\mu,\nu)=l\cdot\big(-\mu J+(\gamma-\nu J)^{-1}\big)l.
	\]
	Taking derivatives, setting $\mu=\nu=0$ and using the fact that $J^{T}=-J$
	leads to the desired result.
	\qed
\end{proof}

\begin{lemma}
	\label{lem:basic_inequalities}
	The following holds: 
	\begin{enumerate}[label=(\alph*)]
		\item \label{it:gaussian_lem1}$\gamma-\frac{4}{\gamma^{3}}-\gamma^{3}-\frac{1}{\gamma}<0$
		for $\gamma\in(0,\infty)$. 
		\item \label{it:gaussian_lem2} Let $J=-J^{T}$ and $K=K^{T}$. Then $\Tr(JKJK)-\Tr(J^{2}K^{2})\ge0$.
		Furthermore, equality holds if and only if $[J,K]=0.$ 
	\end{enumerate}
\end{lemma}
\begin{proof}
	To show \ref{it:gaussian_lem1} we note that $\gamma-\frac{4}{\gamma^{3}}-\gamma^{3}-\frac{1}{\gamma}<\gamma-\frac{4}{\gamma^{3}}-\gamma^{3}=\gamma(1 - \frac{4}{\gamma^4}-\gamma^2)$.  The function $f(\gamma):=1 - \frac{4}{\gamma^4}-\gamma^2$
	has a unique global maximum on $(0,\infty)$ at $\gamma_{min}=8^{1/6}$
	with $f(\gamma_{min})=-2$, so the result follows.
	\\\\
	For \ref{it:gaussian_lem2} we note that $[J,K]^{T}=[J,K],$ and that $[J,K]^{2}$ is symmetric
	and nonnegative definite. We can write 
	\[
	\Tr([J,K]^{2})=\sum_{i}\lambda_{i}^{2},
	\]
	with $\lambda_{i}$ denoting the (real) eigenvalues of $[J,K]$. From
	this it follows that $\Tr([J,K]^{2})\ge0$ with equality if and only
	if $[J,K]=0$. Now expand 
	\[
	\Tr([J,K]^{2})=2\big(\Tr(JKJK)-\Tr(J^{2}K^{2}),
	\]
	which implies the advertised claim. \qed
\end{proof}

%% file: tracefree.tex
Given a symmetric matrix $K\in\mathbb{R}_{sym}^{d\times d}$ with
$\Tr K=0$, we seek to find an orthogonal matrix $U\in O(\mathbb{R}^{d})$
such that $UKU^{T}$ has zeros on the diagonal. This is a crucial
step in Algorithms \ref{alg:optimal J} and \ref{alg:optimal J general}
and has been addressed in various places in the literature (see for
instance \cite{alg_zero_diag} or \cite{Bhatia1997}, Chapter 2,
Section 2). For the convenience of the reader, in the following we
summarize an algorithm very similar to the one in \cite{alg_zero_diag}.
\\\\
Since $K$ is symmetric, there exists an orthogonal matrix $U_{0}\in O(\mathbb{R}^{d})$
such that $U_{0}KU_{0}^{T}=\diag(\lambda_{1},\ldots,\lambda_{d})$.
Now the algorithm proceeds iteratively, orthogonally transforming
this matrix into one with the first diagonal entry vanishing, then
the first two diagonal entries vanishing, etc, until after $d$ steps
we are left with a matrix with zeros on the diagonal. Starting with
$\lambda_{1}$, assume that $\lambda_{1}\neq0$ (otherwise proceed
with $\lambda_{2}$). Since $\Tr(K)=\Tr(U_{0}KU_{0}^{T})=\sum\lambda_{i}=0$,
there exists $\lambda_{j}$, $j\in\{2,\ldots,d\}$ such that $\lambda_{1}\lambda_{j}<0$
(i.e. $\lambda_{1}$ and $\lambda_{j}$ have opposing signs). We now
apply a rotation in the $1j$-plane to transform the first diagonal
entry into zero. More specifically, let 
\[
U_{1}=\begin{blockarray}{ccccccccc}
 ~ & ~ & ~ & ~ & j & ~ & ~ & ~ & ~\\
\begin{block}{(cccccccc)c}
\cos\alpha & 0 & \ldots & 0 & -\sin\alpha & 0 & \ldots & 0 & ~\\
0 & 1 & ~ & ~ & 0 & ~ & ~ & \vdots & ~\\
\vdots & ~ & \ddots & ~ & ~ & ~ & ~ & ~ & ~\\
0 & ~ & ~ & 1 & 0 & ~ & ~ & ~ & ~\\
\sin\alpha & 0 & ~ & 0 & \cos\alpha & 0 & ~ & ~ & j\\
0 & ~ & ~ & ~ & 0 & 1 & ~ & \vdots & ~\\
\vdots & ~ & ~ & ~ & ~ & ~ & \ddots & 0 & ~\\
0 & 0 & \hdots & ~ & ~ & \hdots & 0 & 1 & ~ \\
\end{block}\end{blockarray} \in O(\mathbb{R}^{d})
\]
with $\alpha=\arctan\sqrt{-\frac{\lambda_{1}}{\lambda_{j}}}.$ We
then have $(U_{1}U_{0}KU_{0}^{T}U_{1}^{T})_{11}=0$. Now the same
procedure can be applied to the second diagonal entry $\text{\ensuremath{\lambda}}_{2}$,
leading to the matrix $U_{2}U_{1}U_{0}KU_{0}^{T}U_{1}^{T}U_{2}^{T}$
with 
\[
(U_{2}U_{1}U_{0}KU_{0}^{T}U_{1}^{T}U_{2}^{T})_{11}=(U_{2}U_{1}U_{0}KU_{0}^{T}U_{1}^{T}U_{2}^{T})_{22}=0
\]
Iterating this process, we obtain that $U_{d}\ldots U_{1}U_{0}KU_{0}^{T}U_{1}^{T}\ldots U_{d}^{T}$
has zeros on the diagonal, so $U_{d}\ldots U_{1}U_{0}\in O(\mathbb{R}^{d})$
is the required orthogonal transformation.

%% file: paper.bbl
\newcommand{\etalchar}[1]{$^{#1}$}